\documentclass{amsart}
\usepackage{calc}
\usepackage[nomessages]{fp}
\usepackage{tikz,amsmath,amssymb}
\usepackage{hyphenat}
\usepackage{tkz-euclide} 
\usepackage{cite}
\usetikzlibrary{calc}
\usetikzlibrary{arrows} 
\usetikzlibrary{decorations.markings}
\usepackage[english]{babel}
\usepackage[hidelinks]{hyperref}
\usepackage{amsmath,amsfonts,amssymb}
\usepackage{amsthm}
\usepackage[all]{xy}
\usepackage{ stmaryrd }
\usepackage{ mathrsfs }
\theoremstyle{plain} 

\usepackage{epsfig}
\usepackage[T1]{fontenc}
\usepackage{lmodern}
\usepackage{eurosym}
\usepackage{graphicx}
\usepackage{enumitem}
\usepackage{amsbsy}
\usepackage{rotating}
\usepackage[noabbrev,capitalise]{cleveref}

\definecolor{midnightblue}{rgb}{0.1, 0.1, 0.44}
\definecolor{plum}{rgb}{0.56, 0.27, 0.52}
\definecolor{Plum}{rgb}{0.56, 0.27, 0.52}
\definecolor{patriarch}{rgb}{0.5, 0.0, 0.5}
\definecolor{darkgreen}{rgb}{0.0, 0.2, 0.13}
\definecolor{darkcerulean}{rgb}{0.03, 0.27, 0.49}
\definecolor{jade}{rgb}{0.0, 0.66, 0.42}

\hypersetup{
	colorlinks=true,
	linkcolor=darkcerulean,
	citecolor=jade,
	filecolor=magenta,      
	urlcolor=cyan,
	pdfpagemode=FullScreen,
}

\newcommand{\rep}{\operatorname{rep}}
\newcommand{\Cat}{\operatorname{Cat}}
\newcommand{\Int}{\operatorname{Int}}
\newcommand{\GenJF}{\operatorname{GenJF}}
\newcommand{\GenRep}{\operatorname{GenRep}}
\newcommand{\JF}{\operatorname{JF}}

\newcommand{\add}{\operatorname{add}}
\newcommand{\Ker}{\operatorname{Ker}}
\newcommand{\Coker}{\operatorname{Coker}}
\newcommand{\Hom}{\operatorname{Hom}}
\newcommand{\Id}{\operatorname{Id}}
\newcommand{\IIm}{\operatorname{Im}}
\newcommand{\mult}{\operatorname{mult}}
\newcommand{\diag}{\operatorname{diag}}
\newcommand{\op}{\operatorname{op}}

\newcommand{\refl}{\operatorname{refl}}
\newcommand{\GK}{\operatorname{\mathsf{GK}}}
\newcommand{\NEnd}{\operatorname{\mathsf{N}End}}
\newcommand{\wt}{\operatorname{\mathsf{wt}}}
\newcommand{\End}{\operatorname{End}}

\newcommand{\vdim}{\operatorname{\pmb{\dim}}}

\newcommand{\Adds}{\operatorname{AddS}}
\newcommand{\Ind}{\operatorname{Ind}}
\newcommand{\Supp}{\operatorname{Supp}}
\newcommand{\rev}{\operatorname{rev}}
\newcommand{\C}{\mathbf{C}}
\newcommand{\F}{\mathbf{F}}

\newcommand{\A}{\mathbf A}
\newcommand{\B}{\mathbf B}
\newcommand{\E}{\mathbf E}

\newcommand{\tog}{\operatorname{tog}}
\newcommand{\compl}{\operatorname{compl}}
\newcommand{\eff}{\operatorname{eff}}

\author[B.~Dequêne]{Benjamin Dequêne}
\address[B.~Dequêne]{School of Mathematics, University of Leeds, United Kingdom}
\email{B.D.Dequene@leeds.ac.uk}
\date{\today}

\title[Canonical Jordan recoverability for $A_n$ type quivers]{Canonically Jordan recoverable categories for modules over the path algebra of $A_n$ type quivers}

\usepackage{thmtools}
\declaretheorem[numberwithin=section,name=Theorem,
refname={Theorem,Theorems},
Refname={Theorem,Theorems}]{theorem}
\declaretheorem[numberlike=theorem,name=Lemma,
refname={Lemma,Lemmas},
Refname={Lemma,Lemmas}]{lemma}
\declaretheorem[numberlike=theorem,name=Proposition,
refname={Proposition,Propositions},
Refname={Proposition,Propositions}]{prop}
\declaretheorem[numberlike=theorem,name=Corollary,
refname={Corollary,Corollaries},
Refname={Corollary,Corollaries}]{cor}
\declaretheorem[numberlike=theorem,name=Conjecture,
refname={Conjecture,Conjectures},
Refname={Conjecture,Conjectures}]{conj}

\declaretheorem[style=definition,numberlike=theorem,name=Definition,
refname={Definition,Definitions},
Refname={Definition,Definitions}]{definition}
\declaretheorem[style=definition,numberlike=theorem,name=Algorithm,
refname={Algorithm,Algorithms},
Refname={Algorithm,Algorithms}]{algo}

\declaretheorem[style=definition,numberlike=theorem,name=Example,
refname={Example,Examples},
Refname={Example,Examples}]{ex}

\declaretheorem[style=remark,numberlike=theorem,name=Remark,
refname={Remark,Remarks},
Refname={Remark,Remarks}]{remark}

\usepackage{todonotes}

\newcommand{\new}[1]{\textit{\textbf{\color{patriarch}{#1}}}}
\newcommand{\llrr}[1]{\llbracket #1 \rrbracket}

\begin{document}
	
	\begin{abstract} 
		Let $Q$ be a quiver of $A_n$ type and $\mathbb{K}$ be an algebraically closed field. A nilpotent endomorphism of a quiver representation induces a linear transformation of the vector space at each vertex. Generically, among all nilpotent endomorphisms of a fixed representation $X$, there exists a well-defined Jordan form of each of these linear transformations $\GenJF(X)$, called the generic Jordan form data of $X$. A subcategory of $\rep(Q)$ is Jordan recoverable if we can recover $X$ up to isomorphism from its generic Jordan form data.
		
		There is a procedure that allows one to invert the map from representations to generic Jordan form data. The subcategories for which this procedure applies are called canonically Jordan recoverable. We focus on the subcategories of $\rep(Q)$ that are canonically Jordan recoverable, and we give a combinatorial characterization of them.
	\end{abstract}
	
	\maketitle
	
	\tableofcontents
	
	\section{Introduction}
	\label{s;intro}
	
	\subsection{Jordan recoverability and canonical Jordan recoverability}
	\label{ss:JRandCJRintro}
	
	Let $Q$ be an $A_n$ type quiver. Consider $X$ a finite-dimensional representation of $Q$ over an algebraically closed field $\mathbb{K}$. Denote by $\NEnd(X)$ the set of nilpotent endomorphisms of $X$. Fix $N \in \NEnd(X)$. For each vertex $q \in Q_0$, the morphism $N$ induces a nilpotent endomorphism $N_q$ of $X_q$. We can extract from $N$ a sequence of integer partitions $\lambda^q \vdash \dim(X_q)$, which correspond to the Jordan block sizes of the Jordan form of each $N_q$. Write $\JF(N) = \pmb{\lambda} = (\lambda^q)_{q\in Q_0}$. Thanks to a result from \cite{GPT19}, for any $X \in \rep(Q)$, there is a dense open set in $\NEnd(X)$ on which $\JF$ is constant. We denote $\GenJF(X)$ this constant that we will refer to as the \new{generic Jordan form data} of $X$.
	
	Throughout the article, by subcategory, we mean a full subcategory closed under direct sums and direct summands. Our interest is to characterize the subcategories $\mathscr{C}$ of $\rep(Q)$ such that we can recover up to isomorphism $X \in \mathscr{C}$ from $\GenJF(X)$. Such a subcategory $\mathscr{C}$ is called \new{Jordan recoverable}. 
	
	In general, determining which subcategories of $\rep(Q)$ are Jordan recoverable remains difficult. For some cases, one can reconstruct $X$ from $\GenJF(X)$ thanks to the existence of a generic choice of a representation $Y$ in $\rep(Q)$ such that $Y$ admits a nilpotent endomorphism of Jordan form $\GenJF(X)$, and then we can ask if $Y$ is isomorphic to $X$. 
	
	Concretely, for all $\#Q_0$-tuples of integer partitions $\pmb{\lambda}$, denote $\rep(Q,\pmb{\lambda})$ the variety of representations of $Q$ which admit a nilpotent endomorphism of Jordan form $\pmb{\lambda}$. For a fixed subcategory $\mathscr{C}$ of $\rep(Q)$, we could try to find if for any $X \in \mathscr{C}$ there is a (Zariski) dense open set $\Omega$ in $\rep(Q,\GenJF(X))$, such that any $Y \in \Omega$ is isomorphic to $X$. Such a subcategory $\mathscr{C}$ is said to be \new{canonically Jordan recoverable}.
	
	Note that a subcategory that is Jordan recoverable is not necessarily canonically Jordan recoverable.
	\begin{ex} \label{A2ex}
		Let $Q$ be the following $A_2$ type quiver.
		\begin{center}
			\begin{tikzpicture}[->,line width=0.6mm,>= angle 60,color=black]
				\node (Q) at (-.7,0){$Q=$};
				\node (1) at (0,0){$1$};
				\node (2) at (2,0){$2$};
				\draw (1) -- node[above]{$\alpha$} (2);
			\end{tikzpicture}
		\end{center}
		The only subcategory of $\rep(Q)$ which is not Jordan recoverable is $\rep(Q)$ itself. Indeed, any strict subcategory $\mathscr{C}$ of $\rep(Q)$ is generated by at most two indecomposable representations, and the dimension vectors of these indecomposable representations are linearly independent. It means that we can recover a representation $\mathscr{C}$ from its dimension vector, and a fortiori, from its generic Jordan form. 
		
		However, $\rep(Q)$ is not Jordan recoverable: take for instance $X = S_1 \oplus S_2$ and $Y = P_1$;  they do not admit a nonzero nilpotent endomorphism ($X_i,Y_i$ are $1$-dimensional $\mathbb
		K$-vector-spaces for $i \in \{1,2\}$) and hence $\GenJF(S_1 \oplus S_2) = ((1),(1)) = \GenJF(P_1)$.
		
		Now we give an example of a category that is Jordan recoverable but not canonically Jordan recoverable. Let $\mathscr{C} = \add(S_1, S_2)$. Consider $X = S_1^a \oplus S_2^b$ with $a,b \in \mathbb
		N$. Any pair of nilpotent endomorphisms $(N_1, N_2)$, with $N_i : X_i \longrightarrow X_i$ for $i \in \{1,2\}$, endows $X$ with a nilpotent endomorphism $N=(N_1,N_2)$. A generic nilpotent endomorphism admits a Jordan form given by the tuple $((a),(b))$ of integer partitions. So $\GenJF(X) = ((a),(b))$ and we can check again that $\mathscr{C}$ is Jordan recoverable.  
		
		However, $\mathscr{C}$ is not canonically Jordan recoverable. Fix $X = S_1 \oplus S_2$. Then $\GenJF(X) = ((1),(1))$.  Let $Y \in \rep(Q)$ such that $Y$ admits a nilpotent endomorphism $N$ of the Jordan form $((1),(1))$. Thus $N = 0$. In such case, $Y_1 \cong \mathbb{K} \cong Y_2$ and $Y_\alpha = k \Id$. The endomorphism $N$ does not give any restriction on the value $k \in \mathbb{K}$.
		\begin{center}
			\begin{tikzpicture}[->,line width=0.6mm,>= angle 60,color=black]
				\node (Q) at (-.7,0){$Y\cong$};
				\node (1) at (0,0){$\mathbb{K}$};
				\node (2) at (2,0){$\mathbb{K}$};
				\draw (1) -- node[above]{$k$} (2);
				\node (Q) at (-.7,-2){$Y\cong$};
				\node (a) at (0,-2){$\mathbb{K}$};
				\node (b) at (2,-2){$\mathbb{K}$};
				\draw (a) -- node[below]{$k$} (b);
				
				\draw[line width=0.2mm] (1) -- node[left]{$0$} (a);
				
				\draw[line width=0.2mm] (2) -- node[right]{$0$} (b);
			\end{tikzpicture}
		\end{center}
		Only two choices give different representations $Y$ up to isomorphism: $k=0$ and $k \neq 0$. The first case returns $X$, while the second returns $P_1$. We get a dense open set $\Omega$ in the collection of representations admitting a nilpotent endomorphism of Jordan form $((1),(1))$ in which all the representations are isomorphic to $P_1$. Hence, we did not recover $X$, and $\mathscr{C}$ is not canonically Jordan recoverable as we claimed.
	\end{ex}
	This paper aims to give a combinatorial description of all canonically Jordan recoverable subcategories of $\rep(Q)$ for $Q$ being any $A_n$ type quiver and for $n \in \mathbb{N}^*$.
	
	\subsection{Adjacency-avoiding interval subsets} 
	\label{ss:IntroAdjavoid}
	
	Call \new{intervals} of $\{1, \ldots,n\}$ the sets $\{i, i+1, \ldots,j \}$ with $1 \leqslant i \leqslant j \leqslant n$. Fix an $A_n$ quiver $Q$. The intervals in $\{1, \ldots, n \}$ provide a natural description of $\rep(Q)$: the indecomposable representations are in one-to-one correspondence with the intervals of $\{1, \ldots, n\}$, and morphisms between two indecomposable representations are completely described in terms of specific subintervals of both corresponding intervals. \cref{s:TypeAquiver} gives the precise statement.  Denote by $X_K$ the indecomposable representation of $\rep(Q)$ corresponding to the interval $K$.
	
	For any interval $K$, write $b(K)$ as the lower bound and $e(K)$ as the upper bound of $K$. Two intervals $K$ and $L$ are \new{adjacent} if either $b(K) = e(L)+1$ or $b(L) = e(K)+1$. We have the following result inspired by a previous work \cite{D22} and by \cref{A2ex}. 
	\begin{prop}\label{nec cond CJR1} Let $\mathscr{C}$ be a subcategory of $\rep(Q)$. Write $\mathscr{J}$ for the interval set corresponding to the indecomposable representations that additively generate $\mathscr{C}$. If two intervals exist $K,L \in \mathscr{J}$ such that $K$ and $L$ are adjacent, then $\mathscr{C}$ is not canonically Jordan recoverable.  
	\end{prop}
	Let us first prove this lemma, which will be helpful.
	\begin{lemma}\label{lem:GenJFdisjunion}
		Fix an $A_n$ type quiver $Q$. Let $K_1, \ldots, K_p$ be $p \in \mathbb{N}^*$ disjoint intervals. Write $J = K_1 \cup \ldots \cup K_p$. Then $\GenJF(X_{K_1} \oplus  \ldots \oplus X_{K_p}) = (\lambda^q)_{q \in Q_0}$ with:
		\begin{enumerate}[label = $\bullet$]
			\item $\lambda^q = (1)$ for $q \in J$;
			\item $\lambda^q = (0)$ otherwise.
		\end{enumerate}
	\end{lemma}
	\begin{proof}[Proof] Since, for $q \in Q_0$, $\dim(X_q) \leqslant 1$, we must have $N = 0$. The result follows.
	\end{proof}
	\begin{proof}[Proof of \cref{nec cond CJR1}] Let $K,L \in \mathscr{J}$ be two adjacent intervals. Write $X = X_K \oplus X_L$ and $J = K \cup L$. Note that $J$ is an interval. We get $\GenJF(X) = (\lambda^q)_{q \in Q_0}$ as defined in the previous lemma.
		
		First, note that the only nilpotent endomorphism $N$ such that $\JF(N) = \pmb{\lambda}$ is the zero morphism. Therefore, choosing $Y \in \rep(Q,\pmb{\lambda})$ is equivalent to taking a representation $Y$ such that $Y_q \cong \mathbb{K}$ if $q \in J$, and $Y_q = 0$ otherwise, without any other restrictions. We get that $$\Omega = \{ Z \in \rep(Q,\pmb{\lambda}) \mid Z_{\alpha} \neq 0, \forall \alpha \in Q_1, \{s(\alpha), t(\alpha)\} \subset J\}$$ is a dense open set of $\rep(Q, \pmb{\lambda})$. Following this last statement, and by observing that $Z \cong X_J \ncong X$ for all $Z \in \Omega$, we conclude that $\mathscr{C}$ is not canonically Jordan recoverable.
	\end{proof} 
	This result highlights the necessary condition to avoid the existence of two adjacent intervals among the set of intervals corresponding to indecomposable representations that generate $\mathscr{C}$. We define an \new{adjacency-avoiding} interval set as an interval set with no pair of adjacent intervals.
	
	We aim to prove that the adjacency-avoiding property also gives a sufficient combinatorial criterion to detect canonical Jordan recoverability.
	\begin{theorem} \label{maintheorem}
		Let $Q$ be an $A_n$ type quiver, and $\mathscr{C}$ be a subcategory of $\rep_\mathbb{K}(Q)$. Write $\mathscr{J}$ for the interval set corresponding to the indecomposable representations that additively generate $\mathscr{C}$. Then $\mathscr{C}$ is canonically Jordan recoverable if and only if $\mathscr{J}$ is adjacency-avoiding. 
	\end{theorem}
	This theorem completely characterizes the canonically Jordan recoverable subcategories of $\rep(Q)$ and specializes to give a previous result of Garver, Patrias, and Thomas for $A_n$ type quivers.
	\begin{cor}[\cite{GPT19}] Let $Q$ be an $A_n$ type quiver. For any $m \in Q_0$, the category $\mathscr{C}_{Q,m}$ generated by the indecomposable representations $X_K$ for $K$ intervals containing $m$ is canonically Jordan recoverable. 
	\end{cor}
	\begin{proof}[Proof] Let $K$ and $L$ be two intervals corresponding to two indecomposable representations of $\mathscr{C}_{Q,m}$. By definition, $K \cap L \supseteq \{ m\}$. Therefore, $K$ and $L$ are not adjacent, as two adjacent intervals must have an empty intersection. We conclude the desired result by applying \cref{maintheorem}.
	\end{proof}
	To prove \cref{maintheorem}, we first describe the maximal adjacency-avoiding interval sets (for inclusion). After that, following a revised version of the work of \cite{GPT19}, we give a recursive construction of the subcategories generated by indecomposable representations provided by these interval sets. 
	Then, we prove the main result for the linearly oriented case by showing that operations applied during the construction of those subcategories preserve the canonical Jordan recoverability. We conclude the result in the general case by reducing it to the linearly oriented case.
	
	\subsection{A combinatorial motivation} 
	\label{ss:combmotiv} In \cite[section 6]{GPT19}, Garver, Patrias, and Thomas made some links with the Robinson--Schensted--Knuth (RSK) correspondence. Thanks to Gansner's combinatorics \cite{Ga81Ma,Ga81Hi}, they prove that if $Q$ is the $A_n$ type quiver where only the vertex $m$ is a sink, applying $\GenJF$ on $\mathscr{C}_{Q,m}$ coincides with applying the RSK on an integer matrix recording the multiplicities of the indecomposables in $\mathscr{C}_{Q,m}$.
	
	\cref{maintheorem} can be applied to define an extended RSK correspondence. This extended RSK recovers both the scrambled RSK of \cite{GPT19,Dauv20} and Gansner's version of RSK, which applies to fillings of any partition shape, while being more general than either. The reader can find more details in \cite{Deq24}. A FPSAC extended abstract \cite{DFPSAC24} is also available as a shorter version.
	
	\subsection{Outline of the paper}
	We start, in \cref{s:TypeAquiver}, with some fundamental reminders on type $A$ quiver representations. 
	
	Then, in \cref{s:JRandNJR}, in the context of background work on Jordan recoverability, we highlight a combinatorial computation of the generic Jordan form data via the Greene--Kleitman invariant.
	
	Afterward, in \cref{s:store}, we introduce the notion of storability for tuples of integer partitions, which extracts the general behavior of the generic Jordan form data of any representation in a candidate category. 
	
	Following that, in \cref{s:operationsCJR}, we show the behavior of the generic Jordan form data under elementary operations, such as adding copies of a well-chosen simple module and applying reflection functors.
	
	In \cref{s:adj}, we study the combinatorial behavior of the sets of adjacency-avoiding interval sets, by first describing all of them, and then exhibiting their reaction under the action mimicking the reflection functors in the algebraic world.
	
	Finally, we prove the main result in \cref{s:proof}. First, we use algebraic and combinatorial algorithms, based on the previous elementary operations, to construct categories arising from the maximal adjacency-avoiding interval sets in the linear case. Then, we get to the conclusion by induction using the behavior of those subcategories, as shown combinatorially.
	
	In \cref{ss:tgf}, we discuss some future directions to explore following this work.

	\section{Some generalities about \texorpdfstring{$A_n$}{An} type quiver representations} 
	\label{s:TypeAquiver}
	
	In this section, we give an overview of $A$ type quiver representations. For more details, we refer the reader to the following references: \cite{ASS06,BGP73,S14}.
	
	\subsection{\texorpdfstring{$A_n$}{An} type quivers}
	\label{ss:DefAn}
	
	A \new{quiver} $Q$ is a $4$-tuple $(Q_0,Q_1,s,t)$ where $Q_0$ is \new{the set of vertices}, $Q_1$ is \new{the set of arrows} and $s,t : Q_1 \longrightarrow Q_0$ are respectively \new{source} and \new{target functions}. The \new{opposite quiver} of $Q$, denoted $Q^{\op}$, is the quiver obtained from $Q$ by reversing the direction of all the arrows of $Q$.  We say that $Q$ is a \new{finite} quiver whenever $Q_0$ and $Q_1$ are finite sets. The \new{underlying graph} of a finite quiver $Q$ is a pair $\mathcal{G}(Q) = (\mathcal{G}_0,\mathcal{G}_1)$ where $\mathcal{G}_0 = Q_0$ is the set of vertices of the graph and $\mathcal{G}_1 = \{\{s(\alpha), t(\alpha)\} \mid \alpha \in Q_1\}$ is the (multi)set of edges of the graph. Note that a finite quiver $Q$ can be seen as the graph $\mathcal{G}(Q)$ endowed with an orientation for each edge. 
	
	Let $Q$ be a finite quiver, and $n > 0$ be an integer. Assume that $Q_0 = \{1, \ldots, n\}$. The quiver $Q$ is said to be \new{of $A_n$ type} whenever $\mathcal{G}(Q)$ is of the following shape. \begin{center}
		\begin{tikzpicture}[-,line width=0.5mm,>= angle 60,color=black,scale=0.8]
			\node (1) at (0,0){$1$};
			\node (2) at (2,0){$2$};
			\node (3) at (4,0){$\cdots$};
			\node (4) at (6,0){$n$};
			\draw (1) -- (2) -- (3) -- (4);
		\end{tikzpicture}
	\end{center}  We denote $\overrightarrow{A_n}$ the $A_n$ type quiver where all the arrows of $Q$ are exactly $i \longrightarrow i+1$ for $i \in \{1, \ldots, n-1\}$. An $A_n$ type quiver $Q$ is called \new{linearly oriented} if either $Q = \overrightarrow{A_n}$ or $Q = \overrightarrow{A_n}^{\op}$.  
	
	\subsection{Representations}
	\label{ss:RepAn}
	
	Let $\mathbb{K}$ be an algebraically closed field. This assumption is a restriction that we need to use the results of \cite{GPT19}. They need it because some of their arguments rely on algebraic geometry.
	
	A \new{representation of $Q$ (over $\mathbb{K}$) }is a pair $X = ((X_q)_{q \in Q_0}, (X_\alpha)_{\alpha \in Q_1})$ where:
	\begin{enumerate}[label = $\bullet$]
		\item for each $q \in Q_0$, $X_q$ is a  $\mathbb{K}$-vector space;
		
		\item for each $\alpha \in Q_1$, $X_\alpha : X_{s(\alpha)} \longrightarrow X_{t(\alpha)}$ is a $\mathbb{K}$-linear map.
	\end{enumerate}
	We say that such a representation $X$ is \new{finite dimensional} if $\dim X_q < \infty$ for all $q \in Q_0$. We denote $\vdim(X) = (\dim X_q)_{q \in Q_0}$ the \new{dimension vector} of $X$. From now on, when we talk about representations of a quiver, we mean finite-dimensional representations.
	
	Let $X$ and $Y$ be two representations of $Q$. A \new{morphism} $\phi$ from $X$ to $Y$ is a collection of linear maps $(\phi_q)_{q \in Q_0}$ such that for any $\alpha \in Q_1$, we have $\phi_{t(\alpha)}X_\alpha = Y_\alpha \phi_{s(\alpha)}$. Write $X \cong Y$ whenever $X$ and $Y$ are isomorphic. Denote $\Hom(X,Y)$ the \new{homomorphism space} from $X$ to $Y$ and $\End(X) = \Hom(X,X)$ the \new{endomorphism space} of $X$.
	
	Recall that the representations of $Q$ endowed with the morphisms between them form a category functorially equivalent to the category of (finite-dimensional) left modules over the path algebra associated to $Q$. Denote $\rep(Q)$ \new{the category of finite-dimensional representations of $Q$}. Remember that $\rep(Q)$ depends on the field $\mathbb{K}$, but for simplicity we suppress it from the notation.
	
	A representation $M \neq 0$ is \new{indecomposable} if either $X = 0$ or $Y = 0$, whenever $M \cong X \oplus Y$. We write $\Ind(Q)$ for the indecomposable representations in $\rep(Q)$ up to isomorphism.
	
	Any $M \in \rep(Q)$ admits a unique decomposition into a finite sum of indecomposable representations up to isomorphism. Given $X \in \Ind(Q)$, denote $\mult(X, M)$ the \new{multiplicity of $X$ in $M$} defined as the number of the indecomposable representations isomorphic to $X$ appearing in the decomposition of $M$. 
	
	We now recall a complete description of the indecomposable representations of an $A_n$ type quiver and the morphisms between them. We will state these results in terms of intervals of $\{1, \ldots, n\}$.
	
	Let $n$ be a positive integer. The \new{intervals} of $\{1, \ldots, n\}$ are the sets $\llrr{i,j} := \{i, i+1, \ldots,j \}$ given by all $1 \leqslant i \leqslant j \leqslant n$. If $i=j$, we write $\llrr{i,i} = \llrr{i}$. Denote $\mathcal{I}_n$ the set of intervals in $\{1, \ldots n\}$. For $K = \llrr{i,j} \in \mathcal{I}_n$, write $b(K) = i$ and $e(K) = j$. Call \new{interval set} any subset of $\mathcal{I}_n$.
	\begin{definition}\label{def:intervalbotandtop}  Let $Q$ be an $A_n$ type quiver. Consider $K,L \in \mathcal{I}_n$ such that $K \subseteq L$. We say that:
		\begin{enumerate}[label = $\bullet$]
			\item $K$ is \new{above} $L$ (\new{relative to} $Q$) if the two following assertions are satisfied:\begin{enumerate}[label = $\bullet$]
				\item $b(K) = b(L)$ or we have the arrow $b(K)-1 \longleftarrow b(K)$ in $Q$;
				\item $e(K) = e(L)$ or we have the arrow $e(K) \longrightarrow e(K)+1$ in $Q$.
			\end{enumerate}
			\item $K$ is \new{below} $L$ (\new{relative to} $Q$) if the two following assertions are satisfied:\begin{enumerate}[label = $\bullet$]
				\item $b(K) = b(L)$ or we have the arrow $b(K)-1 \longrightarrow b(K)$ in $Q$;
				\item $e(K) =e(L)$ or we have the arrow $e(K) \longleftarrow e(K)+1$ in $Q$.
			\end{enumerate}
		\end{enumerate}
	\end{definition}
	\begin{ex} Consider the following quiver.
		\begin{center}
			\begin{tikzpicture}[->,line width=0.5mm,>= angle 60,color=black,scale=0.8]
				\node (Q) at (-1,0){$\overrightarrow{A}_3 =$};
				\node (1) at (0,0){$1$};
				\node (2) at (2,0){$2$};
				\node (3) at (4,0){$3$};
				\draw (1) -- (2);
				\draw (2) -- (3);
			\end{tikzpicture}
		\end{center}
		Then $\llrr{2}$ is above $\llrr{2,3}$ and below $\llrr{1,2}$.
	\end{ex}
	Note that any interval is above and below itself, relative to all $A_n$ type quivers.
	Let $Q$ be a quiver of $A_n$ type. To any interval $K \in  \mathcal{I}_n$, we consider $X_K$ the representation of $Q$ defined as it follows:
	\begin{enumerate}[label = $\bullet$]
		\item $(X_K)_q = \mathbb{K}$ if $q \in K$, $(X_K)_q = 0$ otherwise;
		\item $(X_K)_{\alpha} = \Id_\mathbb{K}$ if $\alpha$ is such that $\{s(\alpha), t(\alpha)\} \subseteq K$, $(X_K)_{\alpha} = 0$ otherwise;
	\end{enumerate}
	Note that $X_K$ is an indecomposable representation of $Q$ for all $K \in \mathcal{I}_n$.
	\begin{theorem}[\cite{G72,R84}]\label{thm:indecandmorphtypeA} Let $Q$ be an $A_n$ type quiver. \begin{enumerate}[label = $(\alph*)$]
			\item \label{indectypeA} The isomorphism classes of indecomposable representations of $Q$ are in bijection with $\mathcal{I}_n$; more precisely, they are described by indecomposable representations $X_K$ for $K \in \mathcal{I}_n$;
			\item \label{morphtypeA} The homomorphism space between two indecomposable representations of $Q$ is of dimension at most one; more precisely, $\Hom(X_K, X_L)$ is nonzero if and only if  there exists an interval $J$ such that $J$ is above $K$ and below $L$ relative to $Q$; if such an interval exists, it is unique and 
			$\Hom(X_K,X_L)$ consists of scalar multiples of the morphism $\phi = (\phi_q)_{q \in Q_0}$ such that $\phi_q = \Id_\mathbb{K}$ if $q \in J$ and $\phi_q = 0$ otherwise.
		\end{enumerate}
	\end{theorem}
	In the light of the previous result : \begin{enumerate}[label = $\bullet$]
		\item for all interval sets $\mathscr{J} \subseteq \mathcal{I}_n$ and all quivers $Q$ of $A_n$ type, write $\Cat_Q(\mathscr{J})$ the subcategory of $\rep(Q)$ additively generated by $X_K$ for $K \in \mathscr{J}$;
		\item for all quivers $Q$ of $A_n$ type and for all nonzero subcategories $\mathscr{C}$ of $\rep(Q)$, let $\Int(\mathscr{C})$ be the interval set of $\mathcal{I}_n$ consisting of intervals $K$ such that $X_K \in \mathscr{C}$.
	\end{enumerate}
	Recall that we only consider full subcategories closed under direct sums and direct summands. Such subcategories are additively generated by $X_K$ for $K \in \mathscr{J}$ for some  $\mathscr{J} \subset \mathcal{I}_n$.
	
	Hence, for any $A_n$ type quiver $Q$, for all $\mathscr{J} \subseteq \mathcal{I}_n$ and for all subcategories $\mathscr{C} \subseteq \rep(Q)$, we have $\mathscr{J} = \Int(\Cat_Q(\mathscr{J})) \text{ and } \mathscr{C} = \Cat_Q(\Int(\mathscr{C})).$

	\subsection{Reflection functors}
	\label{ss:defreflections}
	
	In this subsection, we will recall the definition of \emph{reflection functors} for any quiver $Q$. For our purposes in this paper, defining those functors only on objects is sufficient.
	
	Let $Q$ be an arbitrary quiver and $v$ be a vertex of $Q$. Denote $\sigma_v(Q)$ the quiver obtained from $Q$ by reversing the directions of the arrows incident to $v$. If $\alpha \in Q_1$ such that $v \in \{s(\alpha),t(\alpha)\}$, denote $\tilde{\alpha}$ the reversed arrow of $\alpha$ in $\sigma_v(Q)$. 
	
	Now assume that $v$ is a sink of $Q$. Consider $\Xi = \sigma_v(Q)$. The \new{reflection functor} $$\mathcal{R}_v^+ : \rep(Q) \longrightarrow \rep(\Xi)$$ is defined as follows. Let $X = ((X_q)_{q \in Q_0}, (X_\beta)_{\beta \in Q_1}) \in \rep(Q)$. We set $\mathcal{R}_v^+(X) = ((Y_q)_{q \in \Xi_0}, (Y_\beta)_{\beta \in \Xi_1}) \in \rep(\Xi)$ where \begin{enumerate}[label = $\bullet$]
		\item $Y_q = X_q$ for $q \neq v$ and $$\displaystyle Y_v = \Ker \left(\bigoplus_{\alpha \in Q_1,\ t(\alpha) = v} X_\alpha : \bigoplus _{\alpha \in Q_1,\ t(\alpha) = v} X_{s(\alpha)} \longrightarrow X_v \right);$$
		\item $Y_\beta= X_\beta$ if $\beta \in Q_1$ such that $t(\beta) \neq v$, otherwise $Y_{\tilde{\beta}} : Y_v \longrightarrow X_{s(\beta)}$ is the composition of the kernel inclusion of $Y_v$ to $\displaystyle \bigoplus_{\alpha \in Q_1,\ t(\alpha) = v} X_{s(\alpha)}$ with the projection onto the direct summand $X_{s(\beta)}$.
	\end{enumerate}
	If $v$ is a source of $Q$, the \new{reflection functor} $$\mathcal{R}_v^- : \rep(Q) \longrightarrow \rep(\sigma_v(Q))$$ is defined dually.
	\begin{ex} \label{ex:refl}
		Let $Q$ be a quiver and $X\in \rep(Q)$ as below.
		\begin{center}
			\begin{tikzpicture}[->,line width=0.6mm,>= angle 60,color=black,scale=1.3]
				\node (Q) at (-.7,0){$Q=$};
				\node (1) at (0,0){$1$};
				\node (2) at (2,0){$2$};
				\node (3) at (4,0){$3$};
				\node (4) at (6,0){$4$};
				\draw (1) -- node[above]{$\alpha$} (2);
				\draw (3) -- node[above]{$\beta$} (2);
				\draw (4) -- node[above]{$\gamma$}(3);
				\begin{scope}[yshift=-2.5cm]
					\node (Q) at (-0.7,0.5){$X=$};
					\node (1) at (0,0){$\mathbb{K}^3$};
					\node (2) at (2,0){$\mathbb{K}^4$};
					\node (3) at (4,0){$\mathbb{K}^3$};
					\node (4) at (6,0){$\mathbb{K}^2$};
					\draw (1) -- node[above]{\small $\left[\begin{matrix}
							1 & 0 & 0 \\
							0 & 1 & 0\\
							0 & 0 & 0 \\
							0 & 0 & 0
						\end{matrix} \right]$}(2);
					\draw (3) -- node[above]{\small$\left[\begin{matrix}
							1 & 0 & 0 \\
							0 & 0 & 0\\
							0 & 1 & 0 \\
							0 & 0 & 0
						\end{matrix}\right]$}(2);
					\draw (4) -- node[above]{\small$\left[\begin{matrix}
							0 & 0 \\
							0 & 0 \\
							1 & 0 
						\end{matrix} \right]$} (3);
				\end{scope}
			\end{tikzpicture}
		\end{center}
		Apply the reflection functor $\mathcal{R}_2^+$. The arrows $\alpha$ and $\beta$ are the only ones with $2$ as a target. First we get $\Xi = \sigma_2(Q)$ as follows.
		\begin{center}
			\begin{tikzpicture}[->,line width=0.6mm,>= angle 60,color=black,scale=1.3]
				\node (Q) at (-.7,0){$\Xi=$};
				\node (1) at (0,0){$1$};
				\node (2) at (2,0){$2$};
				\node (3) at (4,0){$3$};
				\node (4) at (6,0){$4$};
				\draw (2) -- node[above]{$\tilde{\alpha}$} (1);
				\draw (2) -- node[above]{$\tilde{\beta}$} (3);
				\draw (4) -- node[above]{$\gamma$}(3);
			\end{tikzpicture}
		\end{center}
		We calculate afterward $$\Ker \left(  X_\alpha \oplus X_\beta \right) = \Ker \left( \left[ \begin{matrix}
			1 & 0 & 0 & 1 & 0 & 0\\
			0 & 1 & 0 & 0 &0 &0 \\
			0 & 0 & 0 & 0 & 1 & 0   \\
			0 & 0 & 0 & 0 &0 &0
		\end{matrix} \right] \right) = \left\langle \left( \begin{matrix}
			1 \\
			0 \\
			0 \\
			-1 \\
			0 \\
			0
		\end{matrix} \right), \left( \begin{matrix}
			0 \\
			0 \\
			1 \\
			0 \\
			0 \\
			0
		\end{matrix} \right), \left( \begin{matrix}
			0 \\
			0 \\
			0 \\
			0 \\
			0 \\
			1
		\end{matrix} \right) \right\rangle \cong \mathbb{K}^3.$$
		Thus we get $\mathcal{R}_2^+(X)$ by replacing the vector space at $2$ by $\mathbb{K}^3$ and defining the morphisms for $\tilde{\alpha}$ and $\tilde{\beta}$ by the composition of the kernel inclusion $\mathbb{K}^3 \longrightarrow \mathbb{K}^6$ and the respective projection to $X_1$ and $X_3$. It gives the following result.
		\begin{center}
			\begin{tikzpicture}[->,line width=0.6mm,>= angle 60,color=black, scale=1.3]
				
				\node (Q) at (-0.7,0.5){$\mathcal{R}_2^+(X)=$};
				\node (1) at (0,0){$\mathbb{K}^3$};
				\node (2) at (2,0){$\mathbb{K}^3$};
				\node (3) at (4,0){$\mathbb{K}^3$};
				\node (4) at (6,0){$\mathbb{K}^2$};
				\draw (2) -- node[above]{\small $\left[\begin{matrix}
						1 & 0 & 0 \\
						0 & 0 & 0\\
						0 & 1 & 0 \\
					\end{matrix} \right]$}(1);
				\draw (2) -- node[above]{\small $\left[\begin{matrix}
						-1 & 0 & 0 \\
						0 & 0 & 0\\
						0 & 0 & 1 \\
					\end{matrix}\right]$}(3);
				\draw (4) -- node[above]{\small $\left[\begin{matrix}
						0 & 0 \\
						0 & 0 \\
						1 & 0 
					\end{matrix} \right]$} (3);
			\end{tikzpicture}
		\end{center}
	\end{ex}
	The reflection functors are additive, meaning we can understand their actions on objects by knowing their actions on indecomposable objects.
	
	By the following proposition, we recall the action of the reflection functors on $\Ind(Q)$, for $Q$ an $A_n$ type quiver.
	\begin{prop}\label{prop:reflonAntypequivers} Let $Q$ be an $A_n$ type quiver, $v \in Q_0$ and $\llrr{v} \neq K \in \mathcal{I}_n$. Write $\Xi = \sigma_v(Q)$. If $v$ is a sink of $Q$, then $\mathcal{R}_v^+(X_K) \cong X_{K'} \in \rep(\Xi)$ where $$K' = \begin{cases}
			K \cup \{v\} & \text{if either } e(K) = v-1 \text{ or } b(K) = v+1; \\
			K \setminus \{v\} & \text{if either } e(K)=v \text{ or } b(K) = v;\\
			K & \text{otherwise.}
		\end{cases}.$$ If $v$ is a source of $Q$, then $\mathcal{R}_v^-(X_K) = X_{K'}$ where $K'$ is defined as above.
	\end{prop}
	Note that, if $v$ is a sink of $Q$, $\mathcal{R}_v^+(X_{\llrr{v}}) = 0$, and if $v$ is a source of $Q$, $\mathcal{R}_v^-(X_{\llrr{v}}) = 0$. We also recall the following result, which will be helpful later.
	\begin{theorem} \label{thm:eqofcatrefl}
		Let $Q$ be a quiver, and $v$ be one of its sinks. Write $\Xi = \sigma_v(Q)$. The reflection functor $\mathcal{R}_v^+ : \rep(Q) \longrightarrow \rep(\Xi)$ induces a category equivalence between the full subcategory of $\rep(Q)$ additively generated by the indecomposable representations of $Q$ except the simple projective representation at $v$ and the full subcategory of $\rep(\Xi)$ additively generated by indecomposable representations of $\Xi$ except the simple injective representation at $v$. The quasi-inverse is induced by the reflection functor $\mathcal{R}_v^-:\rep(\Xi) \longrightarrow \rep(Q)$.
	\end{theorem}
	See \cite[Theorem VII.5.3]{ASS06} for more details.
	
	\section{Jordan recoverability and canonical Jordan recoverability}
	\label{s:JRandNJR}
	
	\subsection{The Greene--Kleitman invariant}
	\label{ss:GKinv}
	In this subsection, for a given $A_n$ type quiver $Q$, we introduce a combinatorial invariant for any representation of $Q$, whose representation-theoretic meaning will be shown in the next subsection. Before introducing this invariant, we need to recall some definitions and give some notations.
	
	Recall that, given a quiver $Q$, the \new{Auslander--Reiten quiver of $\rep(Q)$} is a quiver whose vertex set is the set of isomorphism classes of indecomposable representations of $Q$ and whose arrow set is the set of irreducible morphisms between the indecomposable representations.
	
	Given a positive integer $m$, an \new{integer partition of $m$} is a finite weakly decreasing sequence  of positive integers $\lambda = (\lambda_1, \lambda_2, \ldots, \lambda_k)$ such that $\lambda_1 + \cdots + \lambda_k = m$. The \emph{length} $\ell(\lambda)$ of such an integer partition $\lambda$ is $k$. Its \emph{size} is $|\lambda| = \lambda_1 + \cdots + \lambda_k = m$. We write $\lambda \vdash m$ when $\lambda$ is a partition of $m$.
	
	Let $Q$ be an $A_n$ type quiver. For all $q \in Q_0$, define the subcategory \new{$\mathscr{C}_{Q,q}$} of $\rep(Q)$ by $\mathscr{C}_{Q,q} = \Cat_Q(\{K \in \mathcal{I}_n \mid q \in K \}).$ Fix $X \in \rep(Q)$. We decompose $X$ as below,  with $m_K = \mult(X_K,X) \in \mathbb{N}$. $$X \cong \bigoplus_{K \in \mathcal{I}_n} X_K^{m_K}$$ Consider the full subquiver of the Auslander--Reiten quiver whose vertices are isomorphism classes of indecomposable representations of $Q$ in $\mathscr{C}_{Q,q}$. For $\ell \geqslant 1$, we consider $\Pi_{q}^\ell$ the set of all $\ell$-tuples of maximal paths in this subquiver. Note that these paths start at the vertex corresponding to the projective representation $P_q$ and end at the injective representation $I_q$. For all $\ell$-tuples of paths in the Auslander--Reiten quiver $\gamma = (\gamma_1, \ldots, \gamma_\ell)$, we write $\Supp(\gamma)$ for the set of vertices passed through by some $\gamma_i$. For all $\gamma \in \Pi_{q}^\ell$, we consider a \new{weight} depending on $X$ defined as  follows: $$\wt_X(\gamma) = \sum_{X_K \in \Supp(\gamma)} m_K.$$
	\begin{definition}\label{combwayGenJF} The \new{Greene--Kleitman invariant of $X$}, denoted $\GK(X)$, is the $n$-tuple of partitions  $\pmb{\lambda} = (\lambda^q)_{q \in Q_0}$ with $\lambda^q$ such that:
		\begin{enumerate}[label = $\bullet$]
			\item $\displaystyle \lambda^q_1 = \max_{\gamma \in \Pi_{q}^1} \wt_X (\gamma)$;
			\item $\displaystyle \forall i \geqslant 2,\ \lambda_i^q = \max_{\gamma \in \Pi_{q}^i} \wt_X(\gamma) - \max_{\gamma \in \Pi_{q}^{i-1}} \wt_X(\gamma)
			$.
		\end{enumerate}
	\end{definition}
	\begin{remark}
		This definition is a restatement of \cite{GK76}, already made in \cite{GPT19}.
	\end{remark}
	\begin{ex} \label{ex:GK} Consider $Q = \overrightarrow{A_5}$. Let $X \in \rep(Q)$. We can picture $X$, up to isomorphism, as a \new{filling} of the Auslander--Reiten quiver: meaning a function $\phi_X$ which associates each $X_K$ to $m_K = \mult(X_K,X)$. We pictured the Auslander--Reiten quiver of $\overrightarrow{A_5}$, and an example of a representation $X$ in \cref{fig:exAR}.
		\begin{sidewaysfigure*}
			\centering  
			{\color{white}{\rule{0.75\textheight}{0.5\textheight}} }
			\raisebox{-0.5\height}{\begin{tikzpicture}[->,line width=0.6mm,>= angle 60,color=black,scale=0.8]
					\node (Q) at (-1,0){$\overrightarrow{A_5}=$};
					\node (1) at (0,0){$1$};
					\node (2) at (2,0){$2$};
					\node (3) at (4,0){$3$};
					\node (4) at (6,0){$4$};
					\node (5) at (8,0){$5$};
					\draw (1) -- (2);
					\draw (2) -- (3);
					\draw (3) -- (4);
					\draw (4) -- (5);
					
			\end{tikzpicture}} \scalebox{0.75}{\raisebox{-0.5\height}{\begin{tikzpicture}[->,line width=0.2mm,>= angle 60,scale=1.3]
						
						\node[black] (1) at (0,0){$\llrr{5}$};
						\node[black] (2) at (1,1){$\llrr{4,5}$};
						\node[black] (3) at (2,2){$\llrr{3,5}$};
						\node[black] (4) at (3,3){$\llrr{2,5}$};
						\node[black] (5) at (4,4){$\llrr{1,5}$};
						\node[black] (6) at (2,0){$\llrr{4}$};
						\node[black] (7) at (3,1){$\llrr{3,4}$};
						\node[black] (8) at (4,2){$\llrr{2,4}$};
						\node[black] (9) at (5,3){$\llrr{1,4}$};
						\node[black] (10) at (4,0){$\llrr{3}$};
						\node[black] (11) at (5,1){$\llrr{2,3}$};
						\node[black] (12) at (6,2){$\llrr{1,3}$};
						\node[black] (13) at (6,0){$\llrr{2}$};
						\node[black] (14) at (7,1){$\llrr{1,2}$};
						\node[black] (15) at (8,0){$\llrr{1}$};  
						
						\draw (1) -- (2);
						\draw (2) -- (3);
						\draw (2) -- (6);
						\draw (3) -- (4);
						\draw (3) -- (7);
						\draw (4) -- (5);
						\draw (4) -- (8);
						\draw (5) -- (9);
						\draw (6) -- (7);
						\draw (7) -- (10);
						\draw (7) -- (8);
						\draw (8) -- (11);
						\draw (8) -- (9);
						\draw (9) -- (12);
						\draw (10) -- (11);
						\draw (11) -- (12);
						\draw (11) -- (13);
						\draw (12) -- (14);
						\draw (13) -- (14);
						\draw (14) -- (15);

						\draw[dashed] (6) -- (1);
						\draw[dashed] (10) -- (6);
						\draw[dashed] (13) -- (10);
						\draw[dashed] (15) -- (13);
						
						\draw[dashed] (7) -- (2);
						\draw[dashed] (11) -- (7);
						\draw[dashed] (14) -- (11);
						
						\draw[dashed] (8) -- (3);
						\draw[dashed] (12) -- (8);
						
						\draw[dashed] (9) -- (4);
						
						\begin{scope}[xshift=10cm]
							\node[circle,draw,black,line width=0.3mm] (1) at (0,0){$\mathbf 0$};
							\node[circle,draw,black,line width=0.3mm] (2) at (1,1){$\bf 1$};
							\node[circle,draw,black,line width=0.3mm] (3) at (2,2){$\bf 1$};
							\node[circle,draw,black,line width=0.3mm] (4) at (3,3){$\bf 2$};
							\node[circle,draw,black,line width=0.3mm](5) at (4,4){$\bf 2$};
							\node[circle,draw,black,line width=0.3mm](6) at (2,0){$\bf 1$};
							\node[circle,draw,black,line width=0.3mm] (7) at (3,1){$\bf 3$};
							\node[circle,draw,black,line width=0.3mm] (8) at (4,2){$\bf 2$};
							\node[circle,draw,black,line width=0.3mm] (9) at (5,3){$\bf 4$};
							\node[circle,draw,black,line width=0.3mm](10) at (4,0){$\bf 1$};
							\node[circle,draw,black,line width=0.3mm] (11) at (5,1){$\bf 2$};
							\node[circle,draw,black,line width=0.3mm] (12) at (6,2){$\bf 1$};
							\node[circle,draw,black,line width=0.3mm] (13) at (6,0){$\bf 2$};
							\node[circle,draw,black,line width=0.3mm] (14) at (7,1){$\bf 2$};
							\node[circle,draw,black,line width=0.3mm] (15) at (8,0){$\bf 1$};  
							
							\draw (1) -- (2);
							\draw (2) -- (3);
							\draw (2) -- (6);
							\draw (3) -- (4);
							\draw (3) -- (7);
							\draw (4) -- (5);
							\draw (4) -- (8);
							\draw (5) -- (9);
							\draw (6) -- (7);
							\draw (7) -- (10);
							\draw (7) -- (8);
							\draw (8) -- (11);
							\draw (8) -- (9);
							\draw (9) -- (12);
							\draw (10) -- (11);
							\draw (11) -- (12);
							\draw (11) -- (13);
							\draw (12) -- (14);
							\draw (13) -- (14);
							\draw (14) -- (15);

							\draw[dashed] (6) -- (1);
							\draw[dashed] (10) -- (6);
							\draw[dashed] (13) -- (10);
							\draw[dashed] (15) -- (13);
							
							\draw[dashed] (7) -- (2);
							\draw[dashed] (11) -- (7);
							\draw[dashed] (14) -- (11);
							
							\draw[dashed] (8) -- (3);
							\draw[dashed] (12) -- (8);
							
							\draw[dashed] (9) -- (4);
						\end{scope}
						
			\end{tikzpicture}}}
			\caption{\label{fig:exAR} The quiver $\overrightarrow{A}_5$, its Auslander--Reiten quiver and a choice of a representation $X$ seen as a \emph{filling} of the AR-quiver. We label $\llrr{i,j}$ the vertex corresponding to the isomorphism classes of the indecomposable representation $X_{\llbracket i,j \rrbracket}$ in $\rep(\overrightarrow{A_5})$. The dashed arrows correspond to the action of the Auslander--Reiten translation $\tau$.}       
			\scalebox{0.55}{ \begin{tikzpicture}[->,line width=0.2mm,>= angle 60,scale=1.2]
					\tkzDefPoint(4,1.9){O}	\tkzDefPoint(4,4.6){a}
					\tkzDefPointsBy[rotation=center O angle 360/4](a,b,c){b,c,d}
					\tkzDrawPolygon[line width = 2mm, color = black, fill = black!10](a,b,c,d);
					
					\tkzDefPoint(2,2){e}
					\tkzDefPoint(3,1){f}
					\tkzDefPoint(5,3){g}
					\tkzDefPoint(6,2){h}
					\draw[-,line width=6mm, black!60](e) edge (f);
					\draw[-,line width=6mm, black!60](f) edge (g);
					\draw[-,line width=6mm, black!60](g) edge (h);
					
					\node[circle,draw,black,line width=0.3mm] (1) at (0,0){$\mathbf 0$};
					\node[circle,draw,black,line width=0.3mm] (2) at (1,1){$\bf 1$};
					\node[circle,draw,black,line width=0.3mm] (3) at (2,2){$\bf 1$};
					\node[circle,draw,black,line width=0.3mm] (4) at (3,3){$\bf 2$};
					\node[circle,draw,black,line width=0.3mm](5) at (4,4){$\bf 2$};
					\node[circle,draw,black,line width=0.3mm](6) at (2,0){$\bf 1$};
					\node[circle,draw,black,line width=0.3mm] (7) at (3,1){$\bf 3$};
					\node[circle,draw,black,line width=0.3mm] (8) at (4,2){$\bf 2$};
					\node[circle,draw,black,line width=0.3mm] (9) at (5,3){$\bf 4$};
					\node[circle,draw,black,line width=0.3mm](10) at (4,0){$\bf 1$};
					\node[circle,draw,black,line width=0.3mm] (11) at (5,1){$\bf 2$};
					\node[circle,draw,black,line width=0.3mm] (12) at (6,2){$\bf 1$};
					\node[circle,draw,black,line width=0.3mm] (13) at (6,0){$\bf 2$};
					\node[circle,draw,black,line width=0.3mm] (14) at (7,1){$\bf 2$};
					\node[circle,draw,black,line width=0.3mm] (15) at (8,0){$\bf 1$};  
					
					\draw (1) -- (2);
					\draw (2) -- (3);
					\draw (2) -- (6);
					\draw (3) -- (4);
					\draw (3) -- (7);
					\draw (4) -- (5);
					\draw (4) -- (8);
					\draw (5) -- (9);
					\draw (6) -- (7);
					\draw (7) -- (10);
					\draw (7) -- (8);
					\draw (8) -- (11);
					\draw (8) -- (9);
					\draw (9) -- (12);
					\draw (10) -- (11);
					\draw (11) -- (12);
					\draw (11) -- (13);
					\draw (12) -- (14);
					\draw (13) -- (14);
					\draw (14) -- (15);

					\draw[dashed] (6) -- (1);
					\draw[dashed] (10) -- (6);
					\draw[dashed] (13) -- (10);
					\draw[dashed] (15) -- (13);
					
					\draw[dashed] (7) -- (2);
					\draw[dashed] (11) -- (7);
					\draw[dashed] (14) -- (11);
					
					\draw[dashed] (8) -- (3);
					\draw[dashed] (12) -- (8);
					
					\draw[dashed] (9) -- (4);
					\node[black] at (4,-1.5){\Huge $\max_{\Pi_{3}^1} \wt_X(\gamma) = 11$};
					
					\begin{scope}[xshift = 9.5cm]
						\tkzDefPoint(4,1.9){O}	\tkzDefPoint(4,4.6){a}
						\tkzDefPointsBy[rotation=center O angle 360/4](a,b,c){b,c,d}
						\tkzDrawPolygon[line width = 2mm, color = black, fill = black!10](a,b,c,d);
						
						\tkzDefPoint(2,2){e}
						\tkzDefPoint(3,1){f}
						\tkzDefPoint(4,2){g}
						\tkzDefPoint(5,1){h}
						\tkzDefPoint(6,2){i}
						\tkzDefPoint(4,4){j}
						\draw[-,line width=6mm, black!60](e) edge (f);
						\draw[-,line width=6mm, black!60](f) edge (g);
						\draw[-,line width=6mm, black!60](g) edge (h);
						\draw[-,line width=6mm, black!60](h) edge (i);
						\draw[-,line width=6mm,, black!60](e) edge (j);
						\draw[-,line width=6mm, black!60](j) edge (i);
						
						\node[circle,draw,black,line width=0.3mm] (1) at (0,0){$\mathbf 0$};
						\node[circle,draw,black,line width=0.3mm] (2) at (1,1){$\bf 1$};
						\node[circle,draw,black,line width=0.3mm] (3) at (2,2){$\bf 1$};
						\node[circle,draw,black,line width=0.3mm] (4) at (3,3){$\bf 2$};
						\node[circle,draw,black,line width=0.3mm](5) at (4,4){$\bf 2$};
						\node[circle,draw,black,line width=0.3mm](6) at (2,0){$\bf 1$};
						\node[circle,draw,black,line width=0.3mm] (7) at (3,1){$\bf 3$};
						\node[circle,draw,black,line width=0.3mm] (8) at (4,2){$\bf 2$};
						\node[circle,draw,black,line width=0.3mm] (9) at (5,3){$\bf 4$};
						\node[circle,draw,black,line width=0.3mm](10) at (4,0){$\bf 1$};
						\node[circle,draw,black,line width=0.3mm] (11) at (5,1){$\bf 2$};
						\node[circle,draw,black,line width=0.3mm] (12) at (6,2){$\bf 1$};
						\node[circle,draw,black,line width=0.3mm] (13) at (6,0){$\bf 2$};
						\node[circle,draw,black,line width=0.3mm] (14) at (7,1){$\bf 2$};
						\node[circle,draw,black,line width=0.3mm] (15) at (8,0){$\bf 1$};  
						
						\draw (1) -- (2);
						\draw (2) -- (3);
						\draw (2) -- (6);
						\draw (3) -- (4);
						\draw (3) -- (7);
						\draw (4) -- (5);
						\draw (4) -- (8);
						\draw (5) -- (9);
						\draw (6) -- (7);
						\draw (7) -- (10);
						\draw (7) -- (8);
						\draw (8) -- (11);
						\draw (8) -- (9);
						\draw (9) -- (12);
						\draw (10) -- (11);
						\draw (11) -- (12);
						\draw (11) -- (13);
						\draw (12) -- (14);
						\draw (13) -- (14);
						\draw (14) -- (15);

						\draw[dashed] (6) -- (1);
						\draw[dashed] (10) -- (6);
						\draw[dashed] (13) -- (10);
						\draw[dashed] (15) -- (13);
						
						\draw[dashed] (7) -- (2);
						\draw[dashed] (11) -- (7);
						\draw[dashed] (14) -- (11);
						
						\draw[dashed] (8) -- (3);
						\draw[dashed] (12) -- (8);
						
						\draw[dashed] (9) -- (4);
						\node[black] at (4,-1.5){\Huge $\max_{\Pi_{3}^2} \wt_X(\gamma) = 17$};
					\end{scope}
					
					\begin{scope}[xshift=19cm]
						\tkzDefPoint(4,1.9){O}	\tkzDefPoint(4,4.6){a}
						\tkzDefPointsBy[rotation=center O angle 360/4](a,b,c){b,c,d}
						\tkzDrawPolygon[line width = 2mm, color = black, fill = black!10](a,b,c,d);
						
						\tkzDefPoint(2,2){e}
						\tkzDefPoint(3,1){f}
						\tkzDefPoint(4,2){g}
						\tkzDefPoint(5,1){h}
						\tkzDefPoint(6,2){i}
						\tkzDefPoint(4,4){j}
						\tkzDefPoint(4,0){k}
						\draw[-,line width=6mm, black!60](e) edge (f);
						\draw[-,line width=6mm, black!60](f) edge (g);
						\draw[-,line width=6mm, black!60](g) edge (h);
						\draw[-,line width=6mm, black!60](h) edge (i);
						\draw[-,line width=6mm, black!60](e) edge (j);
						\draw[-,line width=6mm, black!60](j) edge (i);
						\draw[-,line width=6mm, black!60](e) edge (k);
						\draw[-,line width=6mm, black!60](k) edge (i);
						
						\node[circle,draw,black,line width=0.3mm] (1) at (0,0){$\mathbf 0$};
						\node[circle,draw,black,line width=0.3mm] (2) at (1,1){$\bf 1$};
						\node[circle,draw,black,line width=0.3mm] (3) at (2,2){$\bf 1$};
						\node[circle,draw,black,line width=0.3mm] (4) at (3,3){$\bf 2$};
						\node[circle,draw,black,line width=0.3mm](5) at (4,4){$\bf 2$};
						\node[circle,draw,black,line width=0.3mm](6) at (2,0){$\bf 1$};
						\node[circle,draw,black,line width=0.3mm] (7) at (3,1){$\bf 3$};
						\node[circle,draw,black,line width=0.3mm] (8) at (4,2){$\bf 2$};
						\node[circle,draw,black,line width=0.3mm] (9) at (5,3){$\bf 4$};
						\node[circle,draw,black,line width=0.3mm](10) at (4,0){$\bf 1$};
						\node[circle,draw,black,line width=0.3mm] (11) at (5,1){$\bf 2$};
						\node[circle,draw,black,line width=0.3mm] (12) at (6,2){$\bf 1$};
						\node[circle,draw,black,line width=0.3mm] (13) at (6,0){$\bf 2$};
						\node[circle,draw,black,line width=0.3mm] (14) at (7,1){$\bf 2$};
						\node[circle,draw,black,line width=0.3mm] (15) at (8,0){$\bf 1$};  
						
						\draw (1) -- (2);
						\draw (2) -- (3);
						\draw (2) -- (6);
						\draw (3) -- (4);
						\draw (3) -- (7);
						\draw (4) -- (5);
						\draw (4) -- (8);
						\draw (5) -- (9);
						\draw (6) -- (7);
						\draw (7) -- (10);
						\draw (7) -- (8);
						\draw (8) -- (11);
						\draw (8) -- (9);
						\draw (9) -- (12);
						\draw (10) -- (11);
						\draw (11) -- (12);
						\draw (11) -- (13);
						\draw (12) -- (14);
						\draw (13) -- (14);
						\draw (14) -- (15);

						\draw[dashed] (6) -- (1);
						\draw[dashed] (10) -- (6);
						\draw[dashed] (13) -- (10);
						\draw[dashed] (15) -- (13);
						
						\draw[dashed] (7) -- (2);
						\draw[dashed] (11) -- (7);
						\draw[dashed] (14) -- (11);
						
						\draw[dashed] (8) -- (3);
						\draw[dashed] (12) -- (8);
						
						\draw[dashed] (9) -- (4);
						\node[black] at (4,-1.5){\Huge $\max_{\Pi_{3}^3} \wt_X(\gamma) = 18$};
					\end{scope}        
					
			\end{tikzpicture}}

			\caption{\label{fig:explitGK} Explicit way to calculate $\displaystyle \max_{\Pi_{3}^\ell} \wt_X(\gamma)$ for the representation $X$ defined in \cref{fig:exAR}.}
			
		\end{sidewaysfigure*}
		In \cref{fig:explitGK}, we represent how we calculate the integer partition $\lambda^3$. One can calculate $\lambda^q$ following the same process for $q \in Q_0$. By doing all the calculations, we get $\displaystyle \GK(X) = ((10), (11,6),(\mathbf{11}, \mathbf{17} - \mathbf {11} = 6,  \mathbf {18} - \mathbf{17} = 1),(11,5),(6)).$ \qedhere
	\end{ex}
	\begin{remark} \label{rem:GKnotcomplete}  We can note that if $X \cong Y$ then $\GK(X) = \GK(Y)$ by definition: the filling of the Auslander--Reiten quivers for $X$ and for $Y$ are the same. This property explains why $\GK$ is an \textit{invariant of $\rep(Q)$}. 
	\end{remark}
	Following this remark, one can be interested in determining for which subcategories $\mathscr{C}$ of $\rep(Q)$ the Greene--Kleitman invariant is complete. 
	
	In the next subsection, we will see that answering this question is equivalent to characterizing all the Jordan recoverable subcategories of $\rep(Q)$. In this paper, we aim to partially answer this question by characterizing all the canonically Jordan recoverable subcategories of $\rep(Q)$. 
	
	\subsection{Jordan recoverability}
	\label{ss:JR}
	
	Let $Q$ be an $A_n$ type quiver. Consider $X \in \rep(Q)$. A \new{nilpotent endomorphism} $N : X \longrightarrow X$ is an endomorphism such that $N^k = 0$ for some integer $k > 0$. One can think of a nilpotent endomorphism $N$ as a collection of nilpotent transformations $(N_q)_{q \in Q_0}$ satisfying an additional compatibility relation. Denote $\NEnd(X)$ the set of nilpotent endomorphisms of $X$. 
	
	Let $\vdim(X) = \pmb{d} = (d_q)_{q \in Q_0}$. For any $N \in \NEnd(X)$, we can consider the Jordan form of $N_q$ at each vertex $q$. It induces a sequence of partitions $\lambda^q \vdash d_q$. We refer to $(\lambda^q)_{q\in Q_0}$ as the \new{Jordan form data of $N$}. Denote it by $\JF(N)$. 
	
	The \new{dominance order} on partitions of an integer $n$ is defined as it follows: for any $\lambda$ and $\mu$ partitions of $n$, $\lambda \unlhd \mu$ if $\lambda_1 + \ldots + \lambda_k \leqslant \mu_1 + \ldots + \mu_k$ for each $k \geqslant 1$ where we add zero parts to $\lambda$ and $\mu$ if necessary. 
	
	We extend this order to any $n$-tuple of partitions. Introduce first a notation: for $\pmb{d} = (d_1, \ldots, d_n) \in \mathbb{N}^n$ and $\pmb{\pi} = (\pi^1, \ldots, \pi^n)$ a $n$-tuple of integer partitions, we write $\pmb{\pi} \vdash \pmb{d}$ if $\pi^i \vdash d_i$ for $i \in \{1, \ldots,n\}$. Now fix $\pmb{d} \in \mathbb{N}^n$. For $\pmb{\lambda} = (\lambda^1, \ldots, \lambda^n)$ and $\pmb{\mu} = (\mu^1, \ldots, \mu^n)$ such that $\pmb{\lambda} \vdash \pmb{d}$ and $\pmb{\mu} \vdash \pmb{d}$, we say that $\pmb{\lambda} \pmb{\unlhd} \pmb{\mu}$ if for all $i \in \{1,\ldots,n\}$ $\lambda^i \unlhd \mu^i$.
	
	Before stating a precise result on the generic Jordan form data for any representation of $Q$, we recall a key result from Gansner.
	
	Let $\Gamma$ be a finite acyclic quiver. Label the vertices of $\Gamma$ from $1$ to $\#\Gamma_0$ such that, for all $i,j \in \{1, \ldots, \#\Gamma_0\}$, if there is an arrow $i \longrightarrow j$ in $\Gamma$, then $i < j$. 
	
	For $\ell \geqslant 1$, we write $\Pi_\Gamma^\ell$ the set of all $\ell$-tuples of maximal paths in $\Gamma$. For $\gamma = (\gamma_1, \ldots, \gamma) \in \Pi_\gamma^\ell$, write $\Supp(\gamma)$ the set of vertices passed through by some $\gamma_i$. We define $\Delta(\Gamma) = (\Delta_\ell(\Gamma))_{\ell \geqslant 1}$ for all $\ell \geqslant 1$ by:
	\begin{enumerate}[label = $\bullet$]
		\item $\Delta_1(\Gamma) = \max_{\gamma \in \Pi_\Gamma^1} \#\Supp(\gamma)$
		\item for $\ell \geqslant 2$, $\Delta_\ell(\Gamma) = \max_{\gamma \in \Pi_\Gamma^\ell} \#\Supp(\gamma) - \max_{\gamma \in \Pi_\Gamma^{\ell-1}} \#\Supp(\gamma)$.
	\end{enumerate}
	Note that $\Delta_\ell(\Gamma) \geqslant 0$ for all $\ell \geqslant 1$, and there exists $\ell_0 \geqslant 1$ such that $\Delta_{\ell_0}(\Gamma) = 0$.
	
	We define a \emph{generic matrix of $\Gamma$} as a $\#\Gamma_0 \times \#\Gamma_0$ matrix $M = (m_{i,j})$ where $m_{i,j} = 0$ whenever there is no arrow $i \longrightarrow j$ in $\Gamma$, and the rest of its entries are complex numbers algebraically independent over $\mathbb{Q}$. 
	
	We now state the following result proven by Gansner \cite[Theorem 2.1]{Ga81Ac} and independently by Saks in his thesis \cite[Theorem 6.3]{Saks80}.
	\begin{theorem} \label{thm:Gansner}
		Let $\Gamma$ be a finite acyclic quiver.  Any generic matrix $M$ of $\Gamma$ is nilpotent. Moreover, we have $\JF(M) = \Delta(\Gamma)$.
	\end{theorem}
	One can notice that the construction of $\Delta(\Gamma)$ is closely similar to the one for calculating the Greene--Kleitman invariant at each vertex of $Q$. These similarities, along with some previous results from \cite{GPT19}, allow us to state the following result. 
	\begin{theorem} \label{thm:GenJFexist}
		Let $Q$ be an $A_n$ type quiver. Let $Y$ be a finite-dimensional representation over an algebraically closed field $\mathbb{K}$. Then $\NEnd(Y)$ is an irreducible algebraic variety. Furthermore, there is a maximum value of $\JF$, with respect to $\pmb{\unlhd}$, on $\NEnd(Y)$ which is attained on a dense open set of $\NEnd(Y)$, and this value is exactly $\GK(Y)$.
	\end{theorem}
	\begin{proof}[Proof] 
		As a direct consequence of \cite[Theorem 2.3]{GPT19} stated in a more considerable generality (instead of finite-dimensional representation $Y$ over an $A_n$ type quiver $Q$, they proved the same result for any finite-dimensional left module $Y$ over a finite-dimensional $\mathbb{K}$-algebra), we can already affirm that $\NEnd(Y)$ is an irreducible algebraic variety, and the fact that $\JF$ admits a maximal value on $\NEnd(Y)$ attained on a dense open set. 
		
		We must prove that this maximal value is $\GK(Y)$. This result is a consequence of \cref{thm:Gansner}. We can in fact notice that $\GK(Y) = (\Delta(\Gamma(q)))_{q \in Q_0}$ where $\Gamma(q)$ is the full subquiver of the Auslander--Reiten quiver of $\rep(Q)$ in which we replace each vertex corresponding to the isomorphism class of  $X_K$ by a chain of length $m_K = \mult(Y,X_K)$. This completes the proof.
	\end{proof} 
	By \cref{thm:GenJFexist}, we can define $\GenJF(X)$, the \new{generic Jordan form data of $X$}, as this maximal value of $\JF$ on $\NEnd(X)$. Keep in mind that $\GenJF(X) = \GK(X)$. We only change its name for representation-theoretic purposes.
	\begin{definition}\label{JRdef}
		A subcategory $\mathscr{C}$ of $\rep(Q)$ is called \new{Jordan recoverable} if from a tuple of partitions $\pmb{\lambda}$ there is at most a unique (up to isomorphism) $X \in \mathscr{C}$ such that $\GenJF(X) = \pmb{\lambda}$.
	\end{definition}
	The Jordan recoverable categories of $\rep(Q)$ are precisely those for which $\GK$ is a complete invariant.
	\begin{ex} \label{ex:JR} Consider the $A_3$ type quiver $Q$ in \cref{fig:AR2}.
		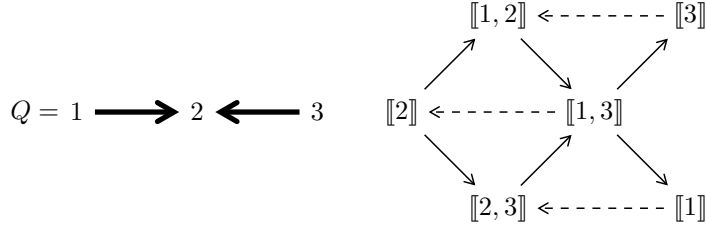
\begin{figure}[ht!]
			\centering
			\begin{tikzpicture}[->,line width=0.7mm,>= angle 60,color=black,scale=0.8]
				\node (Q) at (-.7,0){$Q=$};
				\node (1) at (0,0){$1$};
				\node (2) at (2,0){$2$};
				\node (3) at (4,0){$3$};
				\draw (1) --  (2);
				\draw (3) --  (2);
				\begin{scope}[xshift=7cm, ->,line width=0.2mm,>= angle 60,scale=1.6]
					\node[black] (1) at (-1,0){$\llrr{2}$};
					\node[black] (2) at (0,1){$\llrr{1,2}$};
					\node[black] (3) at (0,-1){$\llrr{2,3}$};
					\node[black] (4) at (1,0){$\llrr{1,3}$};
					\node[black] (5) at (2,1){$\llrr{3}$};
					\node[black] (6) at (2,-1){$\llrr{1}$};
					\draw (1) --  (2);
					\draw (1) -- (3);
					\draw (2) -- (4);
					\draw (3) -- (4);
					\draw (4) -- (5);
					\draw (4) -- (6);
					\draw[dashed] (4) -- (1);
					\draw[dashed] (5) -- (2);
					\draw[dashed] (6) -- (3);
				\end{scope}
			\end{tikzpicture}
			\caption{\label{fig:AR2} The $A_3$ type quiver considered (on the left) and its Auslander--Reiten quiver (on the right).}
		\end{figure}
		
		Here are some examples of Jordan recoverable $\rep(Q)$ subcategories.
		\begin{enumerate}[label = $\bullet$]
			\item The subcategory $\mathscr{C}_1 = \Cat_Q(\{\llrr{1}, \llrr{2}, \llrr{3}\})$ is Jordan recoverable: the dimension vectors of the indecomposable representations which generate $\mathscr{C}_1$ are linearly independent. More explicitly, for $X \in \mathscr{C}_1$, there exists a unique triplet $(a,b,c) \in \mathbb{N}^3$ such that $X \cong X_{\llrr{1}}^a \oplus X_{\llrr{2}}^b \oplus X_{\llrr{3}}^c$. Following calculation of $\GK(X)$ (see \cref{fig:AR2Simples}), we get $\GenJF(X) = ((a),(b),(c))$. This data determines $X \in \mathscr{C}$. 
			\begin{figure}[ht!]
				\centering
				\begin{tikzpicture}[->,line width=0.2mm,>= angle 60,scale=0.75]
					
					\tkzDefPoint(0.15,1){a}
					\tkzDefPoint(3,-1.85){b}
					\tkzDefPoint(3.85,-1){c}
					\tkzDefPoint(1,1.85){d}
					\tkzDrawPolygon[line width = 2mm, color = black, fill = black!10](a,b,c,d);
					
					\node[black,circle,draw] (1) at (0,0){$b$};
					\node[black,circle,draw] (2) at (1,1){$0$};
					\node[black,circle,draw] (3) at (1,-1){$0$};
					\node[black,circle,draw] (4) at (2,0){$0$};
					\node[black,circle,draw] (5) at (3,1){$c$};
					\node[black,circle,draw] (6) at (3,-1){$a$};
					\draw (1) --  (2);
					\draw (1) -- (3);
					\draw (2) -- (4);
					\draw (3) -- (4);
					\draw (4) -- (5);
					\draw (4) -- (6);
					\draw[dashed] (4) -- (1);
					\draw[dashed] (5) -- (2);
					\draw[dashed] (6) -- (3);
					
					\begin{scope}[xshift=5.5cm]
						
						\tkzDefPoint(-0.85,0){a}
						\tkzDefPoint(1,-1.85){b}
						\tkzDefPoint(2.85,0){c}
						\tkzDefPoint(1,1.85){d}
						\tkzDrawPolygon[line width = 2mm, color = black, fill = black!10](a,b,c,d);
						
						\node[black,circle,draw] (1) at (0,0){$b$};
						\node[black,circle,draw] (2) at (1,1){$0$};
						\node[black,circle,draw] (3) at (1,-1){$0$};
						\node[black,circle,draw] (4) at (2,0){$0$};
						\node[black,circle,draw] (5) at (3,1){$c$};
						\node[black,circle,draw] (6) at (3,-1){$a$};
						\draw (1) --  (2);
						\draw (1) -- (3);
						\draw (2) -- (4);
						\draw (3) -- (4);
						\draw (4) -- (5);
						\draw (4) -- (6);
						\draw[dashed] (4) -- (1);
						\draw[dashed] (5) -- (2);
						\draw[dashed] (6) -- (3);
						
					\end{scope}
					
					\begin{scope}[xshift=11cm]
						
						\tkzDefPoint(0.15,-1){a}
						\tkzDefPoint(1,-1.85){b}
						\tkzDefPoint(3.85,1){c}
						\tkzDefPoint(3,1.85){d}
						\tkzDrawPolygon[line width = 2mm, color = black, fill = black!10](a,b,c,d);
						
						\node[black,circle,draw] (1) at (0,0){$b$};
						\node[black,circle,draw] (2) at (1,1){$0$};
						\node[black,circle,draw] (3) at (1,-1){$0$};
						\node[black,circle,draw] (4) at (2,0){$0$};
						\node[black,circle,draw] (5) at (3,1){$c$};
						\node[black,circle,draw] (6) at (3,-1){$a$};
						\draw (1) --  (2);
						\draw (1) -- (3);
						\draw (2) -- (4);
						\draw (3) -- (4);
						\draw (4) -- (5);
						\draw (4) -- (6);
						\draw[dashed] (4) -- (1);
						\draw[dashed] (5) -- (2);
						\draw[dashed] (6) -- (3);
					\end{scope}
				\end{tikzpicture}
				\caption{\label{fig:AR2Simples} The rectangle corresponds to indecomposable representations that are in $\mathscr{C}_{Q,m}$ for $m=1,2,3$ from left to right. We use it as in \cref{fig:explitGK} to get $\GenJF(X)$ for $X \in \mathscr{C}_1$ in \cref{ex:JR}.}
			\end{figure}
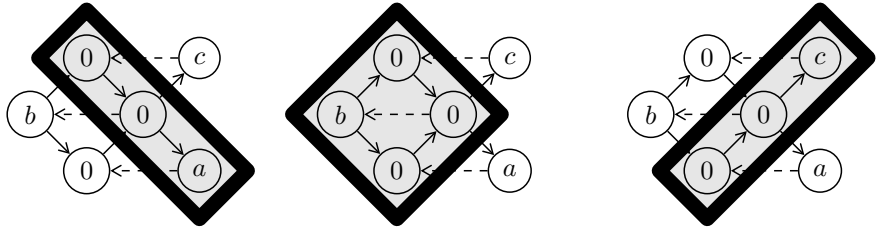 
			
			\item The subcategory $\mathscr{C}_2 = \Cat_Q(\{\llrr{1}, \llrr{1,3}, \llrr{3}\})$ is Jordan recoverable. For $X \in \mathscr{C}$, there exists $(a,b,c) \in \mathbb{N}^3$ such that $X \cong X_{\llrr 1}^a \oplus X_{\llrr{1,3}}^b \oplus X_{\llrr{3}}^c$. Following \cref{fig:AR2Cq2}, we get that $\GenJF(X) = ((a+b),(b),(b+c)).$ Note that we can recover $X \in \mathscr{C}$ from this data. 
			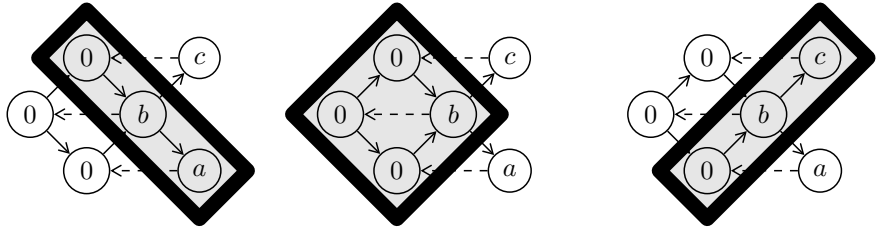
\begin{figure}[ht!]
				\centering
				\begin{tikzpicture}[->,line width=0.2mm,>= angle 60,scale=0.75]
					
					\tkzDefPoint(0.15,1){a}
					\tkzDefPoint(3,-1.85){b}
					\tkzDefPoint(3.85,-1){c}
					\tkzDefPoint(1,1.85){d}
					\tkzDrawPolygon[line width = 2mm, color = black, fill = black!10](a,b,c,d);
					
					\node[black,circle,draw] (1) at (0,0){$0$};
					\node[black,circle,draw] (2) at (1,1){$0$};
					\node[black,circle,draw] (3) at (1,-1){$0$};
					\node[black,circle,draw] (4) at (2,0){$b$};
					\node[black,circle,draw] (5) at (3,1){$c$};
					\node[black,circle,draw] (6) at (3,-1){$a$};
					\draw (1) --  (2);
					\draw (1) -- (3);
					\draw (2) -- (4);
					\draw (3) -- (4);
					\draw (4) -- (5);
					\draw (4) -- (6);
					\draw[dashed] (4) -- (1);
					\draw[dashed] (5) -- (2);
					\draw[dashed] (6) -- (3);
					
					\begin{scope}[xshift=5.5cm]
						
						\tkzDefPoint(-0.85,0){a}
						\tkzDefPoint(1,-1.85){b}
						\tkzDefPoint(2.85,0){c}
						\tkzDefPoint(1,1.85){d}
						\tkzDrawPolygon[line width = 2mm, color = black, fill = black!10](a,b,c,d);
						
						\node[black,circle,draw] (1) at (0,0){$0$};
						\node[black,circle,draw] (2) at (1,1){$0$};
						\node[black,circle,draw] (3) at (1,-1){$0$};
						\node[black,circle,draw] (4) at (2,0){$b$};
						\node[black,circle,draw] (5) at (3,1){$c$};
						\node[black,circle,draw] (6) at (3,-1){$a$};
						\draw (1) --  (2);
						\draw (1) -- (3);
						\draw (2) -- (4);
						\draw (3) -- (4);
						\draw (4) -- (5);
						\draw (4) -- (6);
						\draw[dashed] (4) -- (1);
						\draw[dashed] (5) -- (2);
						\draw[dashed] (6) -- (3);
						
					\end{scope}
					
					\begin{scope}[xshift=11cm]
						
						\tkzDefPoint(0.15,-1){a}
						\tkzDefPoint(1,-1.85){b}
						\tkzDefPoint(3.85,1){c}
						\tkzDefPoint(3,1.85){d}
						\tkzDrawPolygon[line width = 2mm, color = black, fill = black!10](a,b,c,d);
						
						\node[black,circle,draw] (1) at (0,0){$0$};
						\node[black,circle,draw] (2) at (1,1){$0$};
						\node[black,circle,draw] (3) at (1,-1){$0$};
						\node[black,circle,draw] (4) at (2,0){$b$};
						\node[black,circle,draw] (5) at (3,1){$c$};
						\node[black,circle,draw] (6) at (3,-1){$a$};
						\draw (1) --  (2);
						\draw (1) -- (3);
						\draw (2) -- (4);
						\draw (3) -- (4);
						\draw (4) -- (5);
						\draw (4) -- (6);
						\draw[dashed] (4) -- (1);
						\draw[dashed] (5) -- (2);
						\draw[dashed] (6) -- (3);
					\end{scope}
				\end{tikzpicture}
				\caption{\label{fig:AR2Cq2} The way to get $\GenJF(X)$ for $X \in \mathscr{C}_2$ in \cref{ex:JR}.}
			\end{figure} \qedhere
		\end{enumerate}
	\end{ex}
	In general, contrary to the previous example, we must deal with many equations to prove that a subcategory $\mathscr{C} \subset \rep(Q)$ is Jordan recoverable. We can ask whether there is a more general way to recover $X$ from its generic Jordan form data.
	
	\subsection{Generic representation with fixed Jordan form data}
	\label{ss:GenRepDisc}
	
	This section discusses two algebraic ways to get a specific canonical representation from fixed Jordan form data. 
	
	First, we recall the general setting in which we work. Let $Q$ be an $A_n$ type quiver, and $\pmb{d} = (d_q)_{q \in Q_0} \in \mathbb{N}^{n}$. Write $\rep(Q,\pmb{d})$ for the set of $E \in \rep(Q)$ such that $\vdim(E) = \pmb{d}$. One can identify this set as follows \[\rep(Q,\pmb{d}) \cong \prod_{\alpha \in Q_1} \Hom_\mathbb{K}(\mathbb{K}^{d_{s(\alpha)}}, \mathbb{K}^{d_{t(\alpha)}}) = \prod_{\alpha \in Q_1} \operatorname{Mat}_{d_{s(\alpha)} \times d_{t(\alpha)}} (\mathbb{K})\] by choosing a basis for each of the vector spaces. We endow $\rep(Q,\pmb{d})$ with the action of the algebraic group $\pmb{\operatorname{GL}}_{\pmb{d}}(\mathbb{K}) = \prod_{q \in Q_0} \operatorname{GL}_{d_q}(\mathbb{K})$ which changes the basis at each vertex. The orbits of this group action are the isomorphism classes of the representations in $\rep(Q,\pmb{d})$.
	
	Now we present a method developed by \cite{GPT19}. For each $q \in Q_0$, fix a $\mathbb{K}$-vector space $V_q$ such that $\dim(V_q) = d_q$, and a nilpotent endomorphism $N_q : V_q \longrightarrow V_q$. Let $\rep(Q,N)$ be the set of representations $E \in \rep(Q)$ such that, for all $q \in Q_0$, $E_q = V_q$, and $N = (N_q)_{q \in Q_0} \in \NEnd(E)$. 
	
	\begin{prop}[{\cite[Section 2.3]{GPT19}}]
		\label{prop:firstGenRep}
		Let $Q$ be an $A_n$ quiver, and $\pmb{d} \in \mathbb{N}^n$. For any collection of $\mathbb{K}$-vector spaces $(V_q)_{q \in Q_0}$ such that $\dim(V_q) = d_q$, and any collection of nilpotent endomorphisms $N = (N_q: V_q \longrightarrow V_q)_{q \in Q_0}$, the set $\rep(Q,N)$ is an irreducible variety, and there exists a dense open set $\Omega_N \subseteq \rep(Q,N)$ within which the representations are all isomorphic. Moreover, for any collection of nilpotent endomorphisms $N' = (N'_q : V_q \longrightarrow V_q)_{q \in Q_0}$ such that $\JF(N) = \JF(N')$, then the representations of $\Omega_N$ and those of $\Omega_{N'}$ are isomorphic.
	\end{prop}
	
	The proof of this result is mainly based on Kac's theorem \cite[p.85]{Kac80} and the fact that the indecomposable representations of a Dynkin quiver are characterized by their dimension vectors.
	
	Therefore, for any $n$-tuple of integer partitions $\pmb{\lambda}$, we denote by $\GenRep(\pmb{\lambda})$ the representation $\GenRep(N)$ for some $n$-tuple of nilpotent endomorphisms $N = (N_q : \mathbb{K}^{|\lambda_q|} \longrightarrow \mathbb{K}^{|\lambda_q|})_{q\in Q_0}$ whose Jordan form data are parametrized by $\pmb{\lambda}$. We call it the \new{generic representation with Jordan form $\pmb{\lambda}$}.
	
	In the following, we present a slightly different way to define the generic representation of $Q$ with Jordan form $\pmb{\lambda}$
	
	Let $\pmb{\lambda} =(\lambda^q)_{q \in Q_0}$ be a $n$-tuple of integer partitions. Denote by $\rep(Q,\pmb{\lambda})$ the set of representations $Y = ((Y_q)_{q \in Q_0}, (Y_\alpha)_{\alpha \in Q_1})$ such that $\vdim(Y) = (|\lambda^q|)_{q \in Q_0}$, and for which there exists a nilpotent endomorphism $N \in \NEnd(Y)$ with $\JF(N) = \pmb{\lambda}$.
	
	\begin{prop}
		\label{prop:secondGenRep}
		Let $Q$ be an $A_n$ type quiver, and $\pmb{\lambda} = (\lambda^q)_{q \in Q_0}$ be a $n$-tuple of integer partitions. Then $\rep(Q,\pmb{\lambda})$ is an irreducible space. Moreover, there exists a dense open set $\mho \subset \rep(Q,\pmb{\lambda})$ such that all the representations in $\mho$ are isomorphic.
	\end{prop}
	
	\begin{proof} 
		Set $\pmb{d} = (|\lambda^q|)_{q \in Q_0}$. For any $E \in \rep(Q,\pmb{\lambda})$ and $G \in \pmb{\operatorname{GL}}_{\pmb{d}}(\mathbb{K})$, we have $G \cdot E \in \rep(Q,\pmb{\lambda})$. Moreover, the action of $ \pmb{\operatorname{GL}}_{\pmb{d}}(\mathbb{K})$ is transitive on nilpotent endomorphisms of the same Jordan form. Thus, for any collection $N = (N_q)_{q \in Q_0}$ of nilpotent endomorphisms whose Jordan forms are given by $\pmb{\lambda}$, $\rep(Q,\pmb{\lambda}) = \pmb{\operatorname{GL}}_{\pmb{d}}(\mathbb{K}) \cdot \rep(Q,N)$. In the following, we fix such a collection of nilpotent endomorphisms $N = (N_q)_{q \in Q_0}$.
		
		We set \[\Phi_N : \left\{ \begin{matrix}
			\pmb{\operatorname{GL}}_{\pmb{d}}(\mathbb{K}) \times \rep(Q,N) &  \longrightarrow & \rep(Q,\pmb{d}) \\
			(\phi, E) & \longmapsto & \phi \cdot E
		\end{matrix}
		\right.. \] Note that this is an algebraic morphism. The space  $\pmb{\operatorname{GL}}_{\pmb{d}}(\mathbb{K}) \times \rep(Q,N)$ is irreducible as a product of irreducible spaces. So $\rep(Q,\pmb{\lambda}) = \Phi_N(\pmb{\operatorname{GL}}_{\pmb{d}}(\mathbb{K}) \times \rep(Q,N)) $ is an irreducible space.
		
		By \cref{prop:firstGenRep}, there exists a dense open set $\Omega_N \subset \rep(Q,N)$ within which the representations are all isomorphic. Set $\mho = \Phi_N (\pmb{\operatorname{GL}}_{\pmb{d}}(\mathbb{K}) \times \Omega_N)$. It is clear that $\mho$ is dense in $\rep(Q,\pmb{\lambda})$. Moreover, as the image of an open set by an algebraic morphism is open in the closure of its image \cite[Section I.8, Corollary 2]{M06}, we have that $\mho$ is open in $\rep(Q,\pmb{\lambda})$.
	\end{proof}
	
	The previous result allows us to construct $\GenRep(\pmb{\lambda})$ by exhibiting a generic behavior among the representations in $\rep(Q, \pmb{\lambda})$. 
	
	\begin{remark}
		For any Dynkin-type quiver $Q$, the two algebraic ways presented to define the generic representation with a fixed Jordan form coincide. Propositions \ref{prop:firstGenRep} and  \ref{prop:secondGenRep} hold in this setting.
	\end{remark}

	\subsection{Canonical Jordan recoverability}
	\label{ss:CJR}
	
	Consider $\mathscr{C}$ a Jordan recoverable subcategory of $\rep(Q)$. One can legitimately ask if $\GenRep$ gives an inverse of $\GenJF$. However, \cref{A2ex} highlights an example of a Jordan recoverable category for which it does not work. Therefore, we need to refine the notion of Jordan recoverability.
	\begin{definition} \label{CJRdef}
		A subcategory $\mathscr{C}$ of $\rep(Q)$ is said to be \new{canonically Jordan recoverable} if, for any  $X \in \mathscr{C}$, $\GenRep(\GenJF(X)) \cong X$. 
	\end{definition}
	We have an inverse of $\GK$ in such a category. Then, obviously, any canonically Jordan recoverable category is Jordan recoverable. However, there are Jordan recoverable categories that are not canonically Jordan recoverable.
	\begin{ex} \label{ex:CJR} Let $Q$ be the $A_3$ type quiver of \cref{ex:JR}.
		\begin{enumerate}[label = $\bullet$]
			\item The category $\mathscr{C}_1$ is not canonically Jordan recoverable by following the explanations already given in \cref{A2ex}.
			
			\item The category $\mathscr{C}_2$ is canonically Jordan recoverable. Let $\pmb{\lambda} = ((a+b),(b),(b+c))$ for a fixed triplet $(a,b,c) \in \mathbb{N}^3$. Consider $Y \in \rep(Q,\pmb{\lambda})$ and $N = (N_1,N_2,N_3) \in \NEnd(Y)$ such that $\JF(N) = \pmb{\lambda}$. There exists $u_1 \in Y_1$ such that $N_1^{a+b-1}(u_1) \neq 0$. Thus, by writing $u_i = N_1^{i-1}(u_1)$ for $i \in \{1,\ldots,a+b\}$, we get that $(u_1, \ldots, u_{a+b})$ is a basis of $Y_1$, adapted to $N$. Similarly, we construct the bases $(v_1, \ldots, v_b)$ and $(w_1, \ldots, w_{b+c})$ of respectively $Y_2$ and $Y_3$ such that they are adapted to $N$ (see \cref{fig:nilpbases}). We now have to describe $Y_\alpha$ and $Y_\beta$.
			\begin{figure}[ht!]
				\centering
				\begin{tikzpicture}[->,line width=0.7mm,>= angle 60,color=black,scale=1]
					\node (Q) at (-.9,0){$Y=$};
					\node (1) at (0,0){$\mathbb{K}^{a+b}$};
					\node (2) at (3,0){$\mathbb{K}^b$};
					\node (3) at (6,0){$\mathbb{K}^{b+c}$};
					\draw (1) -- node[above]{$Y_\alpha$} (2);
					\draw (3) -- node[above]{$Y_\beta$} (2);
					\begin{scope}[yshift=-1cm, ->,line width=0.2mm,>= angle 60,scale=1]
						\node[black] (a) at (0,0){$u_1$};
						\node[black] (b) at (0,-1){$u_2$};
						\node[black] (c) at (0,-2){$u_{b-1}$};
						\node[black] (d) at (0,-3){$u_b$};
						\node[black] (e) at (0,-4){$u_{b+1}$};
						\node[black] (f) at (0,-5){$u_{a+b}$};
						\node[black] (g) at (0,-6){$0$};
						\draw (a) -- node[left]{$N_1$}  (b);
						\draw[dotted] (b) --  (c);
						\draw (c) -- node[left]{$N_1$}  (d);
						\draw (d) --  node[left]{$N_1$}  (e);
						\draw[dotted](e) -- (f);
						\draw (f) -- node[left]{$N_1$} (g);
						
						\node[black] (a1) at (3,0){$v_1$};
						\node[black] (b1) at (3,-1){$v_2$};
						\node[black] (c1) at (3,-2){$v_{b-1}$};
						\node[black] (d1) at (3,-3){$v_b$};
						\node[black] (e1) at (3,-4){$0$};
						\draw (a1) -- node[left]{$N_2$}  (b1);
						\draw[dotted] (b1) --  (c1);
						\draw (c1) -- node[left]{$N_2$}  (d1);
						\draw (d1) --  node[left]{$N_2$}  (e1);
						
						\node[black] (a2) at (6,0){$w_1$};
						\node[black] (b2) at (6,-1){$w_2$};
						\node[black] (c2) at (6,-2){$w_{b-1}$};
						\node[black] (d2) at (6,-3){$w_b$};
						\node[black] (e2) at (6,-4){$w_{b+1}$};
						\node[black] (f2) at (6,-6){$w_{b+c}$};
						\node[black] (g2) at (6,-7){$0$};
						\draw (a2) -- node[left]{$N_3$}  (b2);
						\draw[dotted] (b2) --  (c2);
						\draw (c2) -- node[left]{$N_3$}  (d2);
						\draw (d2) --  node[left]{$N_3$}  (e2);
						\draw[dotted](e2) -- (f2);
						\draw (f2) -- node[left]{$N_3$} (g2);
						
						\draw [dashed,line width=0.2pt]
						($ (a1.north west)+ (-0.6,0) $)
						rectangle
						($ (d1.south east) + (0.6, -0.3) $);
						
						\draw [dashed,line width=0.2pt]
						($ (b1.north west)+ (-0.4,0) $)
						rectangle
						($ (d1.south east) + (0.4,-0.2) $);
						
						\draw [dashed,line width=0.2pt]
						($ (c1.north west)+ (-0.05,0) $)
						rectangle
						($ (d1.south east) + (0.2,-0.1) $);
						
						\draw [dashed,line width=0.2pt]
						($ (d1.north west) $)
						rectangle
						($ (d1.south east) $);
						
						\draw [line width=0.7mm] (a) -- ($ (a1.west)+ (-0.6,0) $);
						\draw [line width=0.7mm] (b) -- ($ (b1.west)+ (-0.4,0) $);
						\draw [line width=0.7mm] (c) -- ($ (c1.west)+ (-0.05,0) $);
						\draw [line width=0.7mm] (d) -- ($ (d1.west) $);
						
						\draw [line width=0.7mm] (a2) -- ($ (a1.east)+ (0.6,0) $);
						\draw [line width=0.7mm] (b2) -- ($ (b1.east)+ (0.4,0) $);
						\draw [line width=0.7mm] (c2) -- ($ (c1.east)+ (-0.05,0) $);
						\draw [line width=0.7mm] (d2) -- ($ (d1.east) $);
					\end{scope}
				\end{tikzpicture}
				\caption{\label{fig:nilpbases} Illustration of the configuration described to study the canonical Jordan recoverability of $\mathscr{C}_2$. Note that $Y_\alpha(u_i) = 0$ and $Y_\beta(w_i)=0$ for $i > b$ by square commutativity relations.}
			\end{figure}
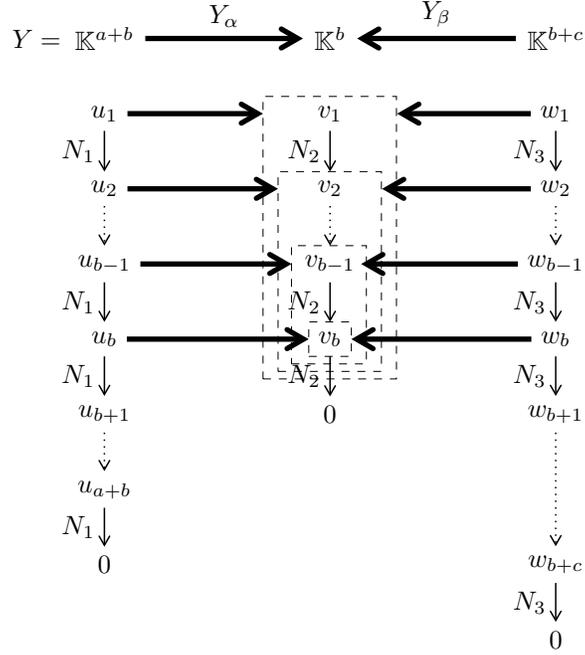
			Thanks to the chosen bases, and square commutativity relations satisfied by $N$, we only have to describe the image of $u_1$ by $Y_\alpha$ to describe all $Y_\alpha$, and the same goes for $Y_\beta$. Then $Y_\alpha(u_1)$ has to be a $\mathbb{K}$-linear combination of $(v_1, \ldots, v_b)$, say $(\alpha) : Y_\alpha(u_1) = k_1 v_1 + \ldots + k_b v_b$ with $k_1, \ldots, k_b \in \mathbb{K}$. Among all the choices we could make, there exists a dense open set $\Omega_1 \subset \rep(Q, \pmb{\lambda})$ such that for $Y \in \Omega_1$, $Y_\alpha(u_1) \notin \IIm(N_2)$ (this dense open set can be seen as taking $k_1 \neq 0$ in $(\alpha)$). Analogously, there exists a dense open set $\Omega_2$ such that $Y_\beta(w_1) \notin \IIm(N_2)$. Therefore, for all $Y$ in the dense open set $\Omega_1 \cap \Omega_2$, we get that $Y \cong X_{\llrr{1}}^a \oplus X_{\llrr{1,3}}^b \oplus X_{\llrr{3}}^c$. This proves our claim. \qedhere
		\end{enumerate}
	\end{ex}
	
	This paper describes all the canonically Jordan recoverable subcategories of $\rep(Q)$ for any $A_n$ type quiver.
	
	\section{Storability}
	\label{s:store} 
	
	In this section, we introduce a relation among integer partitions, which we call \emph{storability}, that corresponds to certain interlacing behaviors. We introduce diagonal transformations on $n$-tuples of integer partitions, which underlie the behaviors of the generic Jordan form data under the action of the reflection functors. Later on, in \cref{s:operationsCJR}, we will use it to define operators on additive subcategories that preserve Jordan recoverability and canonical Jordan recoverability, under certain assumptions. These are relevant tools to prove \cref{maintheorem}.
	
	\subsection{Storable pairs}
	\label{ss:store pairs} 
	
	\begin{definition}\label{def:storable pairs}
		Let $\lambda$ and $\mu$ be two integer partitions. The pair $(\lambda, \mu)$ is  \new{storable} if for all $i \in \mathbb{N}^*$, $ \lambda_i \geqslant \mu_i \geqslant \lambda_{i+1}$ (adding zero parts as needed). Such a pair is \new{strongly storable} if in addition $\lambda_1 = \mu_1$. 
	\end{definition}
	We represent and characterize storable pairs visually. Fill two rows of $45^\circ$ rotated boxes with the entries of $\lambda$ and $\mu$ as in \cref{fig:store}, adding infinitely many zeros to the right. Then $(\lambda, \mu)$ is storable if the entries weakly decrease from left to right.
	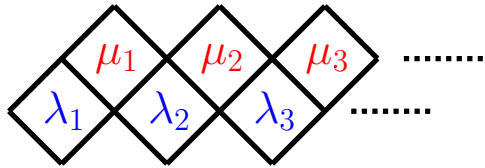
\begin{figure}[ht!]
		\centering
		\begin{tikzpicture}[scale=0.7]
			\draw[line width=0.7mm](-1,1) edge (1,3);
			\draw[line width=0.7mm](0,0) edge (3,3);
			\draw[line width=0.7mm](2,0) edge (5,3);
			\draw[line width=0.7mm](4,0) edge (6,2);
			
			\draw[line width=0.7mm](-1,1) edge (0,0);
			\draw[line width=0.7mm](0,2) edge (2,0);
			\draw[line width=0.7mm](1,3) edge (4,0);
			\draw[line width=0.7mm](3,3) edge (5,1);
			\draw[line width=0.7mm](5,3) edge (6,2);
			
			\draw[line width=0.7mm,dotted](6.5,2) edge (8,2);
			\draw[line width=0.7mm,dotted](5.5,1) edge (7,1);
			
			\tkzLabelPoint[below](0.05,1.55){{\huge $\color{blue}{\lambda_1}$}}
			\tkzLabelPoint[below](2.05,1.55){{\huge $\color{blue}{\lambda_2}$}}
			\tkzLabelPoint[below](4.05,1.55){{\huge $\color{blue}{\lambda_3}$}}
			
			\tkzLabelPoint[below](1.05,2.45){{\huge $\color{red}{\mu_1}$}}
			\tkzLabelPoint[below](3.05,2.45){{\huge $\color{red}{\mu_2}$}}
			\tkzLabelPoint[below](5.05,2.45){{\huge $\color{red}{\mu_3}$}}
		\end{tikzpicture}
		\caption{\label{fig:store} Illustration of storability of $(\lambda, \mu)$.}
	\end{figure}
	We give two results that arise from the definition.
	\begin{lemma} \label{lem:elementarystorable} Let $\lambda$ and $\mu$ be two integer partitions. 
		\begin{enumerate}[label = \arabic*)]
			\item \label{estor1} If $(\lambda, \mu)$ and $(\mu, \lambda)$ are both storable, then $\lambda = \mu$;
			\item \label{estor2} If $(\lambda, \mu)$ is storable, then $\ell(\lambda) \in \{\ell(\mu), \ell(\mu) + 1\}$.
		\end{enumerate}
	\end{lemma}
	
	\subsection{Storable triplets}
	\label{ss:storetriplet}
	
	\begin{definition} \label{def:storabletriplets} Let $\lambda, \mu$ and $\nu$ be three integer partitions. The triplet $(\lambda, \mu, \nu)$ is \new{storable} if the two following conditions are satisfied: \begin{enumerate}[label = $\bullet$]
			\item either $(\lambda, \mu)$ or $(\mu, \lambda)$ is a storable pair;
			\item either $(\mu, \nu)$ or $(\nu, \mu)$ is a storable pair.
		\end{enumerate}
		More precisely, we say that $(\lambda, \mu, \nu)$ is:
		\begin{enumerate}
			\item[$(\boxplus \boxplus)$]  $(\boxplus,\boxplus)$-storable if $(\lambda, \mu)$ and $(\nu, \mu)$ are storable pairs;
			\item[$(\boxplus \boxminus)$]  $(\boxplus,\boxminus)$-storable if $(\lambda, \mu)$ and $(\mu, \nu)$ are storable pairs;
			\item[$(\boxminus \boxplus)$]  $(\boxminus,\boxplus)$-storable if $(\mu, \lambda)$ and $(\nu, \mu)$ are storable pairs;
			\item[$(\boxminus \boxminus)$] $(\boxminus,\boxminus)$-storable if $(\mu, \lambda)$ and $(\mu, \nu)$ are storable pairs.
		\end{enumerate}
		Such a triplet is \new{strongly storable} whenever $\lambda_1 = \mu_1$ or $\mu_1 = \nu_1$. 
	\end{definition}
	In \cref{fig:AllConf}, we illustrate the four storability configurations.
	\begin{figure}[ht!]
		\centering
		
		\scalebox{0.53}{
			\begin{tikzpicture}[scale=0.7]
				\draw[line width=0.7mm](0,2) -- (1,3) -- (4,0) ;
				\draw[line width=0.7mm](0,2) -- (3,-1) -- (5,1) ;
				\draw[line width=0.7mm](0,0) -- (1,-1) -- (4,2) ;
				\draw[line width=0.7mm](0,0) -- (3,3) -- (5,1);
				
				\draw[line width=0.7mm,dotted](4.5,2) edge (6,2);
				\draw[line width=0.7mm,dotted](5.5,1) edge (7,1);
				\draw[line width=0.7mm,dotted](4.5,0) edge (6,0);
				
				\tkzLabelPoint[below](2.05,1.45){{\huge $\color{red}{\mu_1}$}}
				\tkzLabelPoint[below](4.05,1.45){{\huge $\color{red}{\mu_2}$}}

				\tkzLabelPoint[below](1.05,2.55){{\huge $\color{blue}{\lambda_1}$}}
				\tkzLabelPoint[below](3.05,2.55){{\huge $\color{blue}{\lambda_2}$}}
				
				\tkzLabelPoint[below](1.05,0.4){{\huge $\color{darkgreen}{\nu_1}$}}
				\tkzLabelPoint[below](3.05,0.4){{\huge $\color{darkgreen}{\nu_2}$}}
				
				\tkzLabelPoint[below](3.05,-1.5){{\huge $(\boxplus \boxplus)$}}
			\end{tikzpicture} \qquad 
			\begin{tikzpicture}[scale=0.7]
				\draw[line width=0.7mm](-1,1) -- (1,3) -- (4,0) ;
				\draw[line width=0.7mm](-2,2) -- (-1,3) -- (3,-1) -- (4,0) ;
				\draw[line width=0.7mm](-2,2) -- (1,-1) -- (3,1) ;
				\draw[line width=0.7mm](0,0) -- (2,2);
				
				\draw[line width=0.7mm,dotted](2.5,2) edge (4,2);
				\draw[line width=0.7mm,dotted](3.5,1) edge (5,1);
				\draw[line width=0.7mm,dotted](4.5,0) edge (6,0);
				
				\tkzLabelPoint[below](0.05,1.45){{\huge $\color{red}{\mu_1}$}}
				\tkzLabelPoint[below](2.05,1.45){{\huge $\color{red}{\mu_2}$}}

				\tkzLabelPoint[below](-.95,2.55){{\huge $\color{blue}{\lambda_1}$}}
				\tkzLabelPoint[below](1.05,2.55){{\huge $\color{blue}{\lambda_2}$}}
				
				\tkzLabelPoint[below](1.05,0.4){{\huge $\color{darkgreen}{\nu_1}$}}
				\tkzLabelPoint[below](3.05,0.4){{\huge $\color{darkgreen}{\nu_2}$}}
				
				\tkzLabelPoint[below](1.05,-1.5){{\huge $(\boxplus \boxminus)$}}
			\end{tikzpicture} \qquad \begin{tikzpicture}[scale=0.7]
				\draw[line width=0.7mm](-2,0) -- (1,3) -- (3,1) ;
				\draw[line width=0.7mm](0,2) -- (2,0);
				\draw[line width=0.7mm](-1,1) -- (1,-1) -- (4,2) ;
				\draw[line width=0.7mm](-2,0) --(-1,-1) -- (3,3) -- (4,2);
				
				\draw[line width=0.7mm,dotted](4.5,2) edge (6,2);
				\draw[line width=0.7mm,dotted](3.5,1) edge (5,1);
				\draw[line width=0.7mm,dotted](2.5,0) edge (4,0);
				
				\tkzLabelPoint[below](0.05,1.45){{\huge $\color{red}{\mu_1}$}}
				\tkzLabelPoint[below](2.05,1.45){{\huge $\color{red}{\mu_2}$}}

				\tkzLabelPoint[below](1.05,2.55){{\huge $\color{blue}{\lambda_1}$}}
				\tkzLabelPoint[below](3.05,2.55){{\huge $\color{blue}{\lambda_2}$}}
				
				\tkzLabelPoint[below](-0.95,0.4){{\huge $\color{darkgreen}{\nu_1}$}}
				\tkzLabelPoint[below](1.05,0.4){{\huge $\color{darkgreen}{\nu_2}$}}
				
				\tkzLabelPoint[below](1.05,-1.5){{\huge $(\boxminus \boxplus)$}}
			\end{tikzpicture}  \qquad  \begin{tikzpicture}[scale=0.7]
				\draw[line width=0.7mm](-1,1) -- (1,3) -- (4,0) ;
				\draw[line width=0.7mm](0,2) -- (3,-1) -- (4,0) ;
				\draw[line width=0.7mm](-1,1) -- (1,-1) -- (4,2) ;
				\draw[line width=0.7mm](0,0) -- (3,3) -- (4,2);
				
				\draw[line width=0.7mm,dotted](4.5,2) edge (6,2);
				\draw[line width=0.7mm,dotted](3.5,1) edge (5,1);
				\draw[line width=0.7mm,dotted](4.5,0) edge (6,0);
				
				\tkzLabelPoint[below](0.05,1.45){{\huge $\color{red}{\mu_1}$}}
				\tkzLabelPoint[below](2.05,1.45){{\huge $\color{red}{\mu_2}$}}

				\tkzLabelPoint[below](1.05,2.55){{\huge $\color{blue}{\lambda_1}$}}
				\tkzLabelPoint[below](3.05,2.55){{\huge $\color{blue}{\lambda_2}$}}
				
				\tkzLabelPoint[below](1.05,0.4){{\huge $\color{darkgreen}{\nu_1}$}}
				\tkzLabelPoint[below](3.05,0.4){{\huge $\color{darkgreen}{\nu_2}$}}
				
				\tkzLabelPoint[below](2.05,-1.5){{\huge $(\boxminus \boxminus)$}}
		\end{tikzpicture}}

		\caption{\label{fig:AllConf} Illustration of the four storability configurations of $(\lambda, \mu, \nu)$.}
	\end{figure}
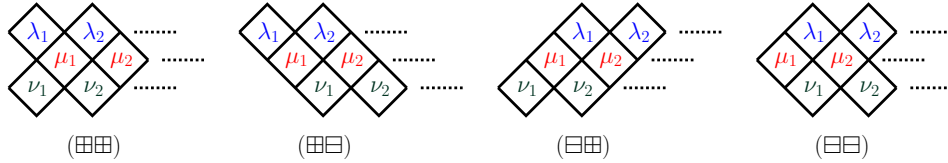
	\begin{definition}\label{def:diag}
		Let $\lambda, \mu$, and $\nu$ be three integer partitions. Assume that $(\lambda, \mu, \nu)$ is a storable triplet. We define the \new{diagonal transformation of $\mu$ in $(\lambda, \mu, \nu)$}, denoted $\diag(\lambda, \mu, \nu)$, to be the integer partition $\theta = (\theta_1, \theta_2, \ldots)$ such that:
		\begin{enumerate}[label = $\bullet$]
			
			\item if $(\lambda, \mu, \nu)$ is $(\boxplus, \boxplus)$-storable, then we define, for all $i \geqslant 1$, $$\theta_i =  \begin{cases}
				\max(\lambda_1, \nu_1) & \text{if } i = 1 \\
				\min(\lambda_{i-1},\nu_{i-1}) + \max(\lambda_i, \nu_i) - \mu_{i-1} & \text{otherwise;}\\
			\end{cases} $$
			
			\item if $(\lambda, \mu, \nu)$ is $(\boxplus, \boxminus)$-storable, then we define, for all $i \geqslant 1$, $$ \theta_i = \begin{cases} \lambda_1 + \max(\lambda_2, \nu_1) - \mu_1 & \text{if } i = 1\\
				\min(\lambda_i,\nu_{i-1}) + \max(\lambda_{i+1}, \nu_i) - \mu_{i}& \text{otherwise;}
			\end{cases}$$
			
			\item if $(\lambda, \mu, \nu)$ is $(\boxminus, \boxplus)$-storable, then we define, for all $i \geqslant 1$, $$ \theta_i = \begin{cases} \nu_1 + \max(\lambda_1, \nu_2) - \mu_1 & \text{if } i = 1\\
				\min(\lambda_{i-1},\nu_{i}) + \max(\lambda_{i}, \nu_{i+1}) - \mu_{i}& \text{otherwise;}
			\end{cases}$$
			
			\item if $(\lambda, \mu, \nu)$ is  $(\boxminus, \boxminus)$-storable, then we define, for all $i \geqslant 1$, $$\theta_i =  \min(\lambda_i, \nu_i) + \max(\lambda_{i+1}, \nu_{i+1}) - \mu_{i+1}.$$
		\end{enumerate}
	\end{definition}
	We can picture the diagonal operation as doing local operations for each square of $\mu$ in the diagram representing the storable triple $(\lambda, \mu, \nu)$ (\cref{fig:GenDiagop}).
	\begin{figure}[ht!]
		\centering
		
		\scalebox{0.6}{
			\begin{tikzpicture}[scale=1.3]
				
				\tkzDefPoint(1,1){a}
				\tkzDefPoint(2,2){b}
				\tkzDefPoint(3,1){c}
				\tkzDefPoint(2,0){d}
				\tkzDrawPolygon[line width = 2mm, color = black, fill = black!10](a,b,c,d);
				
				\draw[line width=0.7mm](0,2) -- (1,3) -- (4,0) ;
				\draw[line width=0.7mm](0,2) -- (3,-1) -- (4,0) ;
				\draw[line width=0.7mm](0,0) -- (1,-1) -- (4,2) ;
				\draw[line width=0.7mm](0,0) -- (3,3) -- (4,2);
				
				\draw[line width=0.7mm,dotted](4.5,2) edge (5.5,2);
				\draw[line width=0.7mm,dotted](3.5,1) edge (4.5,1);
				\draw[line width=0.7mm,dotted](4.5,0) edge (5.5,0);
				
				\draw[line width=0.7mm,dotted](-0.5,2) edge node[above]{\Large $\color{blue}{\lambda}$} (-1.5,2);
				\draw[line width=0.7mm,dotted](0.5,1) edge node[above]{\Large $\color{red}{\mu}$}(-0.5,1);
				\draw[line width=0.7mm,dotted](-0.5,0) edge node[above]{\Large $\color{darkgreen}{\nu}$} (-1.5,0);
				
				\draw[->,>= angle 60, line width=2mm](5,1) -- node[above]{\huge $\diag$} (7,1);
				
				\tkzLabelPoint[below](2,1.27){{\Huge $\color{red}{e}$}}

				\tkzLabelPoint[below](1,2.3){{\Huge $\color{blue}{a}$}}
				\tkzLabelPoint[below](3,2.4){{\Huge $\color{blue}{b}$}}
				
				\tkzLabelPoint[below](1,0.27){{\Huge $\color{darkgreen}{c}$}}
				\tkzLabelPoint[below](3,0.4){{\Huge $\color{darkgreen}{d}$}}
				
				\begin{scope}[xshift = 8cm]
					
					\tkzDefPoint(1,1){a}
					\tkzDefPoint(2,2){b}
					\tkzDefPoint(3,1){c}
					\tkzDefPoint(2,0){d}
					\tkzDrawPolygon[line width = 2mm, color = black, fill = black!10](a,b,c,d);
					
					\draw[line width=0.7mm](0,2) -- (1,3) -- (4,0) ;
					\draw[line width=0.7mm](0,2) -- (3,-1) -- (4,0) ;
					\draw[line width=0.7mm](0,0) -- (1,-1) -- (4,2) ;
					\draw[line width=0.7mm](0,0) -- (3,3) -- (4,2);
					
					\draw[line width=0.7mm,dotted](4.5,2) edge (5.5,2);
					\draw[line width=0.7mm,dotted](3.5,1) edge (4.5,1);
					\draw[line width=0.7mm,dotted](4.5,0) edge (5.5,0);
					
					\draw[line width=0.7mm,dotted](-0.5,2) edge (-1.5,2);
					\draw[line width=0.7mm,dotted](0.5,1) edge (-0.5,1);
					\draw[line width=0.7mm,dotted](-0.5,0) edge (-1.5,0);

					\tkzLabelPoint[below](2,1.5){{$\color{red}{\min(a,c)}$}}
					\tkzLabelPoint[below](2,1.15){{$\color{red}{+\max(b,d)}$}}
					\tkzLabelPoint[below](2,0.8){{$\color{red}{-e}$}}

					\tkzLabelPoint[below](1,2.3){{\Huge $\color{blue}{a}$}}
					\tkzLabelPoint[below](3,2.4){{\Huge $\color{blue}{b}$}}
					
					\tkzLabelPoint[below](1,0.27){{\Huge $\color{darkgreen}{c}$}}
					\tkzLabelPoint[below](3,0.4){{\Huge $\color{darkgreen}{d}$}}
					
				\end{scope};
		\end{tikzpicture}}
		\caption{\label{fig:GenDiagop} Illustration of the local operations to calculate $\diag(\lambda, \mu, \nu)$.}
	\end{figure}
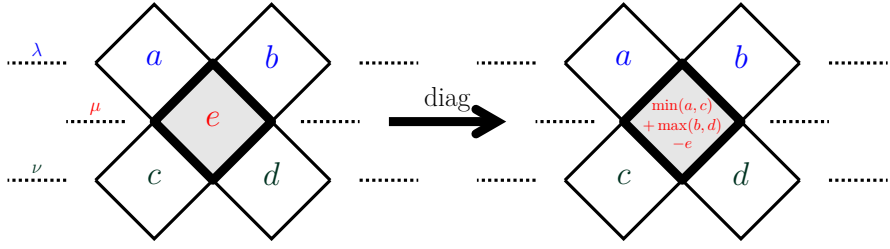
	Remark that $\lambda$ and $\nu$ play symmetric roles : $\diag(\lambda, \mu, \nu) = \diag(\nu, \mu, \lambda)$. Here are some elementary statements we get for the diagonal transformation.
	\begin{lemma}\label{lem:elementary diag store} Let $\lambda, \mu$ and $\nu$ be three integer partitions. When it is well-defined,  consider $\theta = \diag(\lambda, \nu, \mu)$.
		\begin{enumerate}[label = \arabic*)]
			\item \label{ediag0} If $(\lambda, \mu)$ is a storable pair, then $\diag(\lambda,\mu, \mu) = \lambda$.
			
			\item \label{ediag1}  If $(\lambda, \mu, \nu)$ is $(\boxplus, \boxplus)$-storable, then $(\lambda, \theta, \nu)$ is strongly $(\boxminus, \boxminus)$-storable.
			
			\item \label{ediag2}  If $(\lambda, \mu, \nu)$ is $(\boxplus,\boxminus)$-storable, then $(\lambda, \theta, \nu)$ is $(\boxplus, \boxminus)$-storable. 
			
			\item \label{ediag3} If $(\lambda, \mu, \nu)$ is $(\boxminus,\boxplus)$-storable, then $(\lambda, \theta, \nu)$ is $(\boxminus, \boxplus)$-storable.  
			
			\item \label{ediag4} If $(\lambda, \mu, \nu)$ is $(\boxminus,\boxminus)$-storable, then $(\lambda, \theta, \nu)$ is $(\boxplus, \boxplus)$-storable. 
			
			\item \label{ediag5} If $(\lambda, \mu, \nu)$ is either $(\boxplus, \boxplus)$-storable, $(\boxplus, \boxminus)$-storable, $(\boxminus, \boxplus)$-storable or 
			\\ strongly $(\boxminus, \boxminus)$-storable, then $\diag(\lambda, \theta, \nu) = \mu$.
		\end{enumerate}
	\end{lemma}
	
	\subsection{Rephrasing results of Garver--Patrias--Thomas}
	\label{ss:GPT}
	
	In \cite{GPT19}, they use the following notion. 
	\begin{definition} \label{def:interlaced}
		Let $\rho, \mu$ be two integer partitions, and $t \in \mathbb{Z}$. We say that $\rho$ and $\mu$ are \new{$t$-interlaced} if:
		\begin{enumerate}[label = $\bullet$]
			\item in the case $t \geqslant 0$, $$\rho_1 \geqslant \rho_2 \geqslant \ldots \geqslant \rho_t \geqslant \rho_{t+1} \geqslant \mu_1 \geqslant \rho_{t+2} \geqslant \rho_{t+3} \geqslant \mu_2 \geqslant \rho_{t+4} \geqslant \ldots$$
			
			\item in the case $t \leqslant 0$, $\rho_i = \mu_i$ for $1 \leqslant i \leqslant -t$ and $$\rho_{-t+1} \geqslant \mu_{-t+1} \geqslant \rho_{-t+2} \geqslant \rho_{-t+3} \geqslant \mu_{-t+2} \geqslant \rho_{-t+4} \geqslant \ldots$$
		\end{enumerate}
	\end{definition}
	We can state a link between $t$-interlaced pairs and storable triplets. For two integer partitions $\lambda, \nu$, we denote $\lambda + \nu$ the integer partition whose multiset of parts is composed of the parts of $\lambda$ and $\nu$. 
	\begin{lemma} \label{lem:storinter} Let $\lambda$, $\mu$ and $\nu$ be three integer partitions. The following assertions hold:
		\begin{enumerate}[label = \roman*)]
			\item if $(\lambda, \mu, \nu)$ is $(\boxplus, \boxplus)$-storable, then $\lambda+\nu$ and $\mu$ are $1$-interlaced;
			
			\item if $(\lambda, \mu, \nu)$ is either $(\boxplus, \boxminus)$-storable or $(\boxminus, \boxplus)$-storable, then $\lambda+\nu$ and $\mu$ are $0$-interlaced;
			
			\item if $(\lambda, \mu, \nu)$ is strongly $(\boxminus, \boxminus)$-storable, then $\lambda+\nu$ and $\mu$ are $-1$-interlaced.
		\end{enumerate}
	\end{lemma}
	\begin{remark} Note that:
		\begin{enumerate}[label = $\bullet$]
			\item if $(\lambda, \mu, \nu)$ is $(\boxminus, \boxminus)$-storable but not strongly $(\boxminus, \boxminus)$-storable, then $\lambda+\nu$ and $\mu$ are not interlaced.
			
			\item if $(\lambda, \mu, \nu)$ is such that $\lambda+\nu$ and $\mu$ are $t$-interlaced, then this does not imply that $(\lambda, \mu, \nu)$ is a storable triplet.
		\end{enumerate}
	\end{remark}
	Let $q$ be a vertex of an $A_n$ type quiver $Q$, $\pmb{d} \in \mathbb{N}^n$ and $\pmb{\pi} \vdash \pmb{d}$. We extend the tuple of partitions $\pmb{\pi}$ with $\pi^0 = \pi^{n+1} = (0)$. We will write that $\pmb{\pi}$ is \new{(strongly) $(\boxminus, \boxminus)$-storable at $q$} if  $(\pi^{q-1}, \pi^q, \pi^{q+1})$ is a (strongly) $(\boxminus, \boxminus)$-storable triplet. We use the same formulation for the three other storability configurations.
	
	Let $v$ be a source or a sink of $Q$. We define $\sigma_v(\pmb{\pi})$ to be the $n$-tuple of partitions obtained from $\pmb{\pi}$ by replacing $\pi^v$ with $\diag(\pi^{v-1}, \pi^{v}, \pi^{v+1})$.
	\begin{lemma} \label{thm:reflectionGenJF}
		Let $v$ be a vertex of a quiver of $A_n$ type. Let $\pmb{\pi}$ be a $n$-tuple of integer partitions such that $\pmb{\pi}$ is either $(\boxplus, \boxplus)$-storable, $(\boxplus, \boxminus)$-storable, $(\boxminus, \boxplus)$-storable or strongly $(\boxminus, \boxminus)$-storable at $v$. Consider $X = \GenRep(\pmb{\pi})$ and assume that $\GenJF(X) = \pmb{\pi}$. Then, if $v$ is a source, $\GenJF(\mathcal{R}_v^-(X)) = \sigma_v(\pmb{\pi})$. Similarly, if $v$ is a sink, $\GenJF(\mathcal{R}_v^+(X)) = \sigma_v(\pmb{\pi})$. 
	\end{lemma}
	\begin{proof}[Proof]
		Let $\pmb{\pi}$ be as assumed. Then  $\rho = \pi^{v+1} + \pi^{v-1}$ and $\mu = \pi^v$ are $t-$interlaced for $t \in \{1,0,-1\}$ by \cref{lem:storinter}. The desired result follows from \cite[Theorem 3.12]{GPT19}. 
	\end{proof}
	
	\begin{remark} \label{rem:t>0}
		\cite[Theorem 3.12]{GPT19} asks for having $t \geqslant 0$. However, one can check that the proof of this result relies on \cite[Lemma 3.6]{GPT19}, which is true for $t \in \mathbb{Z}$. Thus, in reality, this theorem holds for $t < 0$ as well.
	\end{remark}
	
	\begin{theorem}\label{thm:reflectionGenRep} Let $v$ be a vertex of an $A_n$ type quiver $Q$. Let $\pmb{\pi}$ be a $n$-tuple of integer partitions such that $\pmb{\pi}$ is $(\boxplus, \boxplus)$-storable, $(\boxplus, \boxminus)$-storable, $(\boxminus, \boxplus)$-storable, or strongly $(\boxminus, \boxminus)$-storable at $v$. Assume that $X \cong \GenRep(\pmb{\pi})$. Then if $v$ is a source, then $\mathcal{R}_v^-(X) \cong \GenRep(\sigma_v(\pmb{\pi}))$. Similarly, if $v$ is a sink, $\mathcal{R}_v^{+}(X) \cong \GenRep(\sigma_v(\pmb{\pi}))$
	\end{theorem}
	\begin{proof}[Proof]
		Let $\pmb{\pi}$ be as assumed. Then  $\rho = \pi^{v+1} + \pi^{v-1}$ and $\mu = \pi^v$ are $t-$interlaced for $t \in \{1,0,-1\}$ by \cref{lem:storinter}. The result we wished for follows from \cite[Theorem 3.10]{GPT19}.
	\end{proof}
	
	\section{Operations preserving canonical Jordan recoverability}
	\label{s:operationsCJR}
	
	In this section, we introduce elementary categorical operations: adding simple objects and applying reflection functors. Those operations allow one to build any additive subcategory of a Dynkin-type quiver. We show that, under specific assumptions, they preserve Jordan recoverability and canonical Jordan recoverability. Later on, in \cref{s:proof}, we use those operations to build all the maximal canonically Jordan recoverable, based on combinatorial data, developed in \cref{s:adj}, and an algorithm using those operations.
	
	\subsection{Adding a simple representation}
	\label{ss:addsimples}
	
	We define the operation \new{$\Adds_v$} on subcategories of $\rep(Q)$ by $\Adds_v(\mathscr{C}) = \add(\mathscr{C}, S_v)$ for any subcategory $\mathscr{C}$ of $\rep(Q)$. In general, this operation does not preserve the canonical Jordan recoverability property. This subsection shows it does so under a storability condition on generic Jordan forms of all $X \in \mathscr{C}$.
	
	First, we are interested in preserving Jordan recoverability. Before stating the result, we need the following lemma.
	\begin{lemma}\label{GenJFaddsimple}
		Let $v$ be a source or a sink of an $A_n$ type quiver $Q$. Consider $a \in \mathbb{N}$ and $X \in \rep(Q)$. Write $\pmb{\pi} = \GenJF(X)$. Then $\GenJF(S_v^a \oplus X) = \pmb{\xi}$ where $\xi^q = \pi^q$ if $q \neq v$ and $\xi^v = (\pi_1^v + a, \pi_2^v, \pi_3^v, \ldots).$
	\end{lemma}
	\begin{proof}[Proof]
		The lemma is a direct consequence of the combinatorial way to calculate $\pmb{\xi} = \GenJF(S_k^a \oplus X)$. \cref{thm:GenJFexist} tells us that $\GenJF$ is given by the Greene--Kleitman invariant $\GK$ introduced in \cref{ss:GKinv}.
		
		We first note that the calculation of $\pmb{\xi}$ differs from the one for $\pmb{\pi} = \GenJF(X)$ only at the vertex $v$. Moreover, as $v$ is a source (respectively a sink), any maximal path in the Auslander--Reiten quiver of $Q$ over indecomposable objects of $\mathscr{C}$ admitting $v$ in their support will go through $S_v$, a consequence of the fact that $S_v$ is at the end (respectively at the beginning) of any of those paths. Then $\xi_1^v = \pi_1^v + a$. For the other parts of $\xi^v$, as $S_v$ will not reappear in the calculation, $\xi_i^v = \pi_i^v$ for $i \geqslant 2$.
	\end{proof}
	Now we can give a sufficient assumption on Jordan recoverable subcategories $\mathscr{C}$ such that $\Adds_v(\mathscr{C})$ is also Jordan recoverable. 
	\begin{prop}\label{addsimple}
		Let $Q$ be a quiver of $A_n$ type, and $v$ be a source or a sink of $Q$. Let $\mathscr{C} \subset \rep(Q)$ be a Jordan recoverable category such that \begin{enumerate}[label = $(\star)$]
			\item \label{star} For any $X \in \mathscr{C}$, $ \GenJF(X)$ is  strongly $(\boxminus, \boxminus)$-storable at $v$. 
		\end{enumerate}
		Then $\mathscr{D} = \Adds_v(\mathscr{C})$ is Jordan recoverable.
	\end{prop}
	\begin{proof}[Proof]
		Let $v$, $Q$ and $\mathscr{C}$ be as assumed. Remark that $S_v$ is not an indecomposable object of $\mathscr{C}$ by \ref{star}. Consider $\mathscr{D} = \Adds_v(\mathscr{C})$ and let us prove that $\mathscr{D}$ is Jordan recoverable.
		
		Let $Y,Z \in \mathscr{D}$. We know that $Y \cong S_v^a \oplus Y'$ and $Z \cong S_v^b \oplus Z'$ with $Y',Z' \in \mathscr{C}$ and $a, b \in \mathbb{N}$. Suppose that $\GenJF(Y) = \GenJF(Z)$. If we take $\pmb{\lambda} = \GenJF(Y')$ and $\pmb{\mu} = \GenJF(Z')$, then, by \cref{GenJFaddsimple}, we get :
		\begin{enumerate}[label = $\arabic*)$]
			\item $\lambda^q = \mu^q$ for all $q \neq k$;
			
			\item $\lambda_1^v + a = \mu_1^v + b$;
			
			\item $\lambda_s^v = \mu_s^v$ for all $s > 1$. 
		\end{enumerate}
		By \ref{star} and $1)$, we know that $\lambda_1^v = \max(\lambda_1^{v-1}, \lambda_1^{v+1}) = \max(\mu_1^{v-1}, \mu_1^{v+1}) = \mu_1^v$. Hence $a= b$ and  $\pmb{\lambda} = \pmb{\mu}$. Therefore, we get that $Y' \cong Z'$ using the fact that $\mathscr{C}$ is Jordan recoverable. We finally conclude that $Y \cong Z$ and thus $\mathscr{D}$ is Jordan recoverable.
	\end{proof}
	We now show that we preserve canonical Jordan recoverability under the same assumption \ref{star}.
	\begin{prop}\label{addsimple2} Let $Q$ be a quiver of $A_n$ type, and $v$ be a source or a sink of $Q$. Let $\mathscr{C} \subset \rep(Q)$ be a canonically Jordan recoverable category satisfying \ref{star}. Then $\mathscr{D} = \Adds_v(\mathscr{C})$ is canonically Jordan recoverable.
	\end{prop}
	
	\begin{proof}[Proof] Consider $Y \in \mathscr{D} =\Adds_k(\mathscr{C})$. By definition of $\mathscr{D}$, there exist $a \in \mathbb{N}$ and $Y' \in \mathscr{C}$ such that $Y \cong S_v^a \oplus Y'$. Write $\pmb{\pi} = \GenJF(Y')$.  We get $\pmb{\xi} = \GenJF(Y)$ from $\pmb{\pi}$ as described in \cref{GenJFaddsimple}. 
		
		Note that from \ref{star}, we know that $\pmb{\pi}$ is strongly $(\boxminus, \boxminus)$-storable at $v$. So $\pi^v_1 = \max(\pi_1^{v-1}, \pi_1^{v+1}).$
		
		Let $Z \in \rep(Q,\pmb{\xi})$. Consider $N \in \NEnd(Z)$ such that $\JF(N) = \pmb{\xi}$. Assume, without loss of generality, that $v$ is a source. Denote $\alpha : v-1 \longleftarrow v$ and $\beta : v \longrightarrow v+1$ the arrows incident to $v$. From the relation $\pi^v_1 = \max(\pi_1^{v-1}, \pi_1^{v+1})$ and the definition of $\pmb{\xi}$, we get that $N^{\pi_1^v}(Z_v) \subseteq \Ker(Z_\alpha) \cap \Ker(Z_\beta)$. Thus, $\mult(S_v, Z) \geqslant a$. Saying that $\mult(S_v, Z) > a$ is equivalent to asking the induced morphism from the quotient $Z_v/N^{\pi_1^v}(Z_v)$ to $Z_\alpha \oplus Z_\beta$ to have a nontrivial kernel, which is a closed condition. As $Y \in \rep(Q,\pmb{\xi})$ and $\mult(S_v,Y) = a$, there exists therefore a dense open set $\Theta \subset \rep(Q,\pmb{\xi})$ such that, for any $Z \in \Theta$, $\mult(S_v,Z) = a$, which means there exists $Z' \in \rep(Q)$ such that $Z \cong S_v^a \oplus Z'$ and $Z'$ has no indecomposable summand isomorphic to $S_v$.
		
		Thanks to \cref{GenJFaddsimple}, we know that $\GenJF(Z') = \pmb{\pi}$. By canonical Jordan recoverability of $\mathscr{C}$, we know that there exists a dense open set $\Phi \subset \rep(Q,\pmb{\pi})$ such that for any $Z' \in \Phi$, $Z' \cong Y'$. Hence there exists a dense open set $\Omega \subset \rep(Q,\pmb{\xi})$ such that, for all $Z \in \Omega$, $Z \cong S_v^a \oplus Y' \cong Y$. We conclude that $\mathscr{D}$ is canonically Jordan recoverable.
	\end{proof}
	
	\subsection{Applying reflection functors}
	\label{ss:mutations}This subsection aims to show that under some general storability conditions, reflection functors preserve canonical Jordan recoverability.
	
	First, we prove the following result.
	\begin{prop}\label{mutationJR}Let $v$ be a vertex of an $A_n$ type quiver $Q$.  Let $\mathscr{C} \subseteq \rep(Q)$ be a Jordan recoverable category such that :
		\begin{enumerate}[label = $(\nabla)$] \item \label{nabla} For any $X \in \mathscr{C}$, $ \GenJF(X)
			$ is either $(\boxplus, \boxplus)$-storable, $(\boxminus, \boxplus)$-storable,\\ 
			$(\boxplus, \boxminus)$-storable or strongly $(\boxminus, \boxminus)$-storable at $v$.\end{enumerate} If $v$ is a source, then $\mathcal{R}_v^{-}(\mathscr{C})$ is a Jordan recoverable category of $\rep(\sigma_v(Q))$. Similarly, if $v$ is a sink, then $\mathcal{R}_v^+(\mathscr{C})$ is a Jordan recoverable category of $\rep(\sigma_v(Q))$.
	\end{prop}
	\begin{proof}[Proof] Assume that $v$ is a source. Let $Y , Z \in \mathcal{R}_v^{-}(\mathscr{C})$. By \ref{nabla}, we know that $S_v \notin \mathscr{C}$. Therefore there exists a unique representation $Y' \in \mathscr{C}$ (up to isomorphism) such that $\mathcal{R}_v^-(Y') \cong Y$. Similarly, there exists a unique $Z' \in \mathscr{C}$ such that  $\mathcal{R}_v^-(Z') = Z$. 
		
		Consider $\pmb{\lambda} = \GenJF(Y')$ et $\pmb{\mu} = \GenJF(Z')$. So $\sigma_v(\pmb{\lambda}) = \GenJF(Y)$ and $\sigma_v(\pmb{\mu}) = \GenJF(Z)$. Now assume that $\pmb{\theta} = \sigma_v(\pmb{\lambda}) = \sigma_v(\pmb{\mu})$. We claim that $Y \cong Z$. We know, by \ref{nabla} and \cref{lem:elementary diag store}, that $\pmb{\theta}$ is either $(\boxplus, \boxplus)$-storable, $(\boxminus, \boxplus)$-storable,  $(\boxplus, \boxminus)$-storable or strongly $(\boxminus, \boxminus)$-storable at $v$. As a consequence of \cref{lem:elementary diag store} \ref{ediag5}, we have $\pmb{\lambda} = \pmb{\mu}$. Using the fact that $\mathscr{C}$ is Jordan recoverable, we conclude that $Y' \cong Z'$. Thus, $Y \cong Z$, and we end up with the result we wished for.
		
		The same goes analogously for $\mathcal{R}_v^+(\mathscr{C})$, whenever $v$ is a sink. 
	\end{proof}
	Under the same assumption \ref{nabla}, we can show that reflection functors also preserve canonical Jordan recoverability.
	\begin{prop} 
		\label{mutationCJR}Let $v$ be a source of an $A_n$-type quiver $Q$.  Let $\mathscr{C} \subset \mathsf{rep}(Q)$ be a canonically Jordan recoverable category satisfying \ref{nabla}. Then $\mathcal{R}_v^{-}(\mathscr{C})$ is a canonically Jordan recoverable subcategory of $\rep(\sigma_v(Q))$. Similarly, if $v$ is a sink, $\mathcal{R}_v^+(\mathscr{C})$ is a canonically Jordan recoverable subcategory of $\rep(\sigma_v(Q))$.
	\end{prop}
	\begin{proof}[Proof] Assume that $v$ is a source. Let $Y \in \mathcal{R}_v^-(\mathscr{C})$. By \ref{nabla}, $S_v \notin \mathscr{C}$. So there exists a unique representation $Y' \in \mathscr{C}$ (up to isomorphism) such that $\mathcal{R}_v^-(Y') \cong Y$.
		
		Let $\pmb{\pi} = \GenJF(Y')$. As $\mathscr{C}$ is canonically Jordan recoverable, we have that $\GenRep(\pmb{\pi}) \cong Y'$. We can use \cref{thm:reflectionGenJF} to get that $\GenJF(Y) = \sigma_v(\pmb{\pi})$. Moreover,  \cref{thm:reflectionGenRep} gives us that $\GenRep(\sigma_v(\pmb{\pi})) \cong \mathcal{R}_v^-(Y') \cong Y$.  This completes the proof.
		
		The same goes similarly for $\mathcal{R}_v^+(\mathscr{C})$, whenever $v$ is a sink.
	\end{proof}
	
	\section{Adjacency-avoiding interval sets}
	\label{s:adj}
	
	Recall the following definition.
	\begin{definition}
		Two intervals $K,L \in \mathcal{I}_n$ are \new{adjacent} if either $b(K) = e(L)+1$ or $b(L) = e(K)+1$. An interval set $\mathscr{J}$ is said to be \new{adjacency-avoiding} if there are no pairs of adjacent intervals in $\mathscr{J}$; meaning, in a more affirmative way, that for all $K,L \in \mathscr{J}$, we have either $K \cap L \neq \varnothing$, $b(K) \geqslant e(L)+2$ or $b(L) \geqslant e(K)+2$.
	\end{definition}
	We saw in \cref{ss:IntroAdjavoid} that this notion is a crucial point to describe the canonically Jordan recoverable subcategories of $\rep(Q)$ for any $A_n$ type quiver $Q$. We aim to investigate the combinatorial behavior of this family of interval sets, with their representation-theoretic interest in mind.
	
	More precisely, we give an explicit description of all the (maximal) adjacency-avoiding subsets of intervals, and we explore their behavior under the combinatorial lookalike operators of the ones in \cref{s:operationsCJR}. In \cref{s:proof}, while applying an algorithm using those operators, it allows us to keep track of the category we are constructing, and to check that they correspond to the stated categories in \cref{maintheorem}.
	
	\subsection{Interval sets from shifted bipartitions}
	\label{ss:intsets}
	
	In this section, we describe and characterize all maximal adjacency-avoiding subsets of $\mathcal{I}_n$. 
	\begin{definition}
		Let $\B$ and $\E$ be two subsets of $\{1, \ldots,n\}$. We define $\mathscr{J}(\B,\E)$ to be the following subset of $\mathcal{I}_n$:
		$$\mathscr{J}(\B,\E) = \{K \in \mathcal{I}_n \mid b(K) \in \B \text{ and } e(K) \in \E \}.$$
	\end{definition}
	\begin{ex} Let $n \geqslant 1$. For $m \in \{1, \ldots n\}$ $(\B =\llrr{1,m},\E = \llrr{m,n})$, we get $$\mathscr{J}(\B,\E) = \{K \in \mathcal{I}_n \mid m \in K \}.$$
	\end{ex}
	Note that for a given pair of subsets $(\B,\E)$, there could exist $b \in \B$ such that for all $e \in E$, $e < b$: this implies that there are no intervals $K$ in $\mathscr{J}(\B,\E)$ such that $b(K) = b$. Hence $\mathscr{J}(\B,\E) = \mathscr{J}(\B \setminus \{b\}, \E)$. We consider the following notion, as we want to completely characterize these interval subsets by pairs of subsets $(\B,\E)$. We say that $(\B,\E)$ is a \new{effective pair of subsets} if it satisfies the two assertions below:
	\begin{enumerate}[label = $\bullet$]
		\item for all $b \in \B$ there exists $e \in \E$ such that $b \leqslant e$
		
		\item for all $e \in \E$, there exists $b \in \B$ such that $b \leqslant e$. 
	\end{enumerate}
	\begin{lemma}\label{lem:well-placedcaract} For all $(\mathbf{C}, \mathbf{F})$ pairs of subsets of $\{1, \ldots, n\}$, there exists a unique effective pair of subsets $(\B,\E)$ of $\{1, \ldots, n\}$ such that $\mathscr{J}(\mathbf{C}, \mathbf{F}) = \mathscr{J}(\B,\E)$.
	\end{lemma}
	For all $\A \subset \{1, \ldots, n\}$, we denote by $\A[1]$ \new{the shift of $\A$} defined by $\A[1] = \{ a + 1 \mid a \in \A \}$. 
	\begin{definition}\label{avoidingcandiddef}
		Let $(\B,\E)$ be a pair of subsets of $\{1, \ldots, n\}$. We say that $(\B,\E)$ is a \new{shifted disjoint pair} if $\B \cap \E[1] = \varnothing$. 
	\end{definition}
	\begin{prop}\label{prop:avoid} Let $(\B,\E)$ be an effective shifted disjoint pair of subsets in $\{1,\ldots,n\}$. Then $\mathscr{J}(\B,\E)$ is adjacency-avoiding.
	\end{prop}
	\begin{proof}[Proof]
		Let $K,L \in \mathscr{J}(\B, \E)$. We want to prove that they are not adjacent. If they intersect, we are done. Otherwise, without loss of generality, assume that $b(L) > e(K)$. By definition, $b(L)  \in \B$ and $e(K) \in \E$. We know that $\B\cap\E[1] = \varnothing$. This means that $e(K) + 1 \notin \B$ and so $b(L) \geqslant e(K)+2$. Therefore, $K$ and $L$ are not adjacent.
	\end{proof}
	\begin{definition}\label{candidatedef}
		A pair of subsets $(\B,\E)$ of $\{1, \ldots, n\}$ is a \new{shifted interval bipartition} if $\B \cup \E[1] \in \mathcal{I}_{n+1} \cup \{\varnothing\}$, and $\B \cap \E[1] = \varnothing$. Moreover, such a pair is said to be \new{complete} if $\B \cup \E[1] = \{1, \ldots, n+1\}$.
	\end{definition}
	\begin{remark} \label{rem:completrsib} Some remarks:
		\begin{enumerate}[label = $\bullet$]
			\item Note that if $(\B,\E)$ is a pair of subsets of $\{1, \ldots, n\}$ such that $\B \cup \E[1] = \{1, \ldots, n+1\}$, then $(\B,\E)$ is effective since necessarily $1 \in \B$ and $n \in \E$.

			\item Define $\rev$ the \new{reverse map} on $\{1, \ldots, n\}$, by $\rev(i) =  n+1-i$. Write $\A^{\rev} = \rev(\A)$ for all $\A \subseteq \{1, \ldots, n\}$. Now if $(\B,\E)$ is an effective shifted interval bipartition, then $(\E^{\rev},\B^{\rev})$ is too.
		\end{enumerate}
	\end{remark}
	We will show that the complete shifted bipartitions describe all maximal (with respect to inclusion) adjacency-avoiding interval sets.
	\begin{lemma}\label{lem:crsbipint} Let $(\B,\E)$ be a complete shifted bipartition of $\{1, \ldots, n\}$. Then for all $K \in \mathcal{I}_n \setminus \mathscr{J}(\B,\E)$, either $b(K)-1 \in \E$ or $e(K)+1 \in \B$.
	\end{lemma}
	\begin{proof}[Proof]
		Let $K \in \mathcal{I}_n \setminus \mathscr{J}(\B,\E)$. We have either $b(K) \notin \B$ or $e(K) \notin \E$. We know that $\B \cup \E[1] = \{1, \ldots, n+1\}$ and $\B \cap \E[1] = \varnothing$. If $b(K) \notin \B$ then $b(K) \in \E[1]$ and therefore $b(K)-1 \in \E$. Otherwise $e(K) \notin \E$ and then $e(K)+1 \in \B$.
	\end{proof}
	
	\begin{prop}\label{prop:maxavoid} Let $(\B,\E)$ be a complete shifted  bipartition of $\{1, \ldots, n\}$. Then $\mathscr{J}(\B,\E)$ is a maximal (for inclusion) adjacency-avoiding subset of $\mathcal{I}_n$.
	\end{prop}
	\begin{proof}[Proof]
		By \cref{prop:avoid}, we already know that $\mathscr{J} = \mathscr{J}(\B, \E)$ is adjacency-avoiding.  
		
		Let $K \in \mathcal{I}_n \setminus \mathscr{J}$.  Then either $b(K)-1 \in \E$ or $e(K)+1 \in \B$ by \cref{lem:crsbipint}. In the first case, by taking $\llrr{1,b(K)-1}$ or, in the second case, by taking $\llrr{e(K)+1,n}$, we conclude that $\mathscr{J} \cup \{K\}$ is not adjacency-avoiding.
	\end{proof}
	
	\begin{lemma} \label{lem:maxadja}
		Consider $\mathscr{J}$ a maximal (for inclusion) adjacency-avoiding subset of intervals of $\mathcal{I}_n$. Then:
		\begin{enumerate}[label = $(\roman*)$]
			\item \label{maxadja1n}  $\llbracket 1; n \rrbracket \in \mathscr{J}$;
			
			\item \label{maxadjacup}  if $K,L \in \mathscr{J}$ with $b(K) \leqslant e(L)$, then $\llbracket b(K), e(L) \rrbracket \in \mathscr{J}$
			
			\item \label{maxadjacap} if $K,L\in \mathscr{J}$ and if $K \cap L \neq \varnothing$, then $K \cap L \in \mathscr{J}$
			
			\item \label{maxadjasimp}  if $K \in \mathscr{J}$, then there exists $m \in K$ such that $\llbracket m \rrbracket \in \mathscr{J}$.
		\end{enumerate}
	\end{lemma}
	\begin{proof}[Proof] Let $\mathscr{J}$ be as assumed.
		\begin{enumerate}[label = $(\roman*)$]
			\item No interval $K \in \mathcal{I}_n$ is  adjacent to $\llbracket 1;n \rrbracket$. Hence $\mathscr{J} \cup \{ \llbracket 1;n \rrbracket\}$ is adjacency-avoiding. By maximality of $\mathscr{J}$, $\llbracket 1;n \rrbracket \in \mathscr{J}$.
			
			\item Let $K,L \in \mathscr{J}$ such that $b(K) \leqslant e(L)$. There is no interval $T \in \mathscr{J}$ adjacent to $\llbracket b(K), e(L) \rrbracket$; otherwise, such a $T$ would have to be adjacent to either $K$ or $L$. Hence $\mathscr{J} \cup \{ \llbracket b(K), e(L) \rrbracket \}$ is adjacency-avoiding. By maximality of $\mathscr{J}$, we have $ \llbracket b(K), e(L) \rrbracket \in \mathscr{I}$.
			
			\item Let $K,L \in \mathscr{J}$ such that $K \cap L \neq \varnothing$. Without loss of generality, we may assume that  $b(L) \leqslant e(K)$. Therefore $b(L) \leqslant b(K) \leqslant e(L) \leqslant e(K)$ and $K \cap L = \llrr{b(K),e(L)}$. By \ref{maxadjacup}, we conclude that $K \cap L \in \mathscr{J}$.
			
			\item Let $K \in \mathscr{J}$. Let us consider a minimal interval $T \subseteq K$ such that $T \in \mathscr{J}$
			
			Let $U \in \mathcal{I}_n$ such that $\varnothing \neq U \subseteq T$. If there exists $L \in \mathscr{J}$ such that $L$ is adjacent to $U$, then either $L$ is adjacent to $T$, which is impossible by hypothesis on $\mathscr{J}$, or $L \cap T \neq \varnothing$. Using \ref{maxadjacap}, we get that $T \supseteq L \cap T \in \mathscr{J}$. By minimality of $T$ in $\mathscr{J}$, we assert that $L \cap T = T$ and thus $L \cap U = U \neq \varnothing$, contradicting the fact that $L$ and $U$ are adjacent. Thus, for all $L \in \mathscr{J}$, $L$ and $U$ are not adjacent.
			
			We obtain that $ \mathscr{J} \cup \{U\}$ is adjacency-avoiding. By maximality of $\mathscr{J}$, we get that $T=U$. The only case where any nonempty subset of $T$ is equal to $U$ is when $b(T) = e(T)$. We deduce the desired result. \qedhere
		\end{enumerate}
	\end{proof}
	\begin{prop} \label{themaxavoid} All the maximal adjacency-avoiding subsets of $\mathcal{I}_n$ can be written as $\mathscr{J}(\B, \E)$ where $(\B,\E)$ is a complete shifted bipartition of $\{1, \ldots, n\}$.
	\end{prop}
	\begin{proof}[Proof]
		Let $\mathscr{J}$ be a maximal adjacency-avoiding subset of $\mathcal{I}_n$. By \cref{lem:maxadja} \ref{maxadja1n}, we know that $\llbracket 1; n \rrbracket \in \mathscr{J}$. Applying \cref{lem:maxadja} \ref{maxadjasimp} to $\llbracket 1; n \rrbracket$, we get that there exists $m \in \llbracket 1;n \rrbracket$ such that $\llbracket m \rrbracket  \in \mathscr{J}$. 
		Knowing that there is at least one $m \in \{1, \ldots, n \}$ such that $\llrr m \in \mathscr{J}$, assume that there are $p\in \mathbb{N}^*$ of those. We order and denote them by $1 \leqslant m_1 < m_2 < \ldots < m_p \leqslant n$. Note obviously that we cannot have $m_{s+1} = m_s + 1$ for any $s \in \{1, \ldots, p-1\}$.
		
		For $s \in \{1, \ldots, p-1\}$, let $a_s$ be the maximal index $a$ such that $m_s \leqslant a < m_{s+1}$ and such that $\llbracket m_s;a \rrbracket \in \mathscr{J}$, and $b_s$ be the minimal index $b$ such that $m_s < b \leqslant m_{s+1}$ and such that $\llbracket b;m_{s+1}\rrbracket \in \mathscr{J}$. 
		
		We show that, for all $s \in \{1, \ldots,p-1\}$, $b_s \geqslant a_s + 2$. By contradiction: \begin{enumerate}[label = $\bullet$]
			\item if $b_s = a_s+1$ then $\mathscr{J}$ is not adjacency-avoiding;
			
			\item if $b_s \leqslant a_s$, then $\llbracket b_s,a_s\rrbracket \in \mathscr{J}$ by \cref{lem:maxadja} \ref{maxadjacap}, and this implies by \cref{lem:maxadja} \ref{maxadjasimp} that we should have a $m' \in \llbracket b_s, a_s\rrbracket$ such that $\llbracket m' \rrbracket \in \mathscr{J}$. However, by construction, $m_s < m' < m_{s+1}$.
		\end{enumerate} 
		By taking 
		\begin{enumerate}[label = $\bullet$]
			\item $\B = \llrr{1,m_1} \cup \llrr{a_1+2, m_2} \cup \ldots \cup \llrr{a_{p-1}+2,m_{p}}$ and
			
			\item $\E = \llrr{m_1, a_1} \cup \llrr{m_2,a_2} \cup \ldots \cup \llrr{m_{p},n}$,
		\end{enumerate} we can easily check that $(\B,\E)$ is a complete shifted bipartition of $\{1, \ldots, n\}$. Moreover, by construction of $\mathscr{J}(\B,\E)$ and \cref{lem:maxadja} \ref{maxadjacup}, we can assert that $\mathscr{J}(\B,\E) \supseteq \mathscr{J}$. Thus $\mathscr{J}(\B,\E) = \mathscr{J}$ by \cref{prop:maxavoid}.
	\end{proof}
	
	\subsection{Interval reflections}
	\label{ss:reflk}
	
	\begin{definition}\label{def:intervrefl}
		Let $v \in \{ 1, \ldots, n \}$. The \new{interval reflection at $v$}, denoted $\refl_v$ is a function on $\mathcal{I}_n \setminus \{\llrr v \}$ defined as follows:
		$$\forall K \in \mathcal{I}_n \setminus \{\llrr v \},\ \refl_v(K) = \begin{cases}
			K \cup \{v\} & \text{if } v \notin K \text{ and } K \cup \{v\} \in \mathcal{I}_n \\
			K \setminus \{v\} & \text{if } v \in K \text{ and } K \setminus \{v\} \in \mathcal{I}_n \\
			K & \text{otherwise.}
		\end{cases}$$
	\end{definition}
	\begin{remark} \label{rem:reflkint} The interval reflection at $v$ is an involution on $\mathcal{I}_n \setminus \{\llbracket v \rrbracket\}$. 
	\end{remark}
	For all $\mathscr{J} \subseteq \mathcal{I}_n$, we denote $\refl_v(\mathscr{J})$ the interval subset made of all the intervals $\refl_v(K)$ for $K \in \mathscr{J} \setminus \{\llbracket v \rrbracket\}$. Here is a direct consequence of the definition of $\refl_v$ and \cref{prop:reflonAntypequivers}.
	\begin{cor} \label{catrefl} Let $Q$ be an $A_n$ type quiver and $v \in Q_0$ be either a source or a sink of $Q$. Consider $\mathscr{J} \subset \mathcal{I}_n$. Then the reflection functor at $v$ applied to $\Cat_Q(\mathscr{J})$ yields $\Cat_{\sigma_v(Q)}(\refl_v(\mathscr{J}))$.
	\end{cor}
	The result below shows that the adjacency-avoiding property is stable under $\refl_v$ for all $v \in \{1, \ldots, n\}$.
	\begin{prop}\label{prop:refladja} Let $\mathscr{J}$ be an adjacency-avoiding subset of $\mathcal{I}_n$. Then $\refl_v(\mathscr{J})$ is adjacency-avoiding. Moreover, if $\llbracket v \rrbracket \notin \mathscr{J}$, then $\mathscr{J}$ is adjacency-avoiding if and only if $\refl_v(\mathscr{J})$ is too.
	\end{prop}
	\begin{proof}[Proof]
		Let $\mathscr{J}$ be an adjacency-avoiding interval subset. Suppose that $\refl_v(\mathscr{J})$ is not adjacency-avoiding. We thus have two adjacent intervals $T$ and $U$ in $\refl_v(\mathscr{J})$. Let us say that $b(U) = e(T) + 1$ without loss of generality. 
		
		By definition, let $K,L \in \mathscr{J} \setminus \{\llbracket v \rrbracket\}$ such that $\refl_v(K) = T$ and $\refl_v(L) = U$. By involution, we get $K = \refl_v(T)$ and $L = \refl_v(U)$. Now,
		\begin{enumerate}[label = $\bullet$]
			\item if $v \notin \{e(T), b(U)\}$, then $b(L) = b(U) = e(T)+1 = e(K)+1$;
			
			\item if $v = e(T)$, then $b(L) = b(U) - 1 = e(T) = e(K)+1$;
			
			\item if $v = b(U)$, then $b(L) = b(U)+1 = e(T)+2 = e(K)+1$.
		\end{enumerate}
		In all cases, we find that $K$ and $L$ are adjacent, which is a contradiction. So $\refl_v(\mathscr{J})$ is adjacency-avoiding.
		
		If $\llbracket v \rrbracket \notin \mathscr{J}$, then $\refl_v(\refl_v(\mathscr{J})) = \mathscr{J}$ and we are done.
	\end{proof}
	It seems natural to ask how the maximal adjacency-avoiding interval subsets behave under this action. The example below must motivate us to define an action on effective shifted interval bipartitions by showing that we can describe the image under $\refl_v$ in terms of another effective shifted interval bipartition. 
	\begin{ex}\label{ex:motivate} Let $n = 6$,
		$\B = \{2,4\}$ and $\E = \{2,4,5\}$. We have $\mathscr{J} = \mathscr{J}(\B,\E) = \{\llrr{2}, \llrr{2,4}, \llrr{2,5}, \llrr{4}, \llrr{4,5}\}$. Then we get
		$$\refl_3(\mathscr{J}) = \{\llrr{2,3}, \llrr{2,4}, \llrr{2,5}, \llrr{3,4}, \llrr{3,5}\} = \mathscr{J}(\B', \E') \setminus \{\llrr{3}\}$$ with $\B' = \{2,3\} = (\B \cup \{3\}) \setminus \{4\}$ and $\E' = \{3,4,5\} = (\E \cup \{3\}) \setminus \{2\}$.
	\end{ex}
	Before defining the \emph{toggle action}, let us introduce the \emph{completion of an effective shifted interval bipartition} via \emph{extended shifted bipartitions}.
	\begin{definition} \label{def:extendedshtbipart}
		An \new{extended shifted bipartition} of $\{1, \ldots, n\}$ is a pair $(\C,\F)$ such that $\C \subseteq \{1, \ldots, n+1\}$, $\F \subseteq \{0, \ldots, n\}$ and $\{\C, \F[1]\}$ is a bipartition of $\{1, \ldots, n+1\}$.
	\end{definition}
	\begin{prop} \label{propdef:completion} For any pair $(\B,\E)$, with $\B \neq \varnothing \neq \E$, forming an interval shifted bipartition of $\{1, \ldots, n\}$, there exists a unique extended shifted bipartition $(\overline{\B}, \overline{\E})$ of $\{1, \ldots, n\}$ such that $\mathscr{J}(\B,\E) = \mathscr{J}(\overline{\B}, \overline{\E} )$.
		
		We call $(\overline{\B}, \overline{\E})$ the \new{completion of $(\B, \E)$}.
	\end{prop}
	\begin{remark} \label{rem:emptycompletion} Note that if either $\B = \varnothing$ and $n \notin \E$, or $\E = \varnothing$ and $1 \notin \B$, then the completion is not unique. For instance, $(\varnothing, \varnothing)$ admits $n+2$ different completions. As this pair will remain important throughout this section, we introduce notation for its completions. Let, for $0 \leqslant m \leqslant n+1$, $\overline{\varnothing}_m = (\{m+1, \ldots, n+1\},\{0,\ldots, m-1\})$ be the \new{$m$th completion of $(\varnothing, \varnothing)$}.
	\end{remark}
	\begin{proof}[Proof of \cref{propdef:completion}]
		Let $(\B,\E)$ as assumed. Then let:
		\begin{enumerate}[label = $\bullet$]
			\item $\overline{\B} = \B \cup \{i \mid i > \max(\B \cup \E[1])\}$;
			
			\item $\overline{\E} = \E \cup \{i-1 \mid i < \min(\B \cup \E[1])\}$.
		\end{enumerate}
		It is easy to check that $(\overline{\B}, \overline{\E})$ is an extended shifted bipartition of $\{1, \ldots, n\}$ such that $\mathscr{J}(\B,\E) = \mathscr{J}(\overline{\B}, \overline{\E})$. It is also clear that this extended shifted bipartition is the unique one satisfying the desired properties.
	\end{proof}
	\begin{cor}\label{cor:bijeffshintbi}
		The map $(\B,\E) \longmapsto (\overline{\B}, \overline{\E})$ gives a bijection from pairs $(\B,\E)$ of nonempty subsets of $\{1, \ldots, n\}$ forming effective shifted interval bipartitions, and pairs $(\mathbf{C},\mathbf{F})$ forming extended shifted bipartitions of $\{1, \ldots, n\}$ such that $\mathscr{J}(\mathbf{C}, \mathbf{F}) \neq \varnothing$.
	\end{cor}
	\begin{proof}[Proof]
		This result is a direct consequence of \cref{lem:well-placedcaract} and \cref{propdef:completion}.
	\end{proof} 
	For any pair $(\mathbf{C}, \mathbf{F})$ such that $\mathbf{C} \subseteq \{1, \ldots, n+1\}$ and $\mathbf{F} \subseteq \{0, \ldots, n\}$, denote $\eff(\mathbf{C}, \mathbf{F})$ the unique effective shifted interval bipartition of $\{1, \ldots, n\}$ such that $\mathscr{J}(\mathbf{C}, \mathbf{F}) = \mathscr{J}(\eff(\mathbf{C}, \mathbf{F}))$. 
	
	In the next example, we extend our observations from \cref{ex:motivate} to all other reflections we can perform. We will see in \cref{prop:maxadjareflnicestext} that this is exactly the case.
	\begin{ex}
		\label{ex:motivate2} Let $n = 6$,
		$\B = \{2,4\}$ and $\E = \{2,4,5\}$. So $\overline{\B} = \{2,4,7\}$ and
		$\overline{\E} = \{0,2,4,5\}$.  Let $\mathscr{J} = \mathscr{J}(\B,\E) = \{\llrr{2}, \llrr{2,4}, \llrr{2,5}, \llrr{4}, \llrr{4,5}\}$.  Then,
		\begin{enumerate}[label = $\bullet$,itemsep=1mm]
			\item $\refl_1(\mathscr{J}) =  \{\llrr{1,2}, \llrr{1,4}, \llrr{1,5}, \llrr{4}, \llrr{4,5}\} = \mathscr{J}(\B', \E') \setminus \{\llrr{1}\}$ with:
			\begin{enumerate}[label = $\bullet$,itemsep=0.5mm]
				\item $\overline{\B'} = \{1,4,7\} = (\overline{\B} \cup \{1\})\setminus \{2\}$ 
				
				\item $\overline{\E'} = \{1,2,4,5\} = (\overline{\E} \cup \{1\}) \setminus \{0\}$;  
			\end{enumerate}
			\item $\refl_2(\mathscr{J}) = \{\llrr{3,4}, \llrr{3,5}, \llrr{4}, \llrr{4,5} = \mathscr{J}(\B', \E')$ with:
			\begin{enumerate}[label = $\bullet$,itemsep=0.5mm]
				
				\item $\overline{\B'} = \{3,4,7\} = (\overline{\B} \cup \{3\}) \setminus \{2\}$
				
				\item $\overline{\E'} = \{0,1,4,5\} = (\overline{\E} \cup \{1\}) \setminus \{2\}$;
			\end{enumerate}
			\item $\refl_3(\mathscr{J}) = \{\llrr{2,3}, \llrr{2,4}, \llrr{2,5}, \llrr{3,4}, \llrr{3,5}\} = \mathscr{J}(\B', \E') \setminus \{\llrr{3}\}$ with:
			\begin{enumerate}[label = $\bullet$,itemsep=0.5mm]
				\item $\overline{\B'}= \{2,3,7\} = (\overline{\B} \cup \{3\}) \setminus \{4\}$ 
				\item $\overline{\E'} = \{0,3,4,5\} = (\overline{\E} \cup \{3\}) \setminus \{2\}$. \qedhere
			\end{enumerate}
		\end{enumerate}
	\end{ex}
	\begin{definition} \label{def:toggleext}
		Let $(\mathbf{C},\mathbf{F})$ be an extended shifted bipartition of $\{1,\ldots,n\}$. Let $v \in \{1, \ldots, n\}$. We define \new{$\tog_v(\mathbf{C},\mathbf{F})$ the toggle at $v$ of $(\mathbf{C},\mathbf{F})$} as the pair $(\mathbf{C}', \mathbf{F}')$ where:
		\begin{enumerate}[label = $\bullet$]
			\item $\displaystyle \mathbf{C}' = \begin{cases}
				(\mathbf{C} \cup \{v\})\setminus \{v+1\} & \text{if } v \notin \mathbf{C} \cup \mathbf{F} \\
				(\mathbf{C} \cup \{v+1\}) \setminus \{v\} & \text{if } v \in \mathbf{C} \cap \mathbf{F} \\
				\mathbf{C} & \text{otherwise};
			\end{cases}$
			
			\item $\displaystyle \mathbf{F}' = \begin{cases}
				(\mathbf{F} \cup \{v\})\setminus \{v-1\} & \text{if } v \notin \mathbf{C} \cup \mathbf{F} \\
				(\mathbf{F} \cup \{v-1\}) \setminus \{v\} & \text{if } v \in \mathbf{C} \cap \mathbf{F} \\
				\mathbf{F} & \text{otherwise}.
			\end{cases}$
		\end{enumerate}
	\end{definition}
	\begin{remark} \label{rem:rev} Note that, by construction, $\rev(\tog_v(\C,\F)) = \tog_{n+1-v}(\F^{\rev}, \C^{\rev})$.
	\end{remark}
	The following interpretation of the application of $\tog_v$ follows directly from \cref{def:toggleext}.
	\begin{lemma} \label{lem:exchangev} Let $(\C,\F)$ be an extended shifted bipartition of $\{1, \ldots, n\}$. Consider $(\C', \F') = \tog_v(\C,\F)$. Then $\C'$ is defined from $\C$ by replacing $v$ by $v-1$ and by replacing $v-1$ by $v$, and $\F'$ is defined from $\F$ by replacing $v$ by $v+1$ and by replacing $v+1$ by $v$.
	\end{lemma}
	The following proposition follows immediately from \cref{lem:exchangev}.
	\begin{prop}
		Let $(\mathbf{C},\mathbf{F})$ be an extended shifted bipartition of $\{1, \ldots, n\}$. For all $v \in \{1, \ldots, n\}$, $\tog_v(\mathbf{C}, \mathbf{F})$ is also an extended shifted bipartition of $\{1, \ldots, n\}$.
	\end{prop}
	Thus, the toggles on the extended shifted bipartitions of $\{1,\ldots,n\}$ induce toggle operations on the effective shifted interval bipartitions.
	\begin{definition} \label{def:toggleeff}
		Let $(\B,\E)$ be an effective shifted interval bipartition of the set 
		$\{1, \ldots, n\}$, and $v \in \{1, \ldots, v\}$. We define \new{$\tog_v(\B,\E)$ the toggle at $v$ of $(\B, \E)$} by:
		$$\tog_v(\B,\E) = \begin{cases}
			\eff(\tog_v(\overline{\varnothing}^n_v)) & \text{if } (\B,\E) = (\varnothing, \varnothing);\\
			\eff(\tog_v(\compl_n(\B,\E))) & \text{otherwise.}
		\end{cases}$$
	\end{definition}
	\begin{remark} \label{rem:explicitog} We can describe $\tog_v$ explicitly on effective shifted interval bipartitions of $\{1, \ldots, n\}$:  we have $\tog_v(\B,\E) = (\B', \E')$ where:
		\begin{enumerate}[label = $\bullet$]
			\item $\displaystyle \B' = \begin{cases}
				(\B \cup \{v\}) \setminus \{v+1\} & \text{if } v \notin \B\cup\E \text{ and } v+1 \in \B \\
				\B \cup \{v\} &  \begin{matrix} \text{if } v \notin \B\cup\E,\ v+1 \notin \B  \text{ and either } v-1 \in \E \\
					\text{or } \B = \varnothing \hfill \end{matrix} \\
				(\B \cup \{v + 1\}) \setminus \{v\} & \begin{matrix}\text{if } v \in \B \cap \E \text{ and either } v+1 \in \E \\ \text{or } v+2 \in \B \hfill \end{matrix}\\
				\B  \setminus \{v\} & \text{if } v \in \B \cap \E,\ v+1 \notin \E \text{ and }  v+2 \notin \B \\
				\B & \text{otherwise;}
			\end{cases}$
			
			\item $\displaystyle \E' = \begin{cases}
				(\E \cup \{v\}) \setminus \{v-1\} & \text{if } v \notin \B\cup\E  \text{ and } v-1 \in \E \\
				\E \cup \{v\} & \begin{matrix}\text{if } v \notin \B\cup\E,\ v-1 \notin \E  \text{ and either } v+1 \in \B 
					\\
					\text{or } \E = \varnothing  \hfill \end{matrix} \\
				(\E \cup \{ v-1\}) \setminus \{v\} & \begin{matrix}\text{if } v \in \B \cap \E \text{ and either } v-1 \in \B \\
					\text{or } v-2 \in \E \hfill \end{matrix} \\
				\E \setminus \{v\} & \text{if } v \in \B \cap \E,\ v-1 \notin \B \text{ and }  v-2 \notin \E\\
				\E & \text{otherwise.}
			\end{cases}$
		\end{enumerate}
		However, \cref{def:toggleeff} is easier to handle than this explicit description. 
	\end{remark}
	\begin{prop}\label{prop:maxadjareflnicestext}
		Let $(\C,\F)$ be an extended shifted bipartition of $\{1, \ldots, n\}$, and $v \in \{1, \ldots, n\}$. Write $\mathscr{J} = \mathscr{J}(\C,\F)$ and $\mathscr{T} = \mathscr{J}(\tog_v(\C,\F))$. We get the following results:
		\begin{enumerate}[label = \arabic*)]
			\item if $v \notin \C \cup \F$, then $\mathscr{T} = \refl_v(\mathscr{J}) \cup \{\llrr v\}$.
			
			\item otherwise, $\mathscr{T} = \refl_v(\mathscr{J})$.
		\end{enumerate}
	\end{prop}
	\begin{proof}[Proof]
		Let $(\C,\F)$ and $v$ be as assumed. Write $(\C',\F') = \tog_v(\C,\F)$. It is obvious that $\llrr v \in \mathscr{J}(\C',\F')$ if and only if $\C \cap \{v,v+1\} = \{v+1\}$ and $\F \cap \{v,v-1\} = \{v-1\}$, which is equivalent to $v \notin \C \cup \F$ as claimed.
		
		Let $K' \in \mathcal{I}_n$. By definition, $K' \in \mathscr{J}(\C', \F')$ whenever $b(K') \in \C'$ and $e(K') \in \F'$. As we already treated this case, assume that $K' \neq \llrr v$. By \cref{lem:exchangev}, the conditions $b(K') \in \C'$ and $e(K') \in \F'$ is equivalent to saying that $K \in \mathscr{J}(\C,\F)$ and $K' = \refl_v(K)$. This completes the proof.
	\end{proof}
	\begin{cor}\label{prop:maxadjareflnicest}
		Let $(\B,\E)$ be an effective shifted interval partition of $\{1, \ldots, n\}$ and $v \in \{1, \ldots, n\}$. Write $\mathscr{J} = \mathscr{J}(\B,\E)$ and $\mathscr{T} = \mathscr{J}(\tog_v(\B,\E))$. We get the following results:
		\begin{enumerate}[label = \arabic*)]
			
			\item if $v \notin \B \cup \E$, and, $v-1 \in \E$ or $v+1 \in \B$, then $\mathscr{T} = \refl_v(\mathscr{J}) \cup \{\llrr v\}$.
			
			\item otherwise, $\mathscr{T} = \refl_v(\mathscr{J})$.
		\end{enumerate}
	\end{cor}
	\begin{proof}[Proof]
		This results from \cref{cor:bijeffshintbi}, \cref{def:toggleeff}, and \cref{prop:maxadjareflnicestext}.
	\end{proof}
	\begin{remark}
		Note that the condition $v-1 \in \E$ or $v+1 \in \B$ makes sure that there exists an interval in $\mathscr{J}(\B,\E)$ adjacent to $\llrr{v}$, and so $\mathscr{J}(\B,\E)$ will not be fixed under $\refl_v$.
	\end{remark}
	\begin{cor} \label{cor:antrefl} For any effective shifted interval bipartitions $(\B, \E)$, there exists a unique effective shifted interval bipartition $(\B', \E')$ such that either $\mathscr{J}(\B, \E) = \refl_v(\mathscr{J}(\B',\E'))$ if $v \notin \B \cap \E$ or $\mathscr{J}(\B, \E) = \refl_v(\mathscr{J}(\B',\E')) \cup \{\llbracket v \rrbracket\}$ otherwise.
	\end{cor}
	\begin{proof}[Proof]
		We use the fact that $\refl_v$ is an involution on interval subsets which do not contain $\llrr{v}$ and \cref{prop:maxadjareflnicest}. Therefore $(\B', \E') = \tog_v(\B,\E)$.
	\end{proof}
	\begin{remark} \label{rem:samestatext} A similar statement can be made for extended shifted bipartitions of $\{1, \ldots, n\}$. In what follows, we will focus only on the effective shifted-interval bipartitions.
	\end{remark}
	
	\section{Proof of the main result}
	\label{s:proof}
	In this section, we prove \cref{maintheorem}. To do so, we will prove that for any complete shifted bipartition $(\B,\E)$ of $\{1, \ldots,n\}$, and any quiver $Q$ of $A_n$ type, $\Cat_{Q}(\mathscr{J}(\B,\E))$ is canonically Jordan recoverable. We first prove this claim for $Q$ linearly oriented and then generalize it to any $A_n$ type quiver.
	
	\subsection{The linearly oriented case}
	\label{ss:linearcase}
	We introduce an algorithm as a sequence of operations seen in \cref{s:operationsCJR} which builds the category
	$\mathscr{C}_{\overrightarrow{A_n}}(\B,\E) = \Cat_{\overrightarrow{A_n}}(\mathscr{J}(\B,\E)))$ for any given complete shifted bipartition $(\B,\E)$ of $\{1, \ldots, n\}$. Thanks to it, we will prove that this category is canonically Jordan recoverable.
	\begin{algo} \label{algo:catCJR}
		Let $n \geqslant 1$ and $Q = \overrightarrow{A_n}$. 
		\begin{enumerate}[label=$(\arabic*)$,itemsep=1mm]
			\setcounter{enumi}{-1}
			\item \textbf{Input:} a complete shifted bipartition $(\B,\E)$ of $\{1,\ldots,n\}$.
			\item Set 
			\begin{enumerate}[label=$\bullet$,itemsep=1mm]
				\item $Q_{1,0} = \overrightarrow{A_n}$;
				\item $\mathscr{C}^{1,0} = \{0\} \subseteq \rep(Q_{1,0})$;
				\item $\mathscr{D}^{1,0} = \{0\} \subseteq \rep(Q_{1,0})$.
			\end{enumerate}

			\item For $i \in \{1,\ldots,n+1\}$:
			\begin{enumerate}[label=$(2\alph*)$,itemsep=1mm]
				\item For $j \in \{0,\ldots,n-i\}$, we put
				\begin{enumerate}[label=$\bullet$,itemsep=1mm]
					\item $Q_{i,j+1} = \sigma_{n-j}(Q_{i,j})$;
					\item $\mathscr{D}^{i,j+1} = \mathcal{R}_{n-j}^-(\mathscr{C}^{i,j})$; and,
					\item  $\displaystyle \mathscr{C}^{i,j+1} = \begin{cases}
						\Adds_{n-j}(\mathscr{D}^{i,j+1}) & \text{if } i \in \B  \text{ and } i+j\in \E   \\
						\mathscr{D}^{i,j+1} &  \text{otherwise.}
					\end{cases}$
				\end{enumerate}
				
				\item Set 
				\begin{enumerate}[label=$\bullet$,itemsep=1mm]
					\item $Q_{i+1,0} = Q_{i,n-i+1}$; and,					
					\item $\mathscr{C}^{i+1,0} = \mathscr{C}^{i,n-i+1}$.
				\end{enumerate}
			\end{enumerate}
			\item \textbf{Output:} $Q_{n+1,0}$ and $\mathscr{C}^{n+1,0} \subseteq \rep(Q_{n+1,0})$.
		\end{enumerate}
		In the following, we set $\mathscr{C}^{i,j} = \mathscr{C}_{Q_{i,j}}^{i,j}(\B,\E)$ and $\mathscr{D}^{i,j} = \mathscr{D}_{Q_{i,j}}^{i,j}(\B,\E)$ for $1 \leqslant i \leqslant n+1$ and $0 \leqslant j \leqslant n-i+1$.
	\end{algo}
	
	
		
		We can rephrase \cref{algo:catCJR} by giving a similar algorithm for interval sets corresponding to $\mathscr{C}^{i,j}$ and $\mathscr{D}^{i,j}$. 
		
		\begin{algo} \label{algo:intCJR}
			Let $n \geqslant 1$ and $Q = \overrightarrow{A_n}$. 
			\begin{enumerate}[label=$(\arabic*)$,itemsep=1mm]
				\setcounter{enumi}{-1}
				\item \textbf{Input:} a complete shifted bipartition $(\B,\E)$ of $\{1,\ldots,n\}$.
				\item Set 
				\begin{enumerate}[label=$\bullet$,itemsep=1mm]
					\item $\mathscr{J}^{1,0} = \varnothing$; and $\mathscr{T}^{1,0} = \varnothing$.
				\end{enumerate}

				\item For $i \in \{1,\ldots,n+1\}$:
				\begin{enumerate}[label=$(2\alph*)$,itemsep=1mm]
					\item For $j \in \{0,\ldots,n-i\}$, we put
					\begin{enumerate}[label=$\bullet$,itemsep=1mm]
						\item $\mathscr{T}^{i,j+1} = \refl_{n-j}(\mathscr{J}^{i,j})$; and,
						\item $\displaystyle \mathscr{J}^{i,j+1} = \begin{cases}
							\mathscr{T}^{i,j+1} \cup  \{\llbracket n-j \rrbracket\} & \text{if } i \in \B \text{ and } i+j \in \E   \\
							\mathscr{T}^{i,j+1} &  \text{otherwise.}
						\end{cases}$
					\end{enumerate}
					
					\item Set $\mathscr{J}^{i+1,0} = \mathscr{J}^{i,n-i+1}$.
				\end{enumerate}
				\item \textbf{Output:} $\mathscr{J}^{n+1,0} \subseteq \mathcal{I}_n$.
			\end{enumerate}
			In the following, we set $\mathscr{J}^{i,j} = \mathscr{J}^{i,j}(\B,\E)$ and $\mathscr{T}^{i,j} = \mathscr{T}^{i,j}(\B,\E)$ for $1 \leqslant i \leqslant n+1$ and $0 \leqslant j \leqslant n-i+1$.
		\end{algo}
		
		
			
			
			By construction, we have that $\Cat_{Q_{i,j}}(\mathscr{J}^{i,j}) = \mathscr{C}^{i,j}$ and $\Cat_{Q_{i,j}}(\mathscr{T}^{i,j}) = \mathscr{D}^{i,j}$.
			\begin{prop} \label{algoresult} The \cref{algo:catCJR} returns
				$\mathscr{C}^{n+1,0} = \mathscr{C}_{\overrightarrow{A_n}}(\mathbf B,\mathbf E)$.
			\end{prop}
			\begin{proof}[Proof] First, we can check, by a simple induction, that the quiver $Q_{p,0}$, for $p \geqslant 2$, is as follows.
				\begin{center}
					\begin{tikzpicture}[-,line width=0.5mm,>= angle 60, ->,color=black,scale=0.7]
						\node (1) at (0,0){$1$};
						\node (2) at (2,0){$2$};
						\node (3) at (4,0){$\cdots$};
						\node (4) at (6,0){$p-1$};
						\node (5) at (8,0){$p$};
						\node (6) at (10,0){$\cdots$};
						\node (7) at (12,0){$n$};
						\draw (1) -- (2) ;
						\draw (2) -- (3) ;
						\draw (3) -- (4);
						\draw (7) -- (6) ;
						\draw (6) -- (5) ; 
						\draw (5)-- (4);
					\end{tikzpicture}
				\end{center}
				Thus $Q_{n+1,0} = \overrightarrow{A_n}$ as claimed.
				As there is a bijection from subcategories of any given $A_n$ type quiver $Q$ to interval subsets, we can work with the sequence $\mathscr{J}^{i,j}$ and prove that $\mathscr{J}^{n+1,0} = \mathscr{J}(\B, \E)$. It will imply the result we wished for. 
				
				To do so, we will first study the evolution of $\llbracket n-j \rrbracket$ appearing in $\mathscr{J}^{i,j+1}$ following the algorithm until arriving at $\mathscr{J}^{n+1,0}$. Following the sequence of reflections from $\mathscr{J}^{i,j+1}$ to $\mathscr{J}^{i+1,0} = \mathscr{J}^{i,n-i+1}$ we get $$\refl_{i} \circ \refl_{i+1} \circ \ldots \circ \refl_{n-j-1} (\llbracket n-j\rrbracket) = \llbracket i,n-j \rrbracket .$$ In the sequence of reflections applied to get $\mathscr{J}^{i+2,0}$ from $\mathscr{J}^{i+1,0}$, the only reflection that affects $\llbracket i, n-j \rrbracket$ is $\refl_{n-j+1}$. So $\llbracket n-j \rrbracket$ in $\mathscr{J}^{i,j+1}$ corresponds to $\llbracket i,n-j+1 \rrbracket$ in $\mathscr{J}^{i+2,0}$. 
				
				By the same argument, we get that $\llbracket n-j \rrbracket$ in $\mathscr{J}^{i,j+1}$ corresponds to $\llbracket i,n \rrbracket$ in $\mathscr{J}^{i+j+1,0}$. Here, following the sequence of reflections applied to get $\mathscr{J}^{i+j+2,0} = \mathscr{J}^{i+j+1,n-i-j+2}$ from $\mathscr{I}^{i+j+1,0}$, the interval $\llbracket i,n \rrbracket$ becomes $\llbracket i,i+j \rrbracket$ in $\mathscr{J}^{i+j+2,0}$. As the remainder of the sequence of reflections we still have to apply to get $\mathscr{J}^{n+1,0}$ from $\mathscr{J}^{i+j+2,0}$ does not affect anymore $\llbracket i,i+j \rrbracket$, because they only touch the vertices $q \geqslant i+j+2$, we conclude that $\llbracket n-j \rrbracket$ in $\mathscr{J}^{i,j+1}$ corresponds to $\llbracket i,i+j \rrbracket \in \mathscr{J}^{n+1,0}$.
				
				To end the proof, we only have to notice that during the construction, we add in our interval subset $\llbracket n-j \rrbracket \in \mathscr{J}^{i,j+1}$ if and only if we have $i \in \B$ and $i+j \in \E$.
			\end{proof}
			\begin{theorem}\label{thm:CJRlinear} Let $n \geqslant 1$.
				For all effective shifted interval bipartitions $(\B, \E)$ of $\{1, \ldots, n\}$, $\mathscr{C}_{\overrightarrow{A_n}}(\B, \E)$ is canonically Jordan recoverable. Moreover, for all $X \in \mathscr{C}_{\overrightarrow{A_n}}(\B, \E)$:
				\begin{enumerate}[label = $(\alph*)$]
					\item $\GenJF(X)$ is $(\boxplus, \boxplus)$-storable at $q$ for $q \notin \B \cup \E$;
					
					\item $\GenJF(X)$ is $(\boxplus, \boxminus)$-storable at $q$ for  $q \in \E \setminus \B$;
					
					\item $\GenJF(X)$ is $(\boxminus, \boxplus)$-storable at $q$ for $q \in \B \setminus \E$;
					
					\item $\GenJF(X)$ is $(\boxminus, \boxminus)$-storable at $q$ for $q \in \B \cap \E$;
					
					\item $\GenJF(X)^q = (0)$ whenever $q < \min(\B)$ or $q > \max(\E)$
					
				\end{enumerate}
			\end{theorem}
			\begin{proof}[Proof] Assume first that $(\B,\E)$ is a complete shifted bipartition of $\{1,\ldots,n\}$. In this case, we show the claimed result by an induction proof of the following claim, for all $i \in \{1,...,n+1\}$ and $j \in \{0,\ldots,n-i+1\}$:
				\begin{enumerate}[label = $(\mathbf{H}_{i,j})$]
					\item \label{Hyp:storabilityCat} The category $\mathscr{C}^{i,j}$ is canonically Jordan recoverable and for all $ X \in \mathscr{C}^{i,j}$:
					\begin{enumerate}[label = $(\alph*)$] 
						\item $\GenJF(X)$ is $(\boxplus, \boxplus)$-storable at $q$ for $q \in \{1, \ldots, n\}$ such that either:
						\begin{enumerate}[label*=$(\arabic*)$]
							\item $ q < i-1$ and $q \notin \B \cup \E$, or;
							
							\item $q = n-j$ if $i \in \B$ and $n-j \neq i-1$.
						\end{enumerate}
						
						\item $\GenJF(X)$ is $(\boxplus, \boxminus)$-storable at $q$ for $q \in \{1, \ldots, n\}$ such that either:
						\begin{enumerate}[label*=$(\arabic*)$]
							\item $q < i-1$ and $q \in \E \setminus \B$, or;
							
							\item $q = i-1$, if $i-1 \notin \B$, or;
							
							\item $q > i-1$ and $q \notin \{n-j,n-j+1\}$, or;
							
							\item $q \in \{n-j, n-j+1\}$ if $i \notin \B$.
						\end{enumerate}

						\item $\GenJF(X)$ is $(\boxminus, \boxplus)$-storable at $q$ for $q \in \{1, \ldots, n\}$ such that $q < i-1$ and $q \in \B \setminus \E$;
						
						\item $\GenJF(X)$ is $(\boxminus, \boxminus)$-storable at $q$ for $q \in \{1, \ldots i-1\}$ such that either:
						\begin{enumerate}[label*=$(\arabic*)$]
							\item $q < i-1$ and $q \in \B \cap \E$, or;
							
							\item $q = i-1$, if $i-1 \in \B$, or;
							
							\item $q =n-j+1$ if $i \in \B$.
						\end{enumerate}
						
						\item $\GenJF(X)^q = \GenJF(X)^{q+1}$ for $q \in \{0, \ldots, n\}$ such that either:
						\begin{enumerate}[label*=$(\arabic*)$]
							\item $n-q+i \leqslant n-j$ and $n-q+i \in \B$, or;
							
							\item $n-q+i+1 > n-j$ and $n-q+i+1 \in \B$.
						\end{enumerate}
					\end{enumerate}
				\end{enumerate}
				Recall that we extended the generic Jordan form data of any $X \in \rep(Q_{i,j})$ by writing $\GenJF(X)^0 = \GenJF(X)^{n+1} = (0)$.
				
				Note, by \cref{algoresult}, that the claim $(\mathbf{H}_{n+1,0})$ corresponds to the wished-for result. 
				
				For $i=1$ and $j=0$, $\mathscr{C}^{1,0} = \{0\}$ is canonically Jordan recoverable, and $\GenJF(0) = ((0))_{1 \leqslant q \leqslant n}$ satisfies all the storability conditions and equalities we ask for. 
				
				Now assume that for a fixed $i \in \{1, \ldots, n+1 \}$, and $j \in \{0, \ldots, n-i+1\}$, $\mathscr{C}^{i,j}$ satisfies \ref{Hyp:storabilityCat}. 
				
				We will show that either $\mathscr{C}^{i,j+1}$ satisfies $(\mathbf{H}_{i,j+1})$ if $j < n-i+1$, or $\mathscr{C}^{i+1,0}$ satisfies $(\mathbf{H}_{i+1,0})$ otherwise. We can already say that the only vertices $q$ where the storability conditions change are $q = n-j-1$, $q =n-j$, and $q = n-j+1$. 
				
				We have several cases to treat:
				\begin{enumerate}[label = $\bullet$]
					\item The case $j=n-i+1$ is trivial by the fact that $\mathscr{C}^{i,n-i+1} = \mathscr{C}^{i+1,0}$ by \cref{algo:catCJR}. We only need to check that $(i,n-i+1)$ and $(i+1,0)$ yield the same storability conditions.
					
					\item Assume that $i \notin \B$ and $j < n-i+1$. Following \cref{algo:catCJR}, to get $\mathscr{C}^{i,j+1}$ from $\mathscr{C}^{i,j}$, we have to apply the reflection functor $\mathcal{R}_{n-j}^-$. By \ref{Hyp:storabilityCat}, we know that $\mathscr{C}^{i,j}$ is canonically Jordan recoverable, and, for all $X \in \mathscr{C}^{i,j}$, $\GenJF(X)$ is $(\boxplus, \boxminus)$-storable at $n-j$. As a consequence of \cref{mutationCJR}, $\mathscr{C}^{i,j+1} = \mathcal{R}_{n-j}^-(\mathscr{C}^{i,j})$ is canonically Jordan recoverable. Moreover, by \cref{thm:reflectionGenJF}, and \cref{lem:elementary diag store}  \ref{ediag2}, for all $X \in \mathscr{C}^{i,j}$, $\GenJF(\mathcal{R}_{n-j}^-(X))$ is $(\boxplus, \boxminus)$-storable at $n-j$. Thus, the storability conditions satisfied by representations in $\mathscr{C}^{i,j}$ and those satisfied by representations in $\mathscr{C}^{i,j+1}$ are the same. If $\GenJF(X)^{n-j} = \GenJF(X)^{n-j-1}$, then 
					$\GenJF(\mathcal{R}_{n-j}^-(X))^{n-j} = \GenJF(\mathcal{R}_{n-j}^-(X))^{n-j+1}$ by \cref{lem:elementary diag store} \ref{ediag0}.
					
					\item Assume that $i \in \B$, $j < n-i+1$ and $i+j \notin \E$. To go from $\mathscr{C}^{i,j}$ to $\mathscr{C}^{i,j+1}$, we only have to apply $\mathcal{R}_{n-j}^-$. By induction, we have that $\GenJF(X)$ is $(\boxplus, \boxplus)$-storable at $n-j$ for all $X \in \mathscr{C}^{i,j}$. Then, by \cref{mutationCJR}, we get that $\mathscr{C}^{i,j+1}$ is canonically Jordan recoverable and, by \cref{thm:reflectionGenJF}, $\GenJF(\mathcal{R}_{n-j}^-(X))$ is  strongly $(\boxminus, \boxminus)$-storable at $n-j$. We can also check easily that $\GenJF(\mathcal{R}_{n-j}^-(X))$ is $(\boxplus, \boxplus)$-storable at $n-j-1$ and $(\boxplus, \boxminus)$-storable at $n-j+1$ if $n-j-1 > i-1$ by the diagonal transformation at $n-j$. We can also remark that if $\GenJF(X)^{n-j} = \GenJF(X)^{n-j-1}$, then $\GenJF(\mathcal{R}_{n-j}^-(X))^{n-j} = \GenJF(\mathcal{R}_{n-j}^-(X))^{n-j+1}$ by \cref{lem:elementary diag store} \ref{ediag0}.
					
					\item Assume that $i \in \B$, $j < n-i+1$ and $i+j \in \E $. To go from $\mathscr{C}^{i,j}$ to $\mathscr{C}^{i,j+1}$, we need to use $\mathcal{R}_{n-j}^-$ followed by $\Adds_{n-j}$. By induction, we have that $\GenJF(X)$ is $(\boxplus, \boxplus)$-storable at $n-j$ for all $X \in \mathscr{C}^{i,j}$. By \cref{mutationCJR}, we get that $\mathscr{D}^{i,j+1} = \mathcal{R}_{n-j}^-(\mathscr{C}^{i,j})$  is canonically Jordan recoverable and, by \cref{thm:reflectionGenJF}, $\GenJF(\mathcal{R}_{n-j}^-(X))$ strongly $(\boxminus, \boxminus)$-storable at $n-j$. By \cref{addsimple2}, $\mathscr{C}^{i,j+1} = \Adds_{n-j}(\mathscr{D}^{i,j+1})$ is also canonically Jordan recoverable, and by \cref{GenJFaddsimple}, $\GenJF(Z)$ is $(\boxminus, \boxminus)$-storable at $n-j$ for all $Z \in \mathscr{C}^{i,j+1}$. We also have that $\GenJF(Z)$ is $(\boxplus, \boxplus)$-storable at $n-j-1$, if $n-j-1 > i-1$ and $(\boxplus,\boxminus)$-storable at $n-j+1$. 
				\end{enumerate}
				This completes the induction proof in this case.
				
				Therefore \ref{Hyp:storabilityCat} is true for all $i \in \{1, \ldots , n+1\}$ and $j \in \{0, \ldots, n-i+1\}$, and we get the wished-for result for all complete shifted bipartition $(\B, \E)$ of $\{1, \ldots, n\}$.
				
				Now assume that $(\B,\E)$ is an arbitrary effective shifted interval bipartition of $\{1, \ldots, n\}$. We can see the category $\mathscr{C}_{\overrightarrow{A_n}}(\B,\E)$ has a subcategory in $\mathscr{C}_{\overrightarrow{A_m}}(\B^\circ, \E^\circ)$ where $m = \max(\E) - \min(\B) + 1$, $\B^\circ = \{i-\min(\B)+1 \mid i \in \B\}$ and $\E^\circ = \{j-\min(\B) + 1 \mid j \in \E\}$. As $(\B^\circ, \E^\circ)$ is a complete shifted bipartition of $\{1,\ldots,m\}$, we get that $\mathscr{C}_{\overrightarrow{A_n}}(\B,\E)$ is canonically Jordan recoverable and the storable condition of $\GenJF(X)$ holds for all $X \in \mathscr{C}_{\overrightarrow{A_n}}(\B, \E)$. Note that, for such a representation $X$, it is obvious that $\GenJF(X)^q = (0)$ for $q \notin \{\min(\B), \ldots, \max(\E)\}$. 
				
				This completes the proof.
			\end{proof}

			\subsection{For other orientations}
			\label{ss:Otherorientations} In this subsection, we show that \cref{maintheorem} holds for all quivers of $A_n$ type. To do so, we will first prove a result similar to \cref{thm:CJRlinear} available for any $A_n$ type quiver and then use it to give the final proof.
			\begin{theorem}\label{thm:CJR} Let $n > 0$ and $Q$ be an $A_n$ type quiver. 
				For all effective shifted interval bipartitions  $(\B, \E)$ of $\{1, \ldots, n\}$, $\mathscr{C}_{Q}(\B, \E)$ is canonically Jordan recoverable. Moreover, for all $X \in \mathscr{C}_{Q}(\B, \E)$:
				\begin{enumerate}[label = $(\alph*)$]
					\item $\GenJF(X)$ is $(\boxplus, \boxplus)$-storable at $q$ for $q \notin \B \cup \E$;
					
					\item $\GenJF(X)$ is $(\boxplus, \boxminus)$-storable at $q$ for  $q \in \E \setminus \B$;
					
					\item $\GenJF(X)$ is $(\boxminus, \boxplus)$-storable at $q$ for $q \in \B \setminus \E$;
					
					\item $\GenJF(X)$ is $(\boxminus, \boxminus)$-storable at $q$ for $q \in \B \cap \E$;
					
					\item $\GenJF(X)^q = (0)$ whenever $q < \min(\B)$ or $q > \max(\E)$.
				\end{enumerate}
			\end{theorem}
			\begin{proof}[Proof]
				Note that we can go from $\overrightarrow{A_n}$ to any $A_n$ type quiver by a sequence of mutations only done at sources. Using that fact, we will prove our desired result by induction.
				
				First, we know that for $Q = \overrightarrow{A_n}$, the claim is valid by \cref{thm:CJRlinear}. 
				
				Assume now that for a fixed $A_n$ type quiver $Q$, the same is true. Let $v$ be a source of $Q$ and put $\Xi = \sigma_v(Q)$. We will prove that the same goes for $\Xi$. 
				
				Let $(\B,\E)$ be an effective shifted interval bipartition of $\{1, \ldots, n\}$. We want to prove that $\mathscr{C}_\Xi(\B, \E)$ is canonically Jordan recoverable and for all $X \in \mathscr{C}_\Xi(\B,\E)$, $\GenJF(X)$ satisfies the announced storability conditions. By \cref{cor:antrefl}, we know that there exists $(\B', \E') = \tog_v(\B,\E)$ an effective shifted interval bipartition of $\{1, \ldots, n\}$ such that either $\mathscr{J}(\B,\E) = \refl_v(\mathscr{J}(\B',\E'))$ if $v \notin \B \cap \E$, or $\mathscr{J}(\B,\E) = \refl_v(\mathscr{J}(\B',\E')) \cup \{\llbracket v \rrbracket\}$ otherwise. 
				\begin{enumerate}[label = $\bullet$]
					
					\item If $v \notin \B \cap \E$, then $\mathscr{C}_\Xi(\B,\E) = \mathcal{R}_v^-(\mathscr{C}_Q(\B',\E'))$. By induction, we know that $\mathscr{C}_Q(\B',\E')$ is canonically Jordan recoverable, and for all $X \in \mathscr{C}_Q(\B',\E')$, $\GenJF(X)$ is  either $(\boxplus, \boxminus)$-storable, $(\boxminus, \boxplus)$-storable, or $(\boxminus, \boxminus)$-storable at $v$. 
					
					In the first two cases, we conclude by \cref{addsimple2} and \cref{lem:storinter}.
					
					In the last case, this means $v \in \B' \cap \E'$. By the induction hypothesis, for all $Y \in \mathscr{C}_Q(\B',\E') $, $\GenJF(Y)$ is $(\boxminus, \boxminus)$-storable at $v$. By considering $\mathscr{D}^v_{Q}(\B',\E')$ the subcategory of $\mathscr{C}_Q(\B',\E')$ generated by modules without $S_v$ in its summands, using \cref{GenJFaddsimple}, we obtain that, for all $Z \in \mathscr{D}^v_{Q}(\B',\E')$, $\GenJF(Z)$ is strongly $(\boxminus, \boxminus)$-storable at $v$, By \cref{mutationCJR}, $\mathscr{C}_\Xi(\B,\E) = \mathcal{R}_v^-(\mathscr{D}^v_Q(\B',\E'))$ is canonically Jordan recoverable and, by \cref{lem:elementary diag store} \ref{ediag4}, for all $X \in \mathscr{C}_\Xi(\B,\E)$, $\GenJF(X)$ is $(\boxplus, \boxplus)$-storable at $v$. 
					
					The only other storability conditions that change from 
					$\mathscr{C}_Q(\B',\E')$ to $\mathscr{C}_\Xi(\B, \E)$ are at $v-1$ and at $v+1$. If $v-1 \in \E' \setminus \B'$, then the $(\boxminus, \boxplus)$-storability condition satisfied by $\GenJF(Z)$, for $Z \in \mathscr{C}_Q(\B',\E')$, becomes a $(\boxminus, \boxminus)$-storability condition satisfied by $\GenJF(X)$ for $X \in \mathscr{C}_\Xi(\B,\E)$. This corresponds with $v-1 \in \B \cap \E$. We can treat similarly the case where we go from $(\boxplus, \boxplus)$-storability condition to $(\boxplus, \boxminus)$-storability condition at $v-1$, if $v-1 \in \B' \cap \E'$. A similar, symmetric argument yields the same result at $v+1$, considering the two possible changes in storability conditions.
					
					This completes the proof of the induction step in this case.
					
					\item If $v \in \B \cap \E$, $\mathscr{C}_\Xi(\B,\E) = \Adds_v( \mathcal{R}_v^-(\mathscr{C}_Q(\B',\E')))$. Using the definition of $(\B',\E')$, we have that $v \notin \B' \cup \E'$ and $v-1 \in \E'$ or $v+1 \in \B'$. Therefore, the induction hypothesis allows us to state that for all $X \in \mathscr{C}_Q(\B',\E')$, $\GenJF(X)$ is $(\boxplus, \boxplus)$-storable at $k$. 
					
					Following \cref{thm:reflectionGenRep} and \cref{lem:elementary diag store}, we get that $\mathcal{R}_v^-(\mathscr{C}_Q(\B',\E'))$ is canonically Jordan recoverable and for any representation $Y$ in this category, $\GenJF(Y)$ is strongly $(\boxminus, \boxminus)$-storable at $v$. Hence by \cref{addsimple2}, we get that $\mathscr{C}_\Xi(\B, \E)$ is canonically Jordan recoverable and for any representation $Z$ in it we have that $\GenJF(Z)$ is $(\boxminus, \boxminus)$-storable at $v$.
					
					By analogous arguments to those given in the previous point, we can deduce that the storability conditions satisfied by $\GenJF(X)$ for all $X \in \mathscr{C}_\Xi(\B, \E)$ are the ones we claimed.
					
					This completes the proof of the induction step in this case
				\end{enumerate}
				Thus, we have proved the induction step, and so we have proved the wished-for result.
			\end{proof}
			We can now prove the main result of this paper.
			\begin{proof}[Proof of \cref{maintheorem}] 
				Fix $Q$ a quiver of $A_n$ type. Let $\mathscr{C}$ be a subcategory of $\rep(Q)$. By \cref{nec cond CJR1}, we already know that if $\mathscr{C}$ is canonically Jordan recoverable, then  $\Int(\mathscr{C})$ is adjacency-avoiding.
				
				Now assume that $\Int(\mathscr{C})$ is adjacency-avoiding. Using \cref{themaxavoid}, there exists a complete shifted bipartition $(\B,\E)$ of $\{1, \ldots,n\}$ such that $\Int(\mathscr{C}) \subset \mathscr{J}(\B,\E)$, and $\mathscr{C}$ is therefore a subcategory of $\mathscr{C}_{Q}(\B,\E)$. By \cref{thm:CJR}, we know that $\mathscr{C}_{Q}(\B,\E)$ is canonically Jordan recoverable. Thus, so is $\mathscr{C}$.
			\end{proof}
			
			\section{To go further}
			\label{ss:tgf}
			We could ask ourselves some questions based on this work.
			
			\begin{itemize}
				\item[$\bullet$] \textit{Can we translate the adjacency-avoiding property for intervals into another algebraic property for subcategories of $\rep(Q)$ for any $A_n$ type quiver $Q$?}
			\end{itemize}
			
			First, the following simple lemma allows us to translate the adjacency property into the algebraic world.
			\begin{lemma} \label{lem:adjnonsplit}
				Let $Q$ be an $A_n$ type quiver. Let $K, L \in \mathcal{I}_n$. Then $K$ and $L$ are adjacent if and only if there is a short exact sequence with end terms $X_K$ and $X_L$ in some order and an indecomposable middle term.
			\end{lemma}
			The following theorem follows as a direct consequence of \cref{maintheorem} and \cref{lem:adjnonsplit}.
			\begin{theorem} \label{thm:repcharacterisation}
				Let $Q$ be a quiver of $A_n$ type. Fix a collection of indecomposable representations $\mathscr{X} \subseteq \Ind(Q)$. Then $\add(\mathscr{X})$ is canonically Jordan recoverable if and only if for all $X,Y \in \mathscr{X}$ and for all short exact sequences
				$$0 \longrightarrow X \longrightarrow E \longrightarrow Y \longrightarrow 0,$$ the representation $E$ is not indecomposable.
			\end{theorem}
			Thus, for any $T \in \rep(Q)$ a tilting representation, $\add(T)$ is canonically Jordan recoverable. But the interaction between canonically Jordan recoverable subcategories and tilting representations does not end here. We formulate an exact statement as the next conjecture.
			
			Let $Q$ be an $A_n$ type quiver, and $T$ a tilting representation of $Q$. Write $T = T_1 \oplus \ldots, \oplus T_n$ for the decomposition of $T$. For all $i \in \{1,\ldots, n\}$, we define 
			the \new{mutation of $T$ by $T_i$}, denoted $\mu_{T_i}(T)$, to be, if it is possible, the unique tilting representation (up to isomorphism) isomorphic to $ T_1 \oplus \ldots \oplus T_i' \oplus \ldots T_n$ such that $T_i' \ncong T_i$, otherwise $\mu_{T_i}(T) = T$. Riedtmann and Schofield proved that $T_i'$ can be obtained from $U_i = T_1 \oplus \ldots \oplus T_{i-1} \oplus T_{i+1} \oplus T_n$ and $T_i$ as either the kernel or the cokernel of a minimal $\add(U_i )$-approximation (see \cite{RS91}). 
			
			Following this result, in our case, we can divide the non-trivial mutations into two kinds:
			\begin{enumerate}[label=$\bullet$]
				\item the \new{$1$-term mutations} in the two following cases:
				\begin{enumerate}[label = $\bullet$]
					\item there exists $i \neq k \in \{1, \ldots, n\}$ such that $T_i \longrightarrow T_k$ is a minimal left $\add(U_i)$-approximation, and $T_i' = \Coker(T_i \longrightarrow T_k)$.
					\item there exists $i \neq k \in \{1, \ldots, n\}$ such that $T_k \longrightarrow T_i$ is a minimal right $\add(U_i)$-approximation, and $T_i' = \Ker(T_k \longrightarrow T_i)$.
				\end{enumerate} 
				\item the \new{$2$-terms mutations} in the two following cases:
				\begin{enumerate}[label = $\bullet$]
					\item there exist $\ell,k \in \{1, \ldots, n\}$ such that $\ell \neq k$, $\ell \neq i \neq k$ and $T_i \longrightarrow T_k \oplus T_\ell$ is a minimal left $\add(U_i)$-approximation, and $T_i' = \Coker(T_i \longrightarrow T_k \oplus T_\ell)$.
					\item there exist $\ell, k \in \{1, \ldots, n\}$ such that $\ell \neq k$, $\ell \neq i \neq k$ and $T_k \oplus T_\ell \longrightarrow T_i$ is a minimal right $\add(U_i)$-approximation, and $T_i' = \Ker(T_k \oplus T_\ell \longrightarrow T_i)$.
				\end{enumerate} 
			\end{enumerate}
			Note that $T_i$ can admit both a left and a right $\add(U_i)$-approximation, but only one of them defines the summand $T_i'$ used to construct $\mu_{T_i}(T)$.
			
			Now we can state our conjecture.
			\begin{conj}\label{conj:tilt}
				Let $Q$ be an $A_n$ type quiver. 
				\begin{enumerate}[label=$(\alph*)$]
					\item For any tilting representation $T \in \rep(Q)$, there exists a unique maximal canonically Jordan recoverable subcategory $\mathscr{C}$ such that $T \in \mathscr{C}$;
					
					\item For any maximal canonically Jordan recoverable subcategory $\mathscr{C}$ of $\rep(Q)$, and for any tilting representation $T \in \mathscr{C}$, $\mathscr{C}$ is additively generated by indecomposable summands of tilting representations that can be obtained by a (finite) sequence of $2$-term mutations from $T$.
				\end{enumerate}
			\end{conj}
			This result could open the way to algebraically characterizing the canonically Jordan recoverable subcategories for at least Dynkin quivers. 
			\begin{remark} \label{rem:ralf}
				In \cref{thm:repcharacterisation}, we can see a kind of complementarity with the notion of \textit{maximal almost rigid} modules \cite{BGMS23}: they are defined as modules $M = \bigoplus_{i=1}^s M(i)$ where $M(i) \in \Ind(Q)$ give a maximal collection of indecomposable representations such that for all $1 \leqslant i,j \leqslant s$ and all nonsplit short exact sequence $$0 \longrightarrow M(i) \longrightarrow E \longrightarrow M(j) \longrightarrow 0,$$ the representation $E$ is indecomposable.
			\end{remark}
			
			This conjecture is attacked using exact structure interpretations in \cite{DR25}, offering, by the way, a deeper interaction between combinatorial and categorical tools.
			
			\begin{itemize}
				\item[$\bullet$] \textit{Can we hope to characterize the Jordan recoverable categories of $\rep(Q)$ for $Q$ of $A_n$ type?}
			\end{itemize}
			
			Let $\mathscr{J} \subset \mathcal{I}_n$ and $L \in \mathcal{I}_n$. A \new{$\mathscr{J}$-partition of $L$} is a partition $\{T_1, \ldots, T_p\}$ of $L$ such that for all $i \in \{1, \ldots, p\}$, we have $T_i \in \mathscr{J}$.
			\begin{prop}
				Let $\mathscr{J} \subseteq \mathcal{I}_n$ such that there exists an interval $L \in \mathcal{I}_n$ admitting two distinct $\mathscr{J}$-partitions. Then $\Cat_Q(\mathscr{J})$ is not a Jordan recoverable category of $\rep(Q)$.
			\end{prop}
			\begin{proof}[Proof]
				This result is a consequence of \cref{lem:GenJFdisjunion}.
			\end{proof}
			As with the adjacency-avoiding interval subsets and the canonical Jordan recoverable categories, the following notion seems to play a significant role in determining all Jordan recoverable categories.
			\begin{definition}
				Let $\mathscr{J} \subset \mathcal{I}_n$. We say that $\mathscr{J}$ is \new{double interval partition-avoiding} if any $L \in \mathcal{I}_n$ admits at most one $\mathscr{J}$-partition.
			\end{definition}
			\begin{ex}
				For all $n \in \mathbb{N}^*$, $\mathscr{J} = \{\llrr{i} \mid i \in \{1, \ldots, n\}\}$ is double interval partition-avoiding. Note also that this is a maximal one (with respect to inclusion).
			\end{ex}
			\begin{conj} \label{conj:allJRcatA} Let $Q$ be an $A_n$ type quiver. A subcategory $\mathscr{C} \subset \rep(Q)$ is Jordan recoverable if and only if $\Int(\mathscr{C})$ is double interval partition-avoiding.
			\end{conj}
			We hope to prove this result in the near future.
			\begin{itemize}
				\item[$\bullet$] \textit{Can we hope to extend the definition of the Greene--Kleitman invariant on representations of string quivers?}
			\end{itemize}
			
			Let us define a string quiver.
			\begin{definition}
				A \new{string quiver} is a pair $(Q,R)$ where $Q$ is a finite connected quiver and $R$ is a set of monomial relations of degree $2$ such that :
				\begin{enumerate}[label = $\bullet$]
					\item all the vertices in $Q$ admit at most two ingoing arrows and at most two outgoing arrows;
					\item for any $\alpha,\beta, \gamma \in Q_1$ such that $\beta \alpha$ and $\gamma \alpha$ are paths of $Q$, then $$\{\beta\alpha, \gamma\alpha\} \cap R \neq \varnothing;$$
					
					\item for any $\alpha,\beta, \gamma \in Q_1$ such that $ \alpha \beta$ and $ \alpha \gamma$ are paths of $Q$, then $$\{\alpha \beta, \alpha \gamma\} \cap R \neq \varnothing.$$
				\end{enumerate}
			\end{definition}
			A representation of a string quiver $(Q,R)$ is a representation $X \in \rep(Q)$ such that, for any $\alpha, \beta \in Q_1$ with $\beta \alpha \in R$, we have $X_\beta X_\alpha = 0$. We denote by $\rep(Q,R)$ the finite-dimensional representations of $(Q,R)$. A \emph{string algebra} is a quotient algebra $\mathbb{K}Q/I$ where $(Q,R)$ is a string quiver and $I = \langle R \rangle$. Recall that $\rep(Q,R)$ is equivalent to the category of finitely generated (left) $\mathbb{K}Q/I$-modules.
			
			Note that, at least, if the Auslander--Reiten quiver of a string quiver $(Q,R)$ is acyclic, then we can define a similar Greene--Kleitman invariant. We can first ask how much we can extend this invariant in a larger case than the one we explore in this article.
			
			Garver, Patrias, and Thomas proved that we can define generic Jordan form data for any finite-dimensional module of any algebra \cite{GPT19}. We can therefore ask in which circumstances the two invariants coincide.
			\begin{itemize}
				\item[$\bullet$] \textit{May we expect to extend  \cref{maintheorem} for gentle, locally gentle, or string algebras?}
			\end{itemize}
			
			Recall that the idea of considering adjacency-avoiding interval subsets comes from previous work \cite{D22} for gentle algebras. It seems reasonable to extend this result to gentle and even string algebras.
			
			The reader is invited to explore these problems and related issues.
			
			\section*{Acknowledgments}
			
			I acknowledge the \textit{Institut des Sciences Mathématiques of Canada} and \textit{Engineering and Physical Sciences Research Council (EP/W007509/1)} for their partial funding support. I thank Ralf Schiffler for the week I spent in Storrs discussing this work, which led to some nice algebraic directions.
			
			I want to thank Claire Amiot, François Bergeron, and Frédéric Chapoton for the couple of pieces of advice and help they gave me in their reading of this article as part of my Ph.D. thesis. 
			
			Last, I thank my Ph.D. supervisor, Hugh Thomas, for all our discussions on this subject, his helpful advice, and his support throughout my thesis work.
			
			\bibliography{Canonically_Jordan_Recoverable_categories_Type_A_20260507}
			\bibliographystyle{alpha}
		\end{document}